\documentclass[12pt, letterpaper, reqno]{amsart}

\usepackage{amssymb,amsmath}
\usepackage{amsthm}
\usepackage{color,graphicx}
\usepackage{eucal}

\newcommand{\indentalign}{\hspace{0.3in}&\hspace{-0.3in}}
\newcommand{\la}{\langle}
\newcommand{\ra}{\rangle}

\newcommand{\defeq}{\stackrel{\rm{def}}{=}}
\newcommand{\supp}{\operatorname{supp}}

\newcommand{\sgn}{\operatorname{sgn}}

\newcommand{\R}{\mathbb R}

\newcommand{\lam}{{\lambda}}

\newcommand{\al}{\alpha}

\newcommand{\cR}{\mathbb{R}}

\newcommand{\ep}{\varepsilon}

\newcommand{\ds}{\displaystyle}

\newtheorem{theorem}{Theorem}[section]
\newtheorem{proposition}[theorem]{Proposition}
\newtheorem{remark}[theorem]{Remark}
\newtheorem{lemma}[theorem]{Lemma}
\newtheorem{corollary}[theorem]{Corollary}
\newtheorem{definition}[theorem]{Definition}

\begin{document}

\title[Instablity of solitons in 2d cubic ZK]
{Instability of solitons\\
in the 2d cubic
Zakharov-Kuznetsov equation}

\author[L. G. Farah]{Luiz Gustavo Farah}
\address{Department of Mathematics\\UFMG\\Brazil}
\curraddr{}
\email{lgfarah@gmail.com}
\thanks{}

\author[J. Holmer]{Justin Holmer}
\address{Department of Mathematics\\Brown University\\USA}
\curraddr{} 
\email{holmer@math.brown.edu}
\thanks{}

\author[S. Roudenko]{Svetlana Roudenko}
\address{Department of Mathematics\\The George Washington University\\USA}
\curraddr{}
\email{roudenko@gwu.edu}
\thanks{} 

\subjclass[2010]{Primary: 35Q53, 37K40, 37K45, 37K05}

\keywords{Zakharov-Kuznetsov equation, gKdV, instability of solitons, monotonicity, mass-critical}


\begin{abstract}
We consider the two dimensional generalization of the Korteweg-de Vries equation, the generalized Zakharov-Kuznetsov (ZK) equation, $u_t + \partial_{x_1}(\Delta u + u^p) = 0, (x_1,x_2) \in \R^2$. It is known that solitons are stable for nonlinearities $p<3$ and unstable for $p > 3$,  which was established by Anne de Bouard in \cite{DeB} generalizing the arguments of Bona-Souganidis-Strauss in \cite{BSS} for the gKdV equation. The $L^2$-critical case with $p=3$ has been open and in this paper we prove that solitons are unstable in the cubic ZK equation. This matches the situation with the critical gKdV equation, proved in 2001 by Martel and Merle in \cite{MM-KdV-instability}. While the general strategy follows \cite{MM-KdV-instability}, the two dimensional case creates several difficulties and to deal with them, we design a new virial-type quantity, revisit monotonicity properties and, most importantly, develop new pointwise decay estimates, which can be useful in other contexts.
\end{abstract}

\maketitle


\tableofcontents

\section{Introduction}

In this paper we consider the generalized Zakharov-Kuznetsov equation:
\begin{equation}\label{gZK}
\text{(gZK)} \qquad \qquad u_t + \partial_{x_1} \left(\Delta u + u^p \right) = 0, \qquad x = (x_1, ..., x_N) \in \cR^N, ~ t \in \cR, \qquad
\end{equation}
in two dimensions ($N=2$) and with a specific power of nonlinearity $p=3$. This equation is the higher-dimensional extension of the well-studied model describing, for example, the weakly nonlinear waves in shallow water, the Korteweg-de Vries (KdV) equation:
\begin{equation}
\label{gKdV}
\text{(KdV)} \qquad \qquad
u_t + u_{xxx} + (u^p)_x = 0, \quad p=2, \qquad x \in \cR, \qquad t \in \cR. \qquad
\end{equation}
When other integer powers $p \neq 2$ are considered, it is referred to as the {\it generalized} KdV (gKdV) equation, possibly with one exception of $p=3$, which is also referred to as the {\it modified} KdV (mKdV) equation.
Despite its apparent universality, the gKdV equation is limited as a spatially one-dimensional model. While there are several higher dimensional generalizations of it, in this paper we are interested in the gZK equation \eqref{gZK}.
In the three dimensional setting and quadratic power ($N=3$ and $p=2$), the equation \eqref{gZK} was originally derived by Zakharov and Kuznetsov to describe weakly magnetized ion-acoustic waves in a strongly magnetized plasma \cite{ZK}, thus, the name of the equation. In two dimensions, it is also physically relevant; for example, with $p=2$, it governs the behavior of weakly nonlinear ion-acoustic waves in a plasma comprising cold ions and hot isothermal electrons in the presence of a uniform magnetic field \cite{MP-1, MP-2}. Melkonian and Maslowe \cite{MM-longwaves} showed that the equation \eqref{gZK} is the amplitude equation for two-dimensional long waves on the free surface of a thin film flowing down a vertical plane with moderate values of the surface fluid tension and large viscosity. Lannes, Linares and Saut in \cite{LLS} derived the equation \eqref{gZK} from the Euler-Poisson system with magnetic field in the long wave limit. Yet another derivation was carried by Han-Kwan in \cite{HK} from the Vlasov-Poisson system in a combined cold ions and long wave limit.

In this paper we consider the Cauchy problem of the 2d cubic ZK equation (sometimes it is referred as the modified ZK, $m$ZK, or the generalized ZK, $g$ZK) with initial data $u_0$:
\begin{equation}\label{gzk}
\begin{cases}
{\displaystyle u_t+\partial_{x_1} \left( \Delta_{(x_1,x_2)} u +  u^{3} \right)  =  0,}  \quad (x_1,x_2) \in \mathbb{R}^2, ~ t>0, \\
{\displaystyle  u(0, x_1,x_2)=u_0(x_1,x_2) \in H^1 (\cR^2)}.
\end{cases}
\end{equation}
\smallskip

During their lifespan, the solutions $u(t, x_1,x_2)$ to \eqref{gzk} conserve the mass and energy:
\begin{equation}\label{MC}
M[u(t)]=\int_{\cR^2} \left[u(t,x_1,x_2) \right]^2\, dx_1dx_2 = M[u(0)]
\end{equation}
and
\begin{equation}\label{EC}
E[u(t)]=\dfrac{1}{2}\int_{\cR^2}|\nabla u(t,x_1,x_2)|^2\,dx_1dx_2 - \dfrac{1}{4}\int_{\cR^2} \left[u(t,x_1,x_2) \right]^{4}\,dx_1dx_2 = E[u(0)].
\end{equation}
There is one more conserved quantity of $L^1$-type, but we don't need it in this paper. We'll also mention that unlike the KdV and mKdV, which are completely integrable, the gZK equations do not exhibit complete integrability for any $p$.

One of the useful symmetries in the evolution equations is {\it the scaling invariance}, which states that an appropriately rescaled version of the original solution is also a solution of the equation. For the equation \eqref{gZK} it is
$$
u_\lambda(t, x_1,x_2)=\lambda^{\frac{2}{p-1}} u(\lambda^3t, \lambda x_1, \lambda x_2).
$$
This symmetry makes a specific Sobolev norm $\dot{H}^s$ invariant, i.e.,
\begin{equation*}
\|u(0,\cdot,\cdot) \|_{\dot{H}^s(\R^N)}=\lambda^{\frac{2}{p-1}+s-\frac{N}2} \|u_0\|_{\dot{H}^s(\R^N)},
\end{equation*}
and the index $s$ gives rise to the critical-type classification of equations.
For the gKdV equation \eqref{gKdV} the critical index is $s=\frac12-\frac2{p-1}$, and for the  two dimensional ZK equation \eqref{gZK} the index is $s = 1 - \frac2{p-1}$. When $s=0$ (this corresponds to $p=3$), the equation \eqref{gzk} is referred to as the $L^2$-critical equation. The gZK equation has other invariances such as translation and dilation. 

The generalized Zakharov-Kuznetsov equation has a family of travelling waves (or solitary waves, which sometimes are referred to as solitons), and observe that they travel only in $x_1$ direction
\begin{equation}\label{Eq:TW}
u(t,x_1,x_2) = Q_c(x_1-ct, x_2)
\end{equation}
with $Q_c(x_1,x_2) \to 0$ as $|x| \to + \infty$. Here, $Q_c$ is the dilation of the ground state $Q$:
$$
Q_c(\vec{x}) = c^{1/{p-1}} Q (c^{1/2} \vec{x}), \quad \vec{x} = (x_1,x_2),
$$
with $Q$ being a radial positive solution in $H^1(\cR^2)$ of the well-known nonlinear elliptic equation $-\Delta Q+Q  - Q^p = 0$.
Note that $Q \in C^{\infty}(\R^2)$, $\partial_r Q(r) <0$ for any $r = |x|>0$, and 
for any multi-index $\alpha$
\begin{equation}\label{prop-Q}
|\partial^\al Q(\vec{x})| \leq c(\al) e^{- |\vec{x}|} \quad \mbox{for any}\quad \vec{x} \in \cR^2.
\end{equation}

In this work, we are interested in stability properties of travelling waves in the critical gZK equation \eqref{gzk}, i.e., in the behavior of solutions close to the ground state $Q$ (perhaps, up to translations). We begin with the precise concept of stability and instability used in this paper. For $\alpha>0$, the neighborhood (or ``tube") of radius $\alpha$ around $Q$ (modulo translations) is defined by
\begin{equation}\label{tube}
U_{\alpha}=\left\{u\in H^1(\mathbb{R}^2): \inf_{\vec{y}\in \mathbb{R}^2}\|u(\cdot)-Q(\cdot+\vec{y})\|_{H^1}\leq \alpha \right\}.
\end{equation}
\begin{definition}[Stability of $Q$]\label{D:Stability}
We say that $Q$ is stable if for all $\alpha>0$, there exists $\delta>0$ such that if $u_0\in U_{\delta}$, then the corresponding solution $u(t)$ is defined for all $t\geq 0$ and $u(t)\in U_{\alpha}$ for all $t\geq 0$.
\end{definition}

\begin{definition}[Instability of $Q$]\label{D:Instability}
We say that $Q$ is unstable if $Q$ is not stable, in other words, there exists $\alpha>0$ such that for all $\delta>0$ the following holds:
if $u_0\in U_{\delta}$, then there exists $t_0=t_0(u_0)$ such that $u(t_0)\notin U_{\alpha}$.
\end{definition}

The main goal of this paper is to show that in the two dimensional case $p=3$, the traveling waves are unstable, thus, completing the stability picture in two dimensions. In her study of dispersive solitary waves in higher dimensions, Anne de Bouard \cite{DeB} showed (her result holds in dimensions 2 and 3) that the travelling waves of the form \eqref{Eq:TW} are stable for $p<p_c$ and unstable for $p>p_c$, where $p_c=3$ in 2d. She followed the ideas developed for the gKdV equation by Bona-Souganidis-Strauss \cite{BSS} for the instability, and Grillakis-Shatah-Strauss \cite{GSS} for the stability arguments. Here, we prove the instability of the traveling wave solution of the form \eqref{Eq:TW}, in a spirit of Martel-Merle \cite{MM-KdV-instability}.
We also note that the more delicate questions about different types of stability have also been studied, in particular, Cote, Didier, Munoz and Simpson in \cite{CDMS} obtained the asymptotic stability of solitary waves in 2d for\footnote{The nonlinearity in such gZK equation should be understood as $\partial_{x_1}(|u|^{p-1}u)$.} $2\leq p<p^* <3$ by methods of Martel-Merle for the gKdV equation. The upper bound $p^*$ in their restriction of nonlinearity comes from having a certain bilinear form positive-definite, which is needed for the linear Liouville property, see \cite[Theorem 1.3]{CDMS}, and numerically they show that the proper sign holds only for powers $p$ up to $p^* \approx 2.3$. It would be interesting to investigate if asymptotic stability holds for all $p<3$.

In this paper we prove the instability of the soliton $u(t,x_1,x_2)=Q(x_1-t, x_2)$. Our main result reads as follows.

\begin{theorem}[$H^1$-instability of $Q$ in the 2d critical ZK equation]\label{Theo-Inst}
There exists $\alpha_0>0$ such that for any $\delta > 0$ there exists $u_0 \in H^1(\R^2)$
satisfying
\begin{equation}\label{Eq:hypoth}
\| u_0 - Q \|_{H^1} \leq \delta \quad \mbox{and} \quad  u_0-Q \perp \{Q_{x_1}, Q_{x_2}, \chi_0 \},
\end{equation}
there exists a time $t_0=t_0(u_0) < \infty$ with $u(t_0)\notin U_{\alpha_0}$, or equivalently,
$$
\inf_{\vec{x} \in \R^2} \|u(t_0, \cdot)-Q(\cdot-\vec{x})\|_{H^1} \geq \alpha_0.
$$
\end{theorem}
Here, $\chi_0$ is the eigenfunction corresponding to the (unique) negative eigenvalue of the linearized operator $\mathcal{L}$, for details see Theorem \ref{L-prop}.
\begin{remark}
It suffices to prove Theorem \ref{Theo-Inst} for the following initial data: define $a= -\frac{\int \chi_0 \, Q}{\|\chi_0\|_{L^2}^2}$ and let $\delta > 0$. Fix $n_0 =  \, \left(1+\frac{\|\chi_0\|_{H^1}}{\|\chi_0\|_{L^2}}\right)\|Q\|_{H^1} \cdot \delta^{-1}$.
For any $n \geq n_0$ define
$$
\ep_0 =\frac{1}{n}\left(Q+a\chi_0\right).
$$
Then $u_0 = Q + \ep_0$ satisfies the hypotheses of the theorem, i.e., the conditions \eqref{Eq:hypoth}.
\end{remark}

The main strategy follows the approach introduced by Martel-Merle \cite{MM-KdV-instability} in their study of the same question in the critical gKdV model.
For that they worked out pointwise decay estimates on the shifted linear equation and applied them to the nonlinear equation (bootstraping twice in time).
In \cite{FHR1} we revisited that proof and showed that instead of pointwise decay estimates, it is possible to consider monotonicity properties of the solution, then apply them to the decomposition around the soliton and conclude the instability in the critical gKdV equation. In our two dimensional case of the critical ZK equation, we can not obtain the instability result just relying on monotonicity (and truncation when needed), because there is a new term appearing in the virial-type quantity which is truly two-dimensional, see Lemma \ref{J'_A} and the last term in \eqref{R}. We are forced to consider something else besides the monotonicity (since it would only give the boundedness of the new term, not the smallness), and thus, we develop new 2d pointwise decay estimates, see Sections \ref{S-8}-\ref{S-12}. These estimates by themselves are important results for the two-dimensional Airy-type kernel with the applications to the shifted linear equation as well as to the nonlinear equation. This part is a completely new development in the higher dimensional setting, and we believe that it will be useful in other contexts as well.

We note that we could prove the conditional instability with $\alpha$ being a multiple of $\delta$ in Definition \ref{D:Instability}, i.e.,
there exists a universal constant $c>0$ such that for any $\delta>0$ if $u_0\in U_{\delta}$, then there exists $t_0=t_0(u_0)$ such that $u(t_0)\notin U_{c \, \delta}$,
using only monotonicity. However, we emphasize that in order to show the instability with $\alpha$ independent of $\delta$, we need to use the pointwise decay estimates.

The paper is organized as follows. In Section \ref{S-2} we provide the background information on the well-posedness of the generalized ZK equation in two dimensions. In Section \ref{S-3} we discuss the properties of the linearized operator $L$ around the ground state $Q$ and exhibit the three sets of orthogonality conditions which make it positive-definite; the last one, which involves $\chi_0$, is the one we use in the sequel to control various parameters in modulation theory and smallness of $\ep$. Section \ref{S-4} contains the canonical decomposition of a solution $u$ around $Q$ (thus, introducing $\ep$ and the equation for it), then the modulation theory and control of parameters coming from such a decomposition are described in Section \ref{S-5}.
In Section \ref{S-6} we introduce the key player, the virial-type functional, and make the first attempt to estimate it. Truncation helps us to obtain an upper bound, however, to proceed with the time derivative estimates, we need to develop more machinery, which we do in subsequent sections.
In Section \ref{S-7} we discuss the concept of monotonicity, which allows us to control several terms in the virial functional, but as mentioned above, not all terms. The next few sections, starting from Section \ref{S-8} contain the new pointwise decay estimates. We state the main result in Section \ref{S-8} and also  re-examine $H^1$ well-posedness in the same section, then we develop the poitwise decay estimates on the 2d Airy-type kernel and its derivative in Section \ref{S-9}, after that we proceed with the application of them to the linear equation in Section \ref{S-10}, then for the Duhamel term in Section \ref{S-11}, and finally, for the nonlinear equation in Section \ref{S-12}. After all the tools are developed, we return to the virial-type functional estimates and obtain the lower bound on its time derivative, which allows us to conclude the instability result in Section \ref{S-13}.
\smallskip

{\bf Acknowledgements.} Most of this work was done when the first author was visiting GWU in 2016-17 under the support of the Brazilian National Council for Scientific and Technological Development (CNPq/Brazil), for which all authors are very grateful as it boosted the energy into the research project.
S.R. would like to thank IHES and the organizers for the excellent working conditions during the trimester program ``Nonlinear Waves" in May-July 2016.
L.G.F. was partially supported by CNPq and FAPEMIG/Brazil.
J.H. was partially supported by the NSF grant DMS-1500106.
S.R. was partially supported by the NSF CAREER grant DMS-1151618.

\section{Background on the generalized ZK equation} \label{S-2}
In this section we review the known results on the local and global well-posedness of the generalized ZK equation. To follow the notation
in the literature, in this section we denote the power of nonlinearity as $u^{k+1}$ (instead of $u^p$) and consider the Cauchy problem for the generalized ZK equation as follows:
\begin{equation}\label{gzk2}
\begin{cases}
u_t+\partial_{x_1} \Delta u+ \partial_{x_1}(u^{k+1})  =  0,  \quad (x_1,x_2) \in \mathbb{R}^2, ~ t>0, \\
u(0, x_1,x_2)=u_0(x_1,x_2) \in {H}^s(\R^2).
\end{cases}
\end{equation}
Faminskii \cite{Fa} showed the local well-posedness of the Cauchy problem \eqref{gzk2} for the $k=1$ case considering $H^1$ data (to be precise, he obtained the local well-posedness in $H^m$, for any integer $m \geq 1$.)
The current results on the local well-posedness are gathered in the following statement.
\begin{theorem} 
The local well-posedness in \eqref{gzk2} holds in the following cases:
\begin{itemize}
\item
$k=1$: for $s>\frac12$, see Gr\"unrock-Herr \cite{GH} and Molinet-Pilod \cite{MP15},
\item
$k=2$: for $s>\frac14$, see Ribaud-Vento \cite{RV},
\item
$k=3$: for $s>\frac5{12}$, see Ribaud-Vento \cite{RV},
\item
$k=4, 5, 6, 7$: for $s>1-\frac2{k}$, see Ribaud-Vento \cite{RV},
\item
$k=8$, $s>\frac34$, see Linares-Pastor \cite{LP1} 
\item
$k>8$, $s>s_k=1-2/k$, see Farah-Linares-Pastor \cite{FLP}.
\end{itemize}
\end{theorem}
Note that in the last three cases (i.e., for $k \geq 4$), the bound on $s > s_k$ is optimal from the scaling conjecture.
For previous results on the local well-posedness for $2 \leq k \leq 8$ for $s>3/4$ see \cite{LP} and \cite{LP1}.

Following the approach of Holmer-Roudenko for the $L^2$-supercritical nonlinear Schr\"odinger (NLS) equation, see \cite{HR} and \cite{DHR}, the first author together with F. Linares and A. Pastor obtained the global well-posedness result for the nonlinearities $k \geq 3$ and under a certain mass-energy threshold, see \cite{FLP}.
\begin{theorem}[\cite{FLP}]
Let $k\geq3$ and $s_k=1-2/k$.  Assume $u_0\in H^1(\R^2)$ and suppose that
\begin{equation}\label{GR1}
E(u_0)^{s_k} M(u_0)^{1-s_k} < E(Q)^{s_k} M(Q)^{1-s_k} , \,\,\,
E(u_0) \geq 0.
\end{equation}
If
\begin{equation}\label{GR2}
\|\nabla u_0\|_{L^2}^{s_k}\|u_0\|_{L^2}^{1-s_k} < \|\nabla
Q\|_{L^2}^{s_k}\|Q\|_{L^2}^{1-s_k},
\end{equation}
then for any $t$ from the maximal interval of existence 
\begin{equation*}
\|\nabla u(t)\|_{L^2}^{s_k}\|u_0\|_{L^2}^{1-s_k}=\|\nabla
u(t)\|_{L^2}^{s_k}\|u(t)\|_{L^2}^{1-s_k} <\|\nabla
Q\|_{L^2}^{s_k}\|Q\|_{L^2}^{1-s_k},
\end{equation*}
where $Q$ is the unique positive radial solution of
$$
\Delta Q-Q+Q^{k+1}=0.
$$
In particular, this implies that $H^1$ solutions, satisfying \eqref{GR1}-\eqref{GR2} exist globally in time.
\end{theorem}

\begin{remark}
In the limit case $k=2$ (or $p=3$, the  modified ZK equation), conditions
\eqref{GR1} and \eqref{GR2} reduce to one condition, which is
$$
\|u_0\|_{L^2}<\|Q\|_{L^2}.
$$
Such a condition was already used in \cite{LP} and \cite{LP1} to
show the existence of global solutions, respectively, in $H^1(\R^2)$
and $H^s(\R^2)$, $s>53/63$, see also \cite{FP} for another approach.
\end{remark}
We conclude this section with a note that while it would be important to obtain the local well-posedness down to the scaling index in the gZK equation for $1 \leq k \leq 3$, and the global well-posedness in the subcritical cases for $s < 53/63$ (ideally, all the way down to the $L^2$ level), for the purpose of this paper, it is sufficient to have the well-posedness theory in $H^1(\R^2)$.

\section{The Linearized Operator $L$}\label{S-3}

The operator $L$, which is obtained by linearizing around the ground state $Q$, is defined by
\begin{equation}
\label{L-def}
L  := - \Delta  + 1 - p \, Q^{p-1}.
\end{equation}
We first state the properties of this operator $L$ (see Kwong \cite{K89} for all dimensions, Weinstein \cite{W85} for dimension 1 and 3, also Maris \cite{M02} and \cite{CGNT}).
\begin{theorem}[Properties of $L$]\label{L-prop}
The following holds for an operator $L$ defined in \eqref{L-def}
\begin{itemize}
\item
$L$ is a self-adjoint operator and
$\sigma_{ess}(L) = [ \lam_{ess}, +\infty ) \quad \mbox{for~~ some~~} \lam_{ess} > 0$

\item
$\ker L = \mbox{span} \{Q_{x_1}, Q_{x_2} \}$
	
\item
$L$ has a unique single negative eigenvalue $-\lam_0$ (with $\lam_0 > 0$) associated to a
positive radially symmetric eigenfunction $\chi_0$.
Moreover, there exists $\delta > 0$ such that
\begin{equation}\label{chi_0}
|\chi_0(x)| \lesssim e^{- \delta |x|} \quad \mbox{for ~~all} ~~ x \in \cR^2.
\end{equation}
\end{itemize}
\end{theorem}
We also define the generator $\Lambda$ of the scaling symmetry as
\begin{equation}\label{Eq:Lambda}
\Lambda f = \frac2{p-1} f +  \vec{x} \cdot \nabla f, ~~~(x_1,x_2) \in \cR^2.
\end{equation}
The following identities are useful to have
\begin{lemma}\label{L-prop2}
The following identities hold
\begin{enumerate}
\item
$L (\Lambda Q) = - 2 Q$ 
\item
$\int Q \, \Lambda Q = 0$ if $p=3$ and  $\int Q \, \Lambda Q = \frac{3-p}{p-1} \int Q^2$ for $p \neq 3$.
\end{enumerate}
\end{lemma}
The proof is a direct simple computation and can be found in \cite{FHR2}.
\smallskip

In general, the operator $L$ is not positive-definite, however, on certain subspaces one can expect some positivity properties.
We now consider only the $L^2$-critical case and power $p=3$. First, we summarize known positivity estimates for the operator $L$ (see Chang et al. \cite{CGNT} and Weinstein \cite{W85}):
\begin{lemma}\label{L:perpcond}
The following conditions hold for $L$:
\begin{enumerate}
\item[(i)]
$(L\, Q, Q) = -2 \int Q^4 < 0$,

\item[(ii)]
$L_{|_{ \{ Q^3 \}^{\perp} }} \geq 0$,

\item[(iii)]
$ L_{|_{ \{ Q \}^{\perp} }} \geq 0$,

\item[(iv)]
$L_{|_{ \{Q, xQ,|x|^2 Q \}^{\perp} }} > 0$.
\end{enumerate}
\end{lemma}
The last property provides us with the orthogonality conditions that keep the quadratic form, generated by $L$, positive-definite (see Weinstein \cite[Prop. 2.9]{W85}):
\begin{lemma}\label{Lemma-ort1}
For any $f \in H^1(\R^2)$ such that
\begin{equation}\label{Ort-Cond}
(f, Q) = (f,x_j \,Q) = (f, |x|^2 Q)=0, \quad j=1,2,
\end{equation}
there exists a positive constant $C>0$ such that
$$
(Lf,f) \geq C \, (f,f).
$$
\end{lemma}
While it shows that eliminating directions from \eqref{Ort-Cond} would make the bilinear form $(Lf,f)$ positive, these directions are not quite suitable for our case. An  alternative for the orthogonality conditions \eqref{Ort-Cond}
would be to consider the kernel of $L$ from Theorem \ref{L-prop} and Lemma \ref{L:perpcond}(ii), which we do next.
\begin{lemma}\label{Lemma-ort2}
For any $f \in H^1(\R^2)$ such that
\begin{equation}\label{Ort-Cond2}
(f, Q^3) = (f, Q_{x_j}) =0, \quad j=1,2,
\end{equation}
there exists a positive constant $C>0$ such that
$$
(Lf,f) \geq C \, (f,f).
$$
\end{lemma}
\begin{proof}
From Chang et al. \cite{CGNT} (Lemma $2.2$ $(2.7)$) we have
\begin{equation}\label{lemma-CGNT}
\inf_{(f, Q^3)=0}(Lf,f)\geq 0.
\end{equation}
Let $C_1=\{(L\ep, \ep):\|\ep\|_{L^2}=1, \quad (f, Q^3) = (f, Q_{x_j}) =0, \quad j=1,2. \}$, then $C_1\geq 0$ by \eqref{lemma-CGNT}. Assume, by contradiction, that $C_1=0$. In this case, as in \cite[Proposition 2.9]{W85}, we can find a function $\ep^{\ast}\in H^1$ satisfying
\begin{itemize}
\item[$(i)$] $(L\ep^{\ast}, \ep^{\ast})=0$
\item[$(ii)$] $(L-\alpha)\ep^{\ast}=\beta Q^3 + \gamma Q_{x_1} + \delta  Q_{x_2}$
\item[$(iii)$] $\|\ep^{\ast}\|_{L^2}=1$ and $(\ep^{\ast}, Q^3) = (\ep^{\ast}, Q_{x_j}) =0, \quad j=1,2.$
\end{itemize}
Taking the scalar product of $(ii)$ with $\ep^{\ast}$, we deduce from $(iii)$ that $(L\ep^{\ast}, \ep^{\ast})=\alpha$, and thus, $\alpha=0$ by $(i)$. Now, taking the scalar product with $Q_{x_1}$, integrating by parts and recalling Theorem \ref{L-prop}, we have
$$
0=(\ep^{\ast}, LQ_{x_1})=(L\ep^{\ast}, Q_{x_1})=\gamma \int Q^2_{x_1}+\delta \int Q_{x_1}Q_{x_2}.
$$
Since $Q_{x_1}\perp Q_{x_2}$,
we deduce $\gamma = 0$. In a similar way (taking the scalar product with $Q_{x_2}$), we also have  $\delta =0$. Therefore, $ L\ep^{\ast}=\beta Q^3$, which implies
\begin{equation}\label{ep-ast}
\ep^{\ast}=-\frac{\beta}{2}Q+\theta_1Q_{x_1}+\theta_2Q_{x_2},
\end{equation}
where we have used Theorem \ref{L-prop} and Lemma \ref{L-prop2}.
Taking the scalar product of \eqref{ep-ast} with $Q^3$, from $(iii)$ and integration by parts, we get
$$
0=(\ep^{\ast}, Q^3)=-\frac{\beta}{2}\int Q^4,
$$
which implies $\beta = 0$.

Finally, using $Q_{x_j}$, $j=1,2$, we obtain $\theta_1=\theta_2=0$. Thus, $\ep^{\ast}=0$, which is a contradiction with $(iii)$.
\end{proof}

We deduce yet another set of orthogonality conditions, see \eqref{Ort-Cond3}, to keep the quadratic form, generated by $L$, positive-definite.
This is the set, which we will use in this paper.
\begin{lemma}\label{Lemma-ort3}
Let $\chi_0$ be the positive radially symmetric eigenfunction associated to the unique single negative eigenvalue $-\lam_0$ (with $\lam_0 > 0$). Then, there exists $\sigma_0>0$ such that for any $f \in H^1(\R^2)$ satisfying
\begin{equation}\label{Ort-Cond3}
(f, \chi_0) = (f, Q_{x_j}) =0, \quad j=1,2,
\end{equation}
one has
$$
(Lf,f) \geq \sigma_0 \, \, (f,f).
$$
\end{lemma}
\begin{proof}
The result follows directly from Schechter \cite[Chapter 8, Lemma 7.10]{Sch} (see also \cite[Chapter 1, Lemma 7.17]{Sch})
\end{proof}
In a sense, the last lemma shows that if we exclude the zero eigenvalue and negative eigenvalue directions,
then only the ``positive" directions are left, and thus, the positivity property of $L$ must hold.

\section{The linearized equation around $Q$} \label{S-4}

In this section we decompose our solutions $u(t,\vec{x})$ around the soliton $Q$. Since we consider the $L^2$-critical problem, we must also incorporate the scaling parameter (besides the translation as we did in the supercritical case in \cite{FHR2}). We use the following canonical decomposition of $u$ around $Q$:
\begin{equation}
\label{v-eq}
v(t,y_1,y_2) = \lam(t) \, u(t, \lam(t)y_1 +x_1(t), \lam(t)y_2 + x_2(t)).
\end{equation}
Our next task is to examine the difference $\ep = v - Q$, more precisely,
\begin{equation}
\label{ep-def}
\ep(t,\vec{y}) = v(t,\vec{y}) - Q(\vec{y}), \quad \vec{y} = (y_1,y_2).
\end{equation}

\subsection{Equation for $\ep$}
After we rescale time $t \mapsto s$ by $\frac{ds}{dt} = \frac1{\lam^3}$, we obtain the equation for $\ep$.
\begin{lemma}\label{eq-ep}
For all $s \geq 0$, we have
\begin{align}
\nonumber
\ep_s = (L \ep)_{y_1} & + \frac{\lam_s}{\lam} \Lambda Q + \left( \frac{(x_1)_s}{\lam} -1 \right) Q_{y_1} + \frac{(x_2)_s}{\lam} Q_{y_2} \\
\nonumber
&+ \frac{\lam_s}{\lam} \Lambda \ep + \left( \frac{(x_1)_s}{\lam} -1 \right) \ep_{y_1} + \frac{(x_2)_s}{\lam} \ep_{y_2}\\
& - 3 (Q \ep^2)_{y_1} - (\ep^3)_{y_1}
\label{ep1},
\end{align}
where $\Lambda f = f + \vec{y} \cdot \nabla f$ and  $L$ is the linearized operator around $Q$:
$$
L \ep = - \Delta \ep + \ep - 3 Q^2 \ep.
$$
\end{lemma}

\begin{proof}
Using \eqref{v-eq}, we obtain
$$
v_t = \lam_t u + \lam u_t + \lam u_{x_1} \left(\lam_t y_1 + (x_1)_t \right)+ \lam u_{x_2} \left(\lam_t y_2 + (x_2)_t \right),
$$
and for  $i = 1,2$
$$
v_{y_i} = \lam^2 u_{x_i}, \quad v_{y_i y_i} = \lam^3 u_{x_i x_i}. \hspace{4cm}
$$
Substituting the above into $u_t + \partial_{x_1} (\Delta u + u^3) = 0$, we obtain
$$
v_t = \lam^{-1} \lam_t v + \lam^{-1} \lam_t ( \vec{y} \cdot \nabla v) + \lam^{-1} \left( v_{y_1} (x_1)_t + v_{y_2} (x_2)_t \right)
-\lam^{-3} \partial_{y_1} \left(\Delta v + v^3\right).
$$
Recalling that $\frac{ds}{dt} = \frac1{\lam^3}$, we change the time variable $t \mapsto s$
$$
\lam^{-3} v_s = \lam^{-4}\lam_s v + \lam^{-4} \lam_s ( \vec{y} \cdot \nabla v) + \lam^{-4} \left( v_{y_1} (x_1)_s + v_{y_2} (x_2)_s \right)
-\lam^{-3} \partial_{y_1} \left(\Delta v + v^3\right),
$$
Simplifying, we get
$$
v_s = \frac{\lam_s}{\lam} \left( v +  \vec{y} \cdot \nabla v \right) + \frac{((x_1)_s, (x_2)_s )}{\lam} \cdot \nabla v
-\partial_{y_1} \left(\Delta v + v^3\right).
$$
Next, we use \eqref{ep-def} and the fact that $\Delta Q = Q-Q^3$ to obtain the equation for $\ep$:
$$
\ep_s = \frac{\lam_s}{\lam} \left( \Lambda Q + \Lambda \ep \right) + \frac{((x_1)_s, (x_2)_s )}{\lam} \cdot (\nabla Q + \nabla \ep) -\partial_{y_1} \left(Q + \Delta \ep  + 3Q^2 \ep + 3Q \ep^2 +\ep^3\right).
$$
Simplifying, we get the equation \eqref{ep1}.
\end{proof}

\subsection{Mass and Energy Relations}
Our next task is to derive the basic mass and energy conservations for $\ep$.
First, denote
\begin{equation}\label{M_0}
M_0 = 2 \int_{\cR^2} Q(\vec{y}) \ep(0,\vec{y}) \, d\vec{y} + \int_{\cR^2} \ep^2(0,\vec{y}) \, d\vec{y}.
\end{equation}
For any $s \geq 0$ by the $L^2$ scaling invariance and mass conservation, we have
$$
\int_{\cR^2} v^2(s, \vec{y}) \, d\vec{y} = \int_{\cR^2} \lam^2(t) \, u^2(t,\lam \vec{y} + \vec{x}(t)) \, d\vec{y} = \int_{\cR^2} u^2(t) \, d\vec{x} = M[u(t)] \equiv M[u(0)].
$$
On the other hand,
\begin{align*}
\int_{\cR^2} v^2(s, \vec{y}) \, d\vec{y} & = \int_{\cR^2} \left( Q(\vec{y}) + \ep(s, \vec{y}) \right)^2 \, d\vec{y} \\
&=\int_{\cR^2} Q^2(\vec{y}) \, d\vec{y} + 2 \int_{\cR^2} Q(\vec{y}) \, \ep(s,\vec{y}) \, d\vec{y} + \int_{\cR^2} \ep^2(s, \vec{y}) \, d\vec{y} \\
& = \int_{\cR^2} u_0^2(\vec{x}) \, d\vec{x}
=\int_{\cR^2} Q^2(\vec{y}) \, d\vec{y} + 2 \int_{\cR^2} Q(\vec{y}) \, \ep(0,\vec{y}) \, d\vec{y} + \int_{\cR^2} \ep^2(0,\vec{y}) \, d\vec{y},
\end{align*}
and thus,
\begin{equation}
\label{m-ep}
M[\ep(s)] := 2 \int_{\cR^2} Q(\vec{y}) \, \ep(s,\vec{y}) \, d\vec{y} + \int_{\cR^2} \ep^2(s,\vec{y}) \, d\vec{y} = M_0.
\end{equation}
Next, we examine the energy conservation for $v$, where a straightforward calculation gives
\begin{equation}
\label{E-v}
E[v(s)] = \lam^2(s) \, E[u(t)] = \lam^2(s) \, E[u_0].
\end{equation}
Since $v = Q+\ep$, we also obtain
\begin{align}\label{E_0}
E[Q+\ep] & = \frac12 \int |\nabla(Q+\ep)|^2 - \frac14 \int (Q+\ep)^4 \nonumber \\
& = \frac12 \left( \int |\nabla \ep|^2 +\ep^2 - 3Q^2 \ep^2 \right) + \int \left(\nabla Q \nabla \ep - Q^3 \ep \right) - \frac12 \int \ep^2 - \int Q\ep^3 - \frac14 \int \ep^4\nonumber \\
& = \frac12 \left( L \ep, \ep \right) - \left( \int Q \ep + \frac12 \int \ep^2 \right) - \frac14 \left[ 4 \int Q\ep^3 +  \int \ep^4\right],
\end{align}
where in the second line the integration by parts is used as well as $2\|\nabla Q\|^2_{L^2}=\| Q\|^4_{L^4}$ (since $Q$ is a solution of $\Delta Q + Q^3 = Q$).
By Gagliardo-Nirenberg inequality, we can bound the last term as
\begin{equation}
\label{E-est}
4 \int Q\ep^3 +  \int \ep^4 \leq c_1 \| \nabla \ep \|_{L^2} \| \ep \|^2_{L^2} + c_2 \| \nabla \ep \|^2_{L^2} \| \ep \|^2_{L^2},
\end{equation}
and if $\| \ep \|_{H^1} \leq 1$, we get
$$
\left| E[Q+\ep]  +  \left( \int Q \ep + \frac12 \int \ep^2 \right) - \frac12 (L\ep,\ep) \right| \leq c_0 \, \| \nabla \ep \|_{L^2} \| \ep \|^2_{L^2}.
$$
Putting together \eqref{m-ep}, \eqref{E-v} and \eqref{E-est}, we have the following
\begin{lemma}
\label{Lemma3}
For any $s \geq 0$ we have mass and energy conservations for $\ep$
\begin{equation}
\label{ep-M+E}
M[\ep(s)] = M_0, \quad \mbox{and} \quad E[Q+\ep(s)] = \lam^2(s) \, E[u_0].
\end{equation}
Moreover, the energy linearization is
\begin{equation}
\label{E-lin}
E[Q+\ep]  +  \left( \int Q \ep + \frac12 \int \ep^2 \right) = \frac12 (L\ep,\ep) - \frac14 \left( 4 \int Q\ep^3 +  \int \ep^4\right),
\end{equation}
and if $\|\ep \|_{H^1} \leq 1$, then there exists a $c_0 >0$ such that
\begin{equation}
\label{energy-ep-est1}
\left| E[Q+\ep]  +  \left( \int Q \ep + \frac12 \int \ep^2 \right) - \frac12 (L\ep,\ep) \right| \leq c_0 \, \| \nabla \ep \|_{L^2} \| \ep \|^2_{L^2}.
\end{equation}
\end{lemma}

\section{Modulation Theory and Parameter Estimates} \label{S-5}

In this section, recalling the definition \eqref{tube}, we show that it is possible to choose parameters $\lambda(s)\in \mathbb{R}$ and $\vec{x}(s)=(x_1(s), x_2(s))\in \mathbb{R}^2$ such that $\ep(s) \perp \chi_0$ and $\ep(s) \perp Q_{x_j}$, $j=1,2$. Moreover, assuming an additional symmetry, we can assume $x_2(s)=0$.
\begin{proposition}[Modulation Theory I]\label{ModThI}
There exists $\overline{\alpha},  \overline{\lambda} >0$ and a unique $C^1$ map
$$
(\lambda_1, \vec{x}_1): U_{\overline{\alpha}}\rightarrow (1-\overline{\lambda}, 1+\overline{\lambda})\times \mathbb{R}^2
$$
such that if $u\in U_{\overline{\alpha}}$ and $\ep_{\lambda_1,\vec{x}_1}$ is given by
\begin{equation}\label{ep-def2}
\ds \ep_{\lambda_1,\vec{x}_1}(y_1,y_2) = \lambda_1 \, u\left( \lambda_1y_1 +(x_1)_1, \lambda_1y_2 + (x_1)_2\right) - Q(y_1,y_2),
\end{equation}
then
\begin{equation}\label{ep-perp}
\ep_{\lambda_1, \vec{x}_1} \perp \chi_0 \quad \textrm{and} \quad \ep_{\lambda_1, \vec{x}_1} \perp Q_{y_j}, \quad j=1,2.
\end{equation}
Moreover, there exists a constant $C_1>0$, such that if $u\in U_{\alpha}$, with $0<\alpha<\overline{\alpha}$, then
\begin{equation}\label{ep-H1}
\|\ep_{\lambda_1, \vec{x}_1}\|_{H^1}\leq C_1\alpha \quad \textrm{and} \quad |\lambda_1-1|\leq C_1\alpha.
\end{equation}
\end{proposition}

\begin{proof}
Let $\ep_{\lambda_1,\vec{x}_1}$ be defined as in \eqref{ep-def2}. Differentiating and recalling the definition \eqref{Eq:Lambda}), we have
\begin{equation}\label{ep-p1}
\dfrac{\partial \ep_{\lambda_1,\vec{x}_1} }{\partial (x_1)_j}\Big|_{\lambda_1=1,\vec{x}_1=0}=u_{y_j}, \quad j=1,2
\end{equation}
and
\begin{equation}\label{ep-p2}
\dfrac{\partial \ep_{\lambda_1,\vec{x}_1} }{\partial \lambda_1}\Big|_{\lambda_1=1,\vec{x}_1=0}= \Lambda u.
\end{equation}

Next, consider the following functionals
$$
\rho^j_{\lambda_1, \vec{x}_1}(u)=\int \ep_{\lambda_1, \vec{x}_1} Q_{y_j}, \quad j=1,2, \quad \textrm{and}\quad
\rho^3_{\lambda_1, \vec{x}_1}(u)=\int \ep_{\lambda_1, \vec{x}_1} \chi_0,
$$
and define the function $S:\R^3 \times H^1 \rightarrow \R^3$ such that
$$
S(\lambda_1, \vec{x}_1, u) = (\rho^1_{\lambda_1, \vec{x}_1}(u), \rho^2_{\lambda_1, \vec{x}_1}(u), \rho^3_{\lambda_1, \vec{x}_1}(u)).
$$
From \eqref{ep-p1}-\eqref{ep-p2}, we deduce
\begin{align*}
\dfrac{\partial \rho^j_{\lambda_1, \vec{x}_1}(u)}{\partial \lambda_1}\Big|_{\lambda_1=1,\vec{x}_1=0, u=Q}&=\int \Lambda Q Q_{y_j}; \\
\dfrac{\partial \rho^1_{\lambda_1, \vec{x}_1}(u)}{\partial (x_1)_1}\Big|_{\lambda_1=1,\vec{x}_1=0, u=Q}&=\int Q_{y_1}Q_{y_1}= \int Q_{y_1}^2>0;\\
\dfrac{\partial \rho^2_{\lambda_1, \vec{x}_1}(u)}{\partial (x_1)_1}\Big|_{\lambda_1=1,\vec{x}_1=0, u=Q}&=\int Q_{y_1}Q_{y_2}=0; \\
\dfrac{\partial \rho^1_{\lambda_1, \vec{x}_1}(u)}{\partial (x_1)_2}\Big|_{\lambda_1=1,\vec{x}_1=0, u=Q}&=\int Q_{y_2}Q_{y_1}=0; \\
\dfrac{\partial \rho^2_{\lambda_1, \vec{x}_1}(u)}{\partial (x_1)_2}\Big|_{\lambda_1=1,\vec{x}_1=0, u=Q}&= \int Q_{y_2}Q_{y_2}= \int Q_{y_2}^2>0.\\
\end{align*}
Moreover, since $L(\chi_0)=-\lambda_0 \chi_0$ (with $\lambda_0>0$), $L (\Lambda Q) = - 2 Q$, $\chi_0$ and $Q$ are positive functions and $\chi_0 \perp \mbox{span} \{Q_{y_1}, Q_{y_2} \}$, we also have
\begin{align*}
\dfrac{\partial \rho^3_{\lambda_1, \vec{x}_1}(u)}{\partial \lambda_1}\Big|_{\lambda_1=1,\vec{x}_1=0, u=Q}&=\int \Lambda Q\chi_0=-\frac{1}{\lambda_0}\int \Lambda QL(\chi_0)=\frac{2}{\lambda_0}\int Q \chi_0>0;\\
\dfrac{\partial \rho^3_{\lambda_1, \vec{x}_1}(u)}{\partial (x_1)_1}\Big|_{\lambda_1=1,\vec{x}_1=0, u=Q}&=\int Q_{y_1}\chi_0 = 0;\\
\dfrac{\partial \rho^3_{\lambda_1, \vec{x}_1}(u)}{\partial (x_1)_2}\Big|_{\lambda_1=1,\vec{x}_1=0, u=Q}&=\int Q_{y_2}\chi_0 = 0.\\
\end{align*}
Noting that $S(1,0,0,Q)=(0,0,0)$, we can apply the Implicit Function Theorem to obtain the existence of $\overline{\beta}>0$, a neighborhood $V$ of $(1,0,0)$ in $\mathbb{R}^3$ and a unique $C^1$ map
$$
(\lambda_1, \vec{x}_1):\left\{u\in H^1(\mathbb{R}^2): \|u-Q\|_{H1}< \overline{\beta} \right\}\rightarrow V
$$
such that $S((\lambda_1, \vec{x}_1)(u),u)=0$, in other words, the orthogonality conditions \eqref{ep-perp} are satisfied.

Also note that there exists $C>0$ such that if $\|u-Q\|_{H1}< \alpha \leq \overline{\beta}$ then $|\lambda_1 - 1|+|\vec{x}_1|\leq C\alpha$. Moreover, by \eqref{ep-def} we also have $\|\ep_{\lambda_1, \vec{x}_1}\|_{H^1}\leq C\alpha$, for some $C>0$.

It is straightforward to extend the map $(\lambda_1, \vec{x}_1)$ to the region $U_{\alpha}$. Indeed, applying again the Implicit Function Theorem, there exists $\overline{\alpha}<\overline{\beta}$ and a unique $C^1$ map  $r:U_{\overline{\alpha}}\rightarrow \R^2$, such that
$$
\|u(\cdot)-Q(\cdot - r)\|_{H^1}=\inf_{r\in \R^2}\|u(\cdot)-Q(\cdot - r)\|_{H^1}<\overline{\alpha}<\overline{\beta},
$$
for all $u\in U_{\overline{\alpha}}$.

Finally, defining $\lambda_1=\lambda_1(u(\cdot+r(u)))$ and $\vec{x}_1=\vec{x}_1(u(\cdot+r(u)))+r(u)$, we have that \eqref{ep-perp} and \eqref{ep-H1} are satisfied.
\end{proof}

Note that solitary waves \eqref{Eq:TW} are traveling only in the $x_1$-direction, so it should be reasonable to consider a path
$\vec{x}(t)=(x_1(t), x_2(t))$ so that $x_1(t) \approx c \, t$ and $x_2(t) \approx 0$. Inspired by the work of de Bouard \cite{DeB}, if we assume an additional symmetry, we can consider exactly that, and thus, simplify the choice of parameters:
\begin{proposition}[Modulation Theory II]\label{ModThII}
If we assume that $u$ from the Prop \ref{ModThI} is cylindrically symmetric (i.e., $u(x_1,x_2)=u(x_1,|x_2|)$), then, reducing $\overline{\alpha}>0$ if necessary, we have $(x_1)_2\equiv 0$.
\end{proposition}
\begin{proof}
We first define $\ep_{\lambda_1,{x}_1}$ by
\begin{equation}\label{ep-def3}
\ep_{\lambda_1,{x}_1}(y_1,y_2) = \lambda_1 \, u( \lambda_1y_1 +x_1, \lambda_1y_2) - Q(y_1,y_2),
\end{equation}
and then the functionals
$$
\rho^1_{\lambda_1, {x}_1}(u)=\int \ep_{\lambda_1, {x}_1} Q_{y_j}, \quad j=1,2, \quad \textrm{and}\quad
\rho^3_{\lambda_1, {x}_1}(u)=\int \ep_{\lambda_1, {x}_1} \chi_0,
$$
and the function $S:\R^2 \times H^1 \rightarrow \R^2$ such that
$$
S(\lambda_1, x_1, u) = (\rho^1_{\lambda_1, \vec{x}_1}(u), \rho^3_{\lambda_1, \vec{x}_1}(u)).
$$
Arguing as in the proof of Proposition \ref{ModThI}, we have
\begin{align*}
\dfrac{\partial \rho^1_{\lambda_1, {x}_1}(u)}{\partial \lambda_1}\Big|_{\lambda_1=1,{x}_1=0, u=Q}&=\int \Lambda Q Q_{y_j}; \\
\dfrac{\partial \rho^3_{\lambda_1, {x}_1}(u)}{\partial \lambda_1}\Big|_{\lambda_1=1,{x}_1=0, u=Q}&=\int \Lambda Q\chi_0=-\frac{1}{\lambda_0}\int \Lambda QL(\chi_0)=\frac{2}{\lambda_0}\int Q \chi_0>0;\\
\dfrac{\partial \rho^1_{\lambda_1, {x}_1}(u)}{\partial x_1}\Big|_{\lambda_1=1,{x}_1=0, u=Q}&=\int Q_{y_1}Q_{y_1}= \int Q_{y_1}^2>0; \\
\dfrac{\partial \rho^3_{\lambda_1, {x}_1}(u)}{\partial x_1}\Big|_{\lambda_1=1,\vec{x}_1=0, u=Q}&=\int Q_{y_1}\chi_0 = 0.\\
\end{align*}

Since $S(1,0,Q)=(0,0)$, we again apply the Implicit Function Theorem to obtain the existence of $\alpha_1>0$, a neighborhood $V$ of $(1,0)$ in $\mathbb{R}^2$ and a unique $C^1$ map
$$
(\lambda_1, {x}_1):\left\{u\in H^1(\mathbb{R}^2): \|u-Q\|_{H1}< \alpha_1 \right\}\rightarrow V
$$
such that  $\ep_{\lambda_1, {x}_1} \perp \chi_0$ and $\ep_{\lambda_1, \vec{x}_1} \perp Q_{y_1}$.

Now, using the expression for $\ep_{\lambda_1, {x}_1}$ in \eqref{ep-def3}, we also deduce
$$
\int \ep_{\lambda_1, {x}_1} Q_{y_2}=\int \lambda_1 \, u( \lambda_1y_1 +x_1, \lambda_1y_2)\, Q_{y_2}(y_1,y_2) \, dy_1dy_2=0,
$$
if $u(x_1,x_2)=u(x_1,|x_2|)$, since $Q_{y_2}=\partial_r Q\cdot \frac{y_2}{r}$.

Finally, the uniqueness, which follows from the Implicit Function Theorem, yields (taking a smaller $\alpha_1$ if necessary)
$\vec{x}_1=(x_1,0)$ in Proposition \ref{ModThI}, hence, completing the proof.
\end{proof}

Now, assume that $u(t)\in U_{\overline{\alpha}}$ for all $t\geq 0$. We define the functions $\lambda(t)$ and $x(t)$ as follows.
\begin{definition}\label{eps}
For all $t\geq 0$, let $\lambda(t)$ and $x(t)$ be such that $\ep_{\lambda(t), x(t)}$, defined according to the equation \eqref{ep-def3}, satisfy
\begin{equation}\label{ep-perp2}
\ep_{\lambda(t), x(t)} \perp \chi_0 \quad \textrm{and} \quad \ep_{\lambda(t), x(t)} \perp Q_{y_j}, \quad j=1,2.
\end{equation}
In this case we also define
\begin{equation}\label{eq-ep2}
\ep(t)=\ep_{\lambda(t), x(t)}=\lambda(t) \, u(t, \lambda(t)y_1 +x(t), \lambda(t)y_2) - Q(y_1,y_2).
\end{equation}
\end{definition}

We rescale time $t \mapsto s$ by $\frac{ds}{dt} = \frac1{\lam^3}$ to better understand these parameters, which are now $\lambda(s)$ and $x(s)$. Indeed, the next proposition provides us with the equations and estimates for $\dfrac{\lambda_s}{\lambda}$ and $\left(\dfrac{x_s}{\lambda}-1\right)$.
\begin{lemma}[Modulation parameters]\label{Lemma-param}
There exists $0<\alpha_1<\overline{\alpha}$ such that if for all $t\geq 0$, $u(t)\in U_{\alpha_1}$, then $\lambda$ and ${x}$ are $C^1$ functions of $s$ and they satisfy the following equations:
\begin{eqnarray}\label{eq1}
-\frac{\lambda_s}{\lambda}\int (\vec{y}\cdot \nabla Q_{y_1})\ep + \left(\frac{x_s}{\lambda}-1\right)\left(\int |Q_{y_1}|^2-\int Q_{y_1y_1}\ep\right)\nonumber \\
=6\int QQ^2_{y_1}\ep-3\int Q_{y_1y_1}\ep^2Q-\int Q_{y_1y_1}\ep^3,
\end{eqnarray}
and
\begin{eqnarray}\label{eq3}
\frac{\lambda_s}{\lambda}\left(\frac{2}{\lambda_0} \int \chi_0 Q-\int (\vec{y}\cdot \nabla \chi_0)\ep \right)- \left(\frac{x_s}{\lambda}-1\right)\int (\chi_0)_{y_1}\ep \nonumber \\
=\int L((\chi_0)_{y_1})\ep-3\int (\chi_0)_{y_1}Q\ep^2-\int (\chi_0)_{y_1}\ep^3.
\end{eqnarray}
Moreover, there exists a universal constant $C_2>0$ such that if $\|\ep(s)\|_2\leq \alpha$, for all $s\geq 0$, where $\alpha<\alpha_1$, then
\begin{equation}\label{ControlParam}
\left|\frac{\lambda_s}{\lambda}\right|+\left|\frac{x_s}{\lambda}-1\right|\leq C_2\|\ep(s)\|_2.
\end{equation}
\end{lemma}
\begin{proof}
Let $\chi$ be a smooth function with an exponential decay. We want to calculate $\ds \frac{d}{ds}\int\chi\ep(s)$. Indeed, we have
$$
\frac{d}{ds}\int\chi u(s)=\lambda^3 \frac{d}{dt}\int \chi u(t)= - \lambda^3\int \chi (\partial_x\Delta u + \partial_x(u^3))=\lambda^3\left[\int \partial_x\Delta\chi u + \int \chi_x u^3\right].
$$
Therefore, recalling the definition of $v$ in \eqref{v-eq}, we get
\begin{align*}
\frac{d}{ds}\int\chi v(s)=&\frac{d}{ds}\int\chi(\vec{y})\lambda u(s, \lambda \vec{y}+\vec{x}(s))d\vec{y}\\
=&\frac{d}{ds}\left(\lambda^{-1}\int\chi(\lambda^{-1}(\vec{x}-\vec{x}(s))) u(s, \vec{x})d\vec{x}\right)\\
=&-\lambda^{-2}\lambda_s\int \chi (\lambda^{-1}(\vec{x}-\vec{x}(s)))u(s, \vec{x})d\vec{x}\\
&+\lambda^{-1}\int\left(\frac{d}{ds}\chi (\lambda^{-1}(\vec{x}-\vec{x}(s)))\right)u(s, \vec{x})d\vec{x}\\
&+\lambda^{-1}\lambda^3\int \chi (\lambda^{-1}(\vec{x}-\vec{x}(s)))(\partial_x\Delta u + \partial_x(u^3))d\vec{x}\\
\equiv & (A) + (B) + (C),
\end{align*}
where
\begin{align*}
(A)=&-\frac{\lambda_s}{\lambda}\int \chi v d\vec{y},\\
(B)=&\lambda^{-1}\int \nabla \chi \cdot \frac{d}{ds}(\lambda^{-1}(\vec{x}-\vec{x}(s)))u(s, \vec{x})d\vec{x}\\
=&-\frac{\lambda_s}{\lambda} \int (\nabla \chi \cdot \vec{y})vd\vec{y}-\int \left(\nabla \chi \cdot \frac{\vec{x}_s}{\lambda}\right) vd\vec{y},\\
(C)=&\lambda^{-1}\int (\partial_{x_1}\Delta\chi) (\lambda^{-1}(\vec{x}-\vec{x}(s)))ud\vec{x} + \lambda\int \chi_{x_1} (\lambda^{-1}(\vec{x}-\vec{x}(s)))u d\vec{x}\\
= & \int (\partial_{y_1}\Delta\chi) vd\vec{y}+ \int \chi_{y_1} v^3 d\vec{y}.
\end{align*}
Next, using $v=Q+\ep$ and the definition of $\Lambda$ in \eqref{Eq:Lambda}, we obtain
\begin{align*}
\frac{d}{ds}\int\chi v(s)=&-\frac{\lambda_s}{\lambda}\int (\Lambda \chi) (Q+\ep)- \left(\frac{(x_1)_s}{\lambda}-1\right)\int \chi_{y_1} (Q+\ep)- \frac{(x_2)_s}{\lambda}\int \chi_{y_2} (Q+\ep)\\
&-\int \chi_{y_1} (Q+\ep) + \int (\partial_{y_1}\Delta\chi) (Q+\ep)+ \int \chi_{y_1} (Q+\ep)^3.
\end{align*}
Recalling $L\chi_{y_1}=-\partial_{y_1} \Delta \chi + \chi_{y_1}-3Q^2\chi_{y_1}$ and $-\Delta Q + Q -Q^3=0$, we deduce
\begin{align*}
\frac{d}{ds}\int\chi v(s)=&-\frac{\lambda_s}{\lambda}\left(\int (\Lambda \chi)Q+\int (\Lambda \chi)\ep\right)\\
&- \left(\frac{(x_1)_s}{\lambda}-1\right)\left( \int \chi_{y_1}Q+ \int \chi_{y_1}\ep\right)\\
&- \frac{(x_2)_s}{\lambda}\left( \int \chi_{y_2}Q+ \int \chi_{y_2}\ep\right)\\
&-\int (L\chi_{y_1})\ep + 3\int \chi_{y_1}Q\ep^2+ \int \chi_{y_1} \ep^3.
\end{align*}
Recall that we can assume $x_2\equiv 0$ in view of Proposition \ref{ModThII}. Now, setting $x_1=x$ and taking $\chi=Q_{y_1}$ and using that $\int (\Lambda Q_{y_1})Q=0$, $\int Q_{y_1y_2}Q=\int Q_{y_1}Q_{y_2}=0 $, 
$L(Q_{y_1y_1})=6QQ^2_{y_1}$ and $\int Q_{y_1}\ep=0$, we obtain \eqref{eq1}.

Finally, fix $\chi=\chi_0$ and observe that
\begin{align*}
\int (\Lambda \chi_0) Q&=\int \chi_0 Q+\int y_1 (\chi_0)_{y_1}Q+\int y_2 (\chi_0)_{y_2}Q\\
&=-\int \chi_0 (\Lambda Q)=\frac{1}{\lambda_0}\int (L\chi_0) (\Lambda Q)\\
&=-\frac{2}{\lambda_0}\int  \chi_0 Q\neq 0.
\end{align*}
Since $\int \chi_0\ep=0$ and $\int(\chi_0)_{y_1}Q=-\int \chi_0 Q_{y_1}=0$, we have \eqref{eq3}.

Observe that there exists $\alpha_1>0$ such that
\begin{eqnarray*}
\left(\frac{2}{\lambda_0} \int \chi_0 Q-\int (\vec{y}\cdot \nabla \chi_0)\ep \right)\left(\int |Q_{y_1}|^2-\int Q_{y_1y_1}\ep\right)\\
- \left(\int (\vec{y}\cdot \nabla Q_{y_1})\ep\right)\left(\int (\chi_0)_{y_1}\ep\right)\\ \geq \frac{1}{\lambda_0}\left(\int \chi_0 Q\right)\left(\int |Q_{y_1}|^2\right),
\end{eqnarray*}
if $\|\ep(s)\|\leq \alpha <\alpha_1$, for all $s\geq 0$. Also, without loss of generality, we can assume $\alpha_1<1$.

Hence, we can solve the system of equations given by \eqref{eq1}-\eqref{eq3} and obtain a universal constant (depending only on powers of $Q$ and its partial derivatives) $C_2>0$ such that \eqref{ControlParam} holds.
In particular, if $\alpha< \frac{1}{C_2}$, we have
\begin{equation}\label{ControlParam2}
\left|\frac{\lambda_s}{\lambda}\right|+\left|\frac{x_s}{\lambda}-1\right|\leq 1.
\end{equation}
\end{proof}

\section{Virial-type estimates}\label{S-6}
Our next step is to produce a virial-type functional which will help us to study the stability properties of the solutions close to $Q$.
We first define a quantity depending on the $\ep$ variable, which incorporates the scaling generator $\Lambda$.
This can be compared with the functional we created for the supercritical case, see \cite[Section 5]{FHR2}, where the eigenfunction $\chi_0$ of $L$ for the negative eigenvalue was also used (with the coefficient $\beta$), and it was possible to find $\beta \neq 0$. However, such a functional does not work in the critical case, since $\beta$ becomes zero (due to $\int Q \Lambda Q = 0$).

We first start with defining a truncation function: let $\varphi \in C_0^{\infty}(\R)$ be a function with
\begin{equation*}\label{varphi}
\varphi(y_1)=\begin{cases}
{1, \quad \textrm{if}  \quad  y_1\leq 1 }\\
{0, \quad \textrm{if}  \quad  y_1\geq 2}.
\end{cases}
\end{equation*}
For $A\geq 1$ we also define
$$
\varphi_A(y_1)=\varphi\left(\frac{y_1}{A}\right).
$$
Note that
\begin{equation} \label{varphi1}
\varphi_A(y_1) = \left\{
\begin{array}{lll}
  1, & \quad \textrm{if} &\quad y_1\leq  A \\
  0, & \quad \textrm{if} & \quad y_1\geq  2A.
\end{array}
\right.
\end{equation}
Moreover
$$
\varphi_A^{'}(y_1)=\frac{1}{A}\varphi^{'}\left(\frac{y_1}{A}\right).
$$
We next define the function (note that we are integrating only in the first variable)
\begin{equation}\label{def-F}
F(y_1,y_2)=\int_{-\infty}^{y_1}\Lambda Q (z, y_2) \, dz.
\end{equation}
From the properties of $Q$, see \eqref{prop-Q}, there exists a constant $c>0$ such that
$$
|F(y_1,y_2)|\leq c \, e^{-\frac{1}{2}|y_2|}\int_{-\infty}^{y_1}e^{-\frac{1}{2}|z|}dz,
$$
which implies boundedness in $y_1$ and exponential decay in $y_2$:
\begin{equation}\label{F1}
\sup_{y_1 \in \R} |F(y_1,y_2)|\leq c \, e^{-\frac{1}{2}|y_2|} \quad \textrm{for all} \quad y_2 \in \R
\end{equation}
as well as
\begin{equation}\label{F2}
|F(y_1,y_2)|\leq c\, e^{-\frac{1}{2}|y_2|}e^{\frac{1}{2}y_1} \quad \textrm{for all} \quad y_1<0.
\end{equation}
Hence, $F$ is a bounded function on $\R^2$, i.e., $F\in L^{\infty}(\R^2)$. We also note that $y_2 \, F_{y_2}\in L^{\infty}(\R^2)$.

We next define the virial-type functional
\begin{equation}\label{def-JA}
J_A(s)=\int_{\R^2}\ep(s,y_1,y_2)F(y_1,y_2)\varphi_A(y_1)\, dy_1 dy_2.
\end{equation}

It is clear that $J_A(s)$ is well-defined if $\ep(s)\in L^2(\R^2)$. Indeed, since $\|\varphi_A\|_{\infty} = 1$, we can use the relation \eqref{varphi1} and the properties of $F$ to deduce
\begin{align*}
|J_A(s)|\leq &\int_{\R}\int_{y_1<0}\left|\ep(s)F(y_1,y_2)\right| dy_1dy_2+ \int_{\R}\int_{0}^{2A}\left|\ep(s)F(y_1,y_2)\right|dy_1dy_2\\
\leq &c\|\ep(s)\|_2\left(\int_{\R}\int_{y_1<0} e^{- |y_2|}e^{ y_1}dy_1dy_2 \right)^{1/2}\!\!\!\!\! + cA^{1/2}\int_{\R}\sup_{y_1}|F(y_1,y_2)|\left(\int_{0}^{2A}\left|\ep(s)\right|^2dy_1\right)^{1/2}\!\!\!\!\!dy_2\\
\leq & c\left(\int_{\R} e^{- |y_2|}dy_2\right)^{1/2} \left(\int_{y_1<0}e^{ y_1}dy_1 \right)^{1/2}\|\ep(s)\|_2+cA^{1/2}\left(\int_{\R} e^{- |y_2|} dy_2\right)^{1/2}\|\ep(s)\|_2.
\end{align*}
Therefore, we obtain the boundedness of $J_A$ from above
\begin{equation}\label{Bound-JA}
|J_A(s)|\leq c(1+A^{1/2})\|\ep(s)\|_2.
\end{equation}

Next, we compute the derivative of $J_A(s)$.
\begin{lemma}\label{J'_A}
Suppose that $\ep(s)\in H^1(\R^2)$ for all $s\geq 0$. Then the function $s\mapsto J_A(s)$ is $C^1$ and
$$
\frac{d}{ds}J_A=-\frac{\lambda_s}{\lambda}(J_A-\kappa)
+ 2\left(1-\frac12\left(\frac{x_s}{\lambda}-1\right)\right)\int\ep Q +R(\ep,A),
$$
where
\begin{equation}\label{kappa}
\kappa = \frac12 \int y_2^2\left(\int Q_{y_2}(y_1,y_2)dy_1\right)^2dy_2,
\end{equation}
and, there exists a universal constant $C_3>0$ such that, for $A\geq 1$, we have
\begin{align}\label{R}
|R(\ep,A)|\leq & C_3\left(\|\ep\|^2_2+\|\ep\|^2_2\|\ep\|_{H^1}+A^{-1/2}\|\ep\|_2\right.\nonumber\\
&+\left|\frac{x_s}{\lambda}-1\right|(A^{-1}+\|\ep\|_2)\nonumber\\
&+\left.\left|\frac{\lambda_s}{\lambda}\right|\left(A^{-1}+\|\ep\|_2+A^{1/2}\|\ep\|_{L^2(y_1\geq A)} + \left|\int_{\R^2}y_2F_{y_2}\ep\varphi_A\right|\right)\right).
\end{align}
\end{lemma}
\begin{remark}
Note that by the decay properties of $Q$, see \eqref{prop-Q}, the value of $\kappa$, defined in \eqref{kappa}, is a finite number.
\end{remark}
\begin{proof}
First, arguing as in the proof of Lemma \ref{eq-ep}, we have that the function $\ep(s)$, defined in \eqref{eq-ep2}, satisfy the following equation
\begin{equation}\label{Eq-ep}
\ep_s = (L \ep)_{y_1}  + \frac{\lam_s}{\lam} (\Lambda Q +\Lambda \ep) + \left( \frac{x_s}{\lam} -1 \right) (Q_{y_1} + \ep_{y_1}) - 3 (Q \ep^2)_{y_1} - (\ep^3)_{y_1}.
\end{equation}
Then, setting $R(\ep)= 3 Q \ep^2 + \ep^3$, we have
\begin{align*}
\frac{d}{ds}J_A= & \int \ep_sF\varphi_A \nonumber\\
=&\int \left((L\ep)_{y_1}+\frac{\lam_s}{\lam} \Lambda \ep +\left(\frac{x_s}{\lam} -1 \right) \ep_{y_1}\right)F\varphi_A\\
& +\int  \left(\frac{\lam_s}{\lam} \Lambda Q +\left(\frac{x_s}{\lam} -1 \right) Q_{y_1}\right)F\varphi_A\\
&-\int  R(\ep)_{y_1}F\varphi_A\\
&\equiv (I)+(II)+(III).
\end{align*}
Now, since $\|\varphi_A\|_{\infty}\leq 1$ and $\varphi_{y_1}\in L^{\infty}$, we have
\begin{align}\label{III}
(III)= & \int R(\ep) \Lambda Q \varphi_A+\int R(\ep) F \frac{1}{A}\varphi^{'}\left(\frac{y_1}{A}\right) \nonumber\\
\leq &\|\Lambda Q\|_{\infty}\int |R(\ep)|+\frac{\|F\|_{\infty}\|\varphi^{'}\|_{\infty}}{A}\int |R(\ep)|\nonumber\\
\leq & c_0\left(\|\Lambda Q\|_{\infty}+ \frac{\|F\|_{\infty}\|\varphi^{'}\|_{\infty}}{A}\right)\left( \|\ep\|^2_2+\|\ep\|^2_2\|\ep\|_{H^1}\right),
\end{align}
by Gagliardo-Nirenberg inequality with
$$
c_0=\|3 Q\|_{\infty}+C_{GN},
$$
here, $C_{GN}$ is the best constant for the cubic Gagliardo-Nirenberg inequality.

Furthermore,
\begin{align*}
(II)= & \frac{\lam_s}{\lam} \int \Lambda QF\varphi_A +\left(\frac{x_s}{\lam} -1 \right) \int Q_{y_1}F\varphi_A\\
\equiv &\frac{\lam_s}{\lam} (II.1)+\left(\frac{x_s}{\lam} -1 \right)(II.2),
\end{align*}
where, since $F_{y_1}=\Lambda Q$
\begin{align*}
(II.1)= & \frac12\int (F^2)_{y_1}\varphi_A\\
= &\frac12\int (F^2)_{y_1}+\frac12\int (F^2)_{y_1}(\varphi_A-1)\\
\equiv & \frac12\int \left(\int\Lambda Q (y_1, y_2)dy_1\right)^2dy_2+R_1(A),
\end{align*}
where in the last line we used \eqref{def-F}.

Now, integration by parts yields
\begin{align*}
\int \Lambda Q dy_1=&\int Q dy_1+\int y_1 Q_{y_1} dy_1+\int y_2 Q_{y_2} dy_1\\
=&\int y_2 Q_{y_2} dy_1,
\end{align*}
and thus,
$$
\int \left(\int\Lambda Q (y_1, y_2)dy_1\right)^2dy_2=\int y_2^2\left(\int Q_{y_2}(y_1,y_2)dy_1\right)^2dy_2<+\infty.
$$
Moreover, the error term can be estimated as follows
\begin{align}\label{R_1}
|R_1(A)|=&\left|\int_{\R^2} \Lambda Q F(\varphi_A-1)\right|\\
\leq & \int_{\R}\sup_{y_1}|F(y_1,y_2)|\left(\int_{\R}|\Lambda Q(\varphi_A-1)|d y_1\right)dy_2\\
\leq & 2\int_{\R}\sup_{y_1}|F(y_1,y_2)|\left(\int_{y_1\geq A}\frac{|\Lambda Q||y_1|}{|y_1|}d y_1\right)dy_2\\
\leq & \frac{2\|F\|_{\infty}\|y_1\Lambda Q\|_1}{A}.
\end{align}
On the other hand, since $\Lambda Q \perp Q$
\begin{align*}
(II.2)= & -\int Q\Lambda Q \varphi_A-\int Q F \frac{1}{A}\varphi^{'}\left(\frac{y_1}{A}\right)\\
= & 
-\int Q\Lambda Q(\varphi_A-1)-\int Q F \frac{1}{A}\varphi^{'}\left(\frac{y_1}{A}\right)\\
\equiv & R_2(A),
\end{align*}
where, using again the definition of $\varphi_A$, we have
\begin{align}\label{R_2}
|R_2(A)|\leq&2 \int_{\R}\int_{y_1\geq A}|Q\Lambda Q |\frac{|y_1|}{|y_1|}+\frac{\|F\|_{\infty}\|\varphi^{'}\|_{\infty}\|Q\|_1}{A}\nonumber \\
\leq & \frac{1}{A}\left(2\|\Lambda Q\|_{2}\|y_1 Q\|_2 +\|F\|_{\infty}\|\varphi^{'}\|_{\infty}\|Q\|_{1}\right).
\end{align}
Next we estimate the term $(I)$. Applying integration by parts, we get
\begin{align*}
(I)= & -\int (L\ep)\Lambda Q \varphi_A-\int (L\ep)F\frac{1}{A}\varphi^{'}\left(\frac{y_1}{A}\right)\\
&+\frac{\lam_s}{\lam}\int \Lambda \ep F \varphi_A \\
&-\left(\frac{x_s}{\lam} -1 \right)\left(\int \ep \Lambda Q \varphi_A+\int \ep F \frac{1}{A}\varphi^{'}\left(\frac{y_1}{A}\right)\right)\\
\equiv &(I.1)+\frac{\lam_s}{\lam} (I.2)-\left(\frac{x_s}{\lam} -1 \right)(I.3).
\end{align*}
Let us first consider the term $(I.3)$. Using the definition \eqref{Eq:Lambda}, we have
\begin{align*}
(I.3)= &\int \ep  Q + \int \ep  Q (\varphi_A-1)+\int \ep  y_1Q_{y_1} \varphi_A+\int \ep  y_2Q_{y_2} \varphi_A+\frac{1}{A}\int \ep F \varphi^{'}\left(\frac{y_1}{A}\right)\\
\equiv &\int \ep  Q+R_3(\ep,A).
\end{align*}
Next, it is easy to see that
$$
\int \ep  Q (\varphi_A-1)\leq 2\int_{\R}\int_{y_1\geq A}|\ep Q |\frac{|y_1|}{|y_1|}\leq \frac{2}{A}\|\ep\|_2\|Q\|_2,
$$
and, for $j=1,2$, we get
$$
\int \ep  y_j Q_{y_j}\varphi_A\leq \|\ep\|_2\|y_j Q_{y_j}\|_2.
$$
Moreover,
\begin{align*}
\int \ep F \varphi^{'}\left(\frac{y_1}{A}\right)\leq & \|\varphi^{'}\|_{\infty}\int_{\R}\left(\sup_{y_1}|F(y_1,y_2)|\int_A^{2A}|\ep|dy_1\right)dy_2\\
\leq & A^{1/2}\|\varphi^{'}\|_{\infty}\left(\int_{\R} \sup_{y_1}|F(y_1,y_2)|^2dy_2\right)^{1/2}\|\ep\|_2.
\end{align*}
Collecting the last three inequalities and using \eqref{F1}, we deduce
 \begin{align}\label{R_3}
|R_3(\ep,A)|\leq c\left(1+\frac{1}{A}+\frac{1}{A^1/2}\right)\|\ep\|_2,
\end{align}
where the constant $c>0$ is independent of $\ep$ and $A$.

Next, we turn to the term $(I.2)$. Integration by parts yields
 \begin{align*}
(I.2)= &\int  \ep F \varphi_A+\int  y_1\ep_{y_1} F \varphi_A+\int  y_2\ep_{y_2} F \varphi_A\\
=& -\int  y_1\ep \Lambda Q \varphi_A-\int y_1\ep F \frac{1}{A}\varphi' \left(\frac{y_1}{A}\right)\\
&-\int \ep F \varphi_A-\int  y_2\ep F_{y_2} \varphi_A\\
\equiv& -J_A + R_4(\ep,A),
\end{align*}
where in the last line we used definition \eqref{def-JA}. Let us first estimate the terms in $R_4(\ep,A)$. Indeed, it is clear that
$$
\int y_1\ep \Lambda Q \varphi_A \leq \|y_1\Lambda Q\|_2\|\ep\|_2.
$$
Furthermore,
\begin{align*}
\int y_1\ep F \frac{1}{A}\varphi'\left(\frac{y_1}{A}\right)\leq & \frac{1}{A}\|\varphi'\|_{\infty}\int_{\R}\left(\sup_{y_1}|F(y_1,y_2)|\int_A^{2A}|y_1\ep|dy_1\right)dy_2\\
\leq & 2\|\varphi'\|_{\infty}\int_{\R}\left(\sup_{y_1}|F(y_1,y_2)|\int_A^{2A}|\ep|dy_1\right)dy_2\\
\leq & 2A^{1/2}\|\varphi^{'}\|_{\infty}\left(\int_{\R} \sup_{y_1}|F(y_1,y_2)|^2dy_2\right)^{1/2}\|\ep\|_{L^2(y_1\geq A)}\\
\leq & cA^{1/2}\|\varphi^{'}\|_{\infty}\|\ep\|_{L^2(y_1\geq A)},
\end{align*}
where in the last line we used the inequality \eqref{F1}.

Collecting the last two estimates, we obtain
\begin{equation}\label{R_4}
|R_4(\ep,A)|\leq c(\|\ep\|_2+A^{1/2}\|\ep\|_{L^2(y_1\geq A)}) + \left|\int  y_2\ep F_{y_2} \varphi_A\right|,
\end{equation}
where $c>0$ is again independent of $\ep$ and $A$.

To estimate $(I.1)$, we recall the definition of the operator $L$ to get
\begin{align*}
L(fg) = & -(\Delta f)g-f(\Delta g)-2f_{y_1}g_{y_1}-2f_{y_2}g_{y_2}+fg-3Q^2fg\\
=&(Lf)g-2f_{y_1}g_{y_1}-2f_{y_2}g_{y_2}-f(\Delta g).
\end{align*}
Hence,
\begin{align*}
L(\Lambda Q \varphi_A) =&(L\Lambda Q)\varphi_A-2(\Lambda Q)_{y_1}\frac{1}{A}\varphi'\left(\frac{y_1}{A}\right)-\Lambda Q \frac{1}{A^2}\varphi^{''}\left(\frac{y_1}{A}\right)\\
\equiv& L(\Lambda Q)\varphi_A+G_A,
\end{align*}
and
\begin{align*}
L\left(F\frac{1}{A}\varphi ' \left(\frac{y_1}{A}\right)\right) =&\frac{1}{A}\left[L\left(\varphi '\left(\frac{y_1}{A}\right)\right)F-2\Lambda Q \frac{1}{A^2}\varphi^{''}\left(\frac{y_1}{A}\right)\right.\\
&\left.-\varphi ' \left(\frac{y_1}{A}\right)((\Lambda Q)_{y_1}+F_{y_2y_2})\right]\\
\equiv & \frac{1}{A} H_A.
\end{align*}
Using the fact that $L$ is a self-adjoint operator and $L(\Lambda Q)=-2Q$, we get
\begin{align*}
(I.1)= &-\int  \ep L(\Lambda Q \varphi_A)-\int \ep L\left(F\frac{1}{A}\varphi'\left(\frac{y_1}{A}\right)\right)\\
=& 2\int \ep Q \varphi_A-\int \ep\left(G_A + \frac{1}{A}H_A\right)\\
=& 2\int \ep Q+ 2\int \ep Q(\varphi_A-1)-\int \ep\left(G_A + \frac{1}{A}H_A\right)\\
\equiv & 2\int \ep Q + R_5(\ep, A).
\end{align*}
Again, we estimate the terms in $R_5(\ep, A)$ separately. First, we observe that
$$
\int\ep Q(\varphi_A-1)\leq \int_{\R}\int_{A}^{+\infty} |\ep Q|\frac{|y_1|}{|y_1|}dy_1dy_2\leq \frac{1}{A}\|y_1Q\|_2\|\ep\|_2.
$$
Moreover,
\begin{align*}
\int \ep G_A \leq & \frac{2}{A}\|\varphi '\|_{\infty}\|(\Lambda Q)_{y_1}\|_2\|\ep\|_2 + \frac{1}{A^2}\|\varphi^{''}\|_{\infty}\|\Lambda Q\|_2\|\ep\|_2.
\end{align*}
Now, note that $\|H_A\|_{\infty}\leq c$ (independent of $A\geq 1$) and
$$
\supp(H_A)\subset \left\{A\leq y_1 \leq 2A\right\}.
$$
Then, it is easy to see that (using that $F_{y_2y_2}$ also satisfies similar estimates as the ones in \eqref{F1} and \eqref{F2})
$$
\frac{1}{A}\int \ep H_A \leq \dfrac{c}{A^{1/2}}\|\ep\|_2.
$$
Finally, for $A\geq 1$, we obtain
\begin{equation}\label{R_5}
|R_5(\ep, A)|\leq \frac{c}{A^{1/2}}\|\ep\|_2,
\end{equation}
where, once again, $c>0$ is independent of $\ep$ and $A$.

Collecting all the above estimates, we finally obtain
\begin{align*}
\frac{d}{ds}J_A= & 2\int \ep Q + R_5(\ep, A)+\frac{\lam_s}{\lam} (-J_A + R_4(\ep,A))-\left(\frac{x_s}{\lam} -1 \right)\left(\int \ep  Q+R_3(\ep,A)\right)\\
& +\frac{\lam_s}{\lam} (\kappa +R_1(A))+\left(\frac{x_s}{\lam} -1 \right)R_2(A)\\
&-(III)\\
=& -\frac{\lam_s}{\lam}\left(J_A-\kappa\right)+2\left(1-\frac{1}{2}\left(\frac{x_s}{\lam}-1\right)\right)\int \ep Q + R(\ep, A),
\end{align*}
where $\kappa$ is given by \eqref{kappa} and
$$
R(\ep, A)= (III)+R_5(\ep,A)+\left(\frac{x_s}{\lam}-1\right)\left(R_2(A)-R_3(\ep, A)\right)+\frac{\lam_s}{\lam}\left(R_1(A)+R_4(\ep, A)\right).
$$
Furthermore, there exists a universal constant $C_3>0$ (independent of $\ep$ and $A$) such that, in view of \eqref{III}, \eqref{R_1}, \eqref{R_2}, \eqref{R_3}, \eqref{R_4} and \eqref{R_5}, for $A\geq 1$ the inequality \eqref{R} holds.
\end{proof}

\subsection{Control of parameters}
Before we proceed with examining further properties of $J_A$, we need to understand how various parameters are interconnected and controlled by the initial time values, especially $\ep(s)$. We proceed with the following two lemmas.
\begin{lemma}[Comparison between $M_0$, $\ep_0$ and $\int \ep_0Q$]\label{lemma-M_0}
There exists a universal constant $C_4>0$ such that, if $\|\ep_0\|_{H^1}\leq 1$, then
$$
\left|M_0-2\int\ep_0Q\right|+\left|E_0+\int\ep_0Q\right|+\left|E_0+\frac12 M_0\right|\leq C_4 \|\ep_0\|^2_{H^1}.
$$
\end{lemma}
\begin{proof}
First, observe that from the definition \eqref{M_0}, we have
$$
M_0-2\int\ep_0Q=\int\ep_0^2,
$$
and thus, $\left|M_0-2\int\ep_0Q\right|=\|\ep_0\|^2_{2}$.
Next, from \eqref{E_0}, we obtain
$$
E_0=E[Q+\ep_0]=\frac12 \left( L \ep_0, \ep_0 \right) - \frac{1}{2}M_0 - \frac14 \left[ 4 \int Q\ep_0^3 +  \int \ep_0^4\right],
$$
which implies, for some universal constant $c>0$, that
$$
\left|E_0+\frac12 M_0\right|\leq c\|\ep_0\|^2_{H^1},
$$
by the definition of $L$, the Gagliardo-Nirenberg inequality \eqref{E-est} and the fact that $\|\ep_0\|_{H^1}\leq 1$.
Finally,
$$
\left|E_0+\int\ep_0Q\right|\leq \left|E_0+\frac12 M_0\right|+\frac{1}{2}\left|M_0-2\int\ep_0Q\right|\leq \left(c+\frac{1}{2}\right) \|\ep_0\|^2_{H^1},
$$
and setting $C_4=c+\frac{1}{2}$, we conclude the proof.
\end{proof}

\begin{lemma}[Control of $\|\ep(s)\|_{H^1}$]\label{H^1-control}
There exists $\alpha_2>0$ such that, if $\|\ep(s)\|_{H^1}<\alpha$, $|\lambda(s)-1|<\alpha$ and $\ep(s)\perp \{Q_{y_1}, Q_{y_2}, \chi_0\}$ for all $s\geq 0$, where $\alpha<\alpha_2$, then there exists a universal constant $C_5>0$ such that
$$
\left(L\ep(s), \ep(s)\right)\leq \|\ep(s)\|^2_{H^1}\leq C_5\left(\alpha \left|\int \ep_0 Q\right|+\|\ep_0\|^2_{H^1}\right).
$$
\end{lemma}

\begin{proof}
From \eqref{E_0} we have
\begin{equation}
\left( L \ep(s), \ep(s) \right)= 2E[Q+\ep(s)] + M_0 + \frac14 \left[ 4 \int Q\ep^3(s) +  \int \ep^4(s)\right].
\end{equation}
Therefore, from the Gagliardo-Nirenberg inequality \eqref{E-est}, there exists a universal constant $c>0$ such that if $\|\ep(s)\|_{H^1}\leq 1$
\begin{align}\label{Coerc}
\left( L \ep(s), \ep(s) \right)\leq & 2E[Q+\ep(s)] + M_0 + c \|\ep(s)\|_{H^1}\|\ep(s)\|^2_{2} \nonumber\\
\leq &2E[Q+\ep(s)] + M_0 + \frac{c}{\sigma_0} \|\ep(s)\|_{H^1}\left( L \ep(s), \ep(s) \right),
\end{align}
where in the last line we used the coercivity of the quadratic form $(L\cdot, \cdot)$, provided $\ep(s)\perp \{Q_{y_1}, Q_{y_2}, \chi_0\}$, which was obtained in Lemma \ref{Lemma-ort3}.

Now, there exists $\alpha_2>0$ such that if $\|\ep(s)\|_{H^1}<\alpha$ for all $s\geq 0$, where $\alpha<\alpha_2$, then
$$
\frac{c}{\sigma_0}\|\ep(s)\|_{H^1}\leq \frac{1}{2}.
$$
Therefore, the last term in the RHS of \eqref{Coerc} may be absorbed by the left-hand term, and we get
\begin{align*}
\left( L \ep(s), \ep(s) \right)\leq & 4E[Q+\ep(s)] + 2M_0 \nonumber\\
\leq &4\lambda^2(s)E_0 + 2M_0,
\end{align*}
where in the last line we have used relation \eqref{E-v}.

Next, we use the last estimate to control the $H^1$-norm of $\ep(s)$ as well. Indeed, from the definition of $L$ we have
\begin{align*}
\|\ep(s)\|^2_{H^1}=&\int \ep^2(s)+\int|\nabla \ep(s)|^2=\left( L \ep(s), \ep(s) \right)+ 3\int Q^2\ep^2(s)\\
\leq & \left( L \ep(s), \ep(s) \right)+ \|3Q^2\|_{\infty}\|\ep(s)\|^2_2\\
\leq &\left(1+\frac{\|3Q^2\|_{\infty}}{\sigma_0}\right)\left( L \ep(s), \ep(s) \right)\\
\leq &\left(1+\frac{\|3Q^2\|_{\infty}}{\sigma_0}\right)(4\lambda^2(s)E_0 + 2M_0)\\
\leq &4\left(1+\frac{\|3Q^2\|_{\infty}}{\sigma_0}\right)\left((\lambda(s)-1)(\lambda(s)+1)|E_0| + \left|E_0+\frac12M_0\right|\right).
\end{align*}
Finally, since $|\lambda(s)-1|<\alpha$, choosing $\alpha <1$, we get $|\lambda(s)+1|\leq 3$, and applying Lemma \ref{lemma-M_0}, we deduce
\begin{align*}
(\lambda(s)-1)(\lambda(s)+1)|E_0|+\left|E_0+\frac12M_0\right|\leq &3 \alpha \left(\left|E_0+\int \ep_0 Q\right|+\left|\int \ep_0 Q\right|\right)+C_4 \|\ep_0\|^2_{H^1}\\
\leq & 3 \alpha \left(C_4 \|\ep_0\|^2_{H^1}+\left|\int \ep_0 Q\right|\right)+C_4 \|\ep_0\|^2_{H^1}\\
\leq & 4 C_4 \|\ep_0\|^2_{H^1}+3 \alpha \left|\int \ep_0 Q\right|,
\end{align*}
which implies the existence of a universal constant $C_5>0$ such that
$$
\left( L \ep(s), \ep(s) \right)\leq \|\ep(s)\|^2_{H^1} \leq C_5\left(\alpha \left|\int \ep_0 Q\right|+\|\ep_0\|^2_{H^1}\right).$$
\end{proof}

Our next task is to bound the time derivative of $J_A$ that we obtained in Lemma \ref{J'_A} from below. The main concern is to estimate the remainder $R(\ep,A)$ from Lemma \ref{J'_A}, and in particular, the last line \eqref{R}. Via truncation we can always choose $A$ to be large, so that the terms, which involve negative powers of $A$, would be controlled, however, the third term in \eqref{R} involves a positive power of $A$ and the tail of the $L^2$ norm of $\ep$, and hence, needs a delicate estimate. This can be done via monotonicity property, which we discuss in the next section. This is similar to our analysis of the supercritical gZK, see \cite{FHR2}. However, together with truncation and monotonicity the last term in \eqref{R} is still troublesome, and it is possible to bound it, but more is needed, namely, the smallness of that term. Thus, we need to develop another tool, the pointwise decay estimates, which we do in Sections \ref{S-8}-\ref{S-12}.

\section{Monotonicity}\label{S-7}
For $M\geq 4$, define
$$
\psi(x_1)=\frac{2}{\pi}\arctan{(e^{\frac{x_1}{M}})}.
$$
The following properties hold for $\psi$:
\begin{enumerate}
\item
$\displaystyle \psi(0)=\frac{1}{2}$,\\

\item
$\displaystyle \lim_{x_1\rightarrow -\infty} \psi(x_1)=0$ and $\displaystyle \lim_{x_1\rightarrow +\infty} \psi(x_1)=1$,\\

\item
$1-\psi(x_1)=\psi(-x_1)$,\\

\item
$\displaystyle \psi^{\prime}(x_1)=\left(\pi M \cosh\left(\frac{x_1}{M}\right)\right)^{-1}$,\\

\item
$\displaystyle \left| \psi^{\prime\prime\prime}(x_1)\right|\leq \frac{1}{M^2}\psi^{\prime}(x_1)\leq \frac{1}{16}\psi^{\prime}(x_1)$.\\
\end{enumerate}

Let $(x_1(t), x_2(t))\in C^1(\R, \R^2)$ and for $x_0, t_0>0$ and $t\in [0,t_0]$ define
\begin{equation}\label{E:I}
I_{x_0,t_0}(t)=\int u^2(t, x_1, x_2)\psi(x_1-x_1(t_0)+\frac{1}{2}(t_0-t)-x_0)dx_1dx_2,
\end{equation}
where $u\in C(\R, H^1(\R^2))$ is a solution of the gZK equation \eqref{gZK}, satisfying
\begin{equation}\label{u-Q}
\|u(t, x_1+x_1(t), x_2+x_2(t))-Q(x_1,x_2)\|_{H^1}\leq \alpha, \quad \mbox{for ~~ some} \quad \alpha>0.
\end{equation}
While the functional $I_{x_0,t_0}(t)$, which localizes the mass of the solution respectively to the moving soliton, is a concept similar to the one originated in works of Martel-Merle, and is used to study the decay of the mass of the solution to the right of the soliton, and can be applied to a variety of questions for the gKdV equations (see also our review of the instability of the critical gKdV case via monotonicity \cite{FHR1}), we note that the integration in the definition \eqref{E:I} is two dimensional. Note that the function $\psi$ is defined only in one variable $x_1$, this is similar to \cite{CDMS}. We next study the behavior of $I$ in time and we have the following monotonicity-type result.

\begin{lemma}[Almost Monotonicity]\label{AM}
Let $M\geq 4$ fixed and assume that $x_1(t)$ is an increasing function satisfying $x_1(t_0)-x_1(t)\geq \frac{3}{4}(t_0-t)$ for every $t_0, t\geq 0$ with $t\in [0,t_0]$. Then there exist $\alpha_0>0$ and $\theta=\theta(M)>0$ such that, if $u\in C(\R, H^1(\R^2))$ verify \eqref{u-Q} with $\alpha<\alpha_0$, then for all $x_0>0$, $t_0, t\geq 0$ with $t\in [0,t_0]$, we have
$$
I_{x_0,t_0}(t_0)-I_{x_0,t_0}(t)\leq \theta e^{-\frac{x_0}{M}}.
$$
\end{lemma}
\begin{proof}
Using the equation and the fact that $ \left| \psi^{\prime\prime\prime}(x)\right|\leq \frac{1}{M^2}\psi^{\prime}(x)\leq \frac{1}{16}\psi^{\prime}(x)$, we deduce
\begin{align}\label{Ix0t0}
\nonumber
\frac{d}{dt} I_{x_0,t_0}(t) = &2 \int uu_t \psi - \frac{1}{2} \int u^2\psi^{\prime}\\
\nonumber
= & -\int \left(3u_{x_1}^2+u_{x_2}^2-\frac{5}{3}u^{6}\right)\psi^{\prime}+\int u^2\psi^{\prime\prime\prime}-\frac{1}{2}\int u^2\psi^{\prime}\\
\leq & -\int \left(3u_{x_1}^2+u_{x_2}^2+\frac{1}{4}u^2\right)\psi^{\prime}+\frac{5}{3}\int u^{6}\psi^{\prime}.
\end{align}
We start with the estimate of the last term in \eqref{Ix0t0}, by using its closeness to $Q$,
\begin{align}\label{RHS}
\int u^{6}\psi^{\prime}=\int Q(\cdot-\vec{x}(t))u^5\psi^{\prime}+\int (u-Q(\cdot-\vec{x}(t)))u^5\psi^{\prime},
\end{align}
where $\vec{x}(t)=(x_1(t), x_2(t))$. To estimate the second term, we use the Sobolev embedding $H^1(\R^2)\hookrightarrow L^{q}(\R^2)$, for all $2\leq q<+\infty$, to obtain
\begin{align}\label{RHS1}
\nonumber
\int (u-Q(\cdot-\vec{x}(t)))u^5\psi^{\prime}\leq& \|(u-Q(\cdot-\vec{x}(t)))u^{3}\|_{4/3}\|u^2\psi^{\prime}\|_{4}\\
\nonumber
\leq & c \, \|u-Q(\cdot-\vec{x}(t))\|_{2}\|u\|^{3}_{12}\|u\sqrt{\psi^{\prime}}\|^2_{8}\\
\leq & c \, \|Q\|^{3}_{H^1}\alpha \int (|\nabla u|^2+|u|^2)\psi^{\prime}.
\end{align}
For the first term on the right hand side of \eqref{RHS}, we divide the integration into two regions $|\vec{x}-\vec{x}(t)|> R_0$ and $|\vec{x}-\vec{x}(t)|\leq R_0$, where $R_0$ is a positive number to be chosen later. Since $|Q(\vec{x})|\leq c \,e^{-|\vec{x}|}$, 
we obtain
\begin{align}\label{RHS2}
\nonumber
\int_{|\vec{x}-\vec{x}(t)| > R_0} Q(\cdot-\vec{x}(t))u^5\psi^{\prime}\leq
& c \, e^{- R_0}\|u^{3}\|_{3}\|u\sqrt{\psi^{\prime}}\|^2_{3}\\
\leq & c \, e^{- R_0}\|Q\|^{3}_{H^1} \int (|\nabla u|^2+|u|^2)\psi^{\prime}.
\end{align}
Next, when $|\vec{x}-\vec{x}(t)|\leq R_0$, we have
\begin{align*}
\nonumber
\left|x_1-x_1(t_0)+\frac{1}{2}(t_0-t)-x_0\right| \geq
& (x_1(t_0)-x_1(t)+x_0)-\frac{1}{2}(t_0-t)-|x_1-x_1(t)|\\
\geq & \frac{1}{4}(t_0-t)+x_0-R_0,
\end{align*}
where in the first inequality we used that $x_1(t)$ is increasing, $t_0\geq t$ and $x_0>0$ to compute the modulus of the first term, and in the second line we used the assumption $x_1(t_0)-x_1(t)\geq \frac{3}{4}(t_0-t)$.

Since $\psi^{\prime}(z)\leq \frac{2}{M\pi} e^{-\frac{|z|}{M}}$, we can use again the Sobolev embedding $H^1(\R^2)\hookrightarrow L^{q}(\R^2)$, for all $2\leq q<+\infty$, to deduce that
\begin{align}\label{RHS3}
\nonumber
\int_{|\vec{x}-\vec{x}(t)|\leq R_0} Q(\cdot-\vec{x}(t))u^{5}\psi^{\prime}\leq& \frac{2}{M\pi}\|Q\|_{\infty}e^{\frac{R_0}{M}}e^{-\frac{\left(\frac{1}{4}(t_0-t)+x_0\right)}{M}}\|u\|^5_{H^1}\\
\leq & \frac{2}{M\pi}\|Q\|_{\infty}\|Q\|^5_{H^1}e^{\frac{R_0}{M}}e^{-\frac{\left(\frac{1}{4}(t_0-t)+x_0\right)}{M}}.
\end{align}

Therefore, choosing $\alpha>0$ such that $c\,\alpha\|Q\|^{3}_{H^1}<\frac{3}{5} \cdot \frac{1}{16}$ 
and $R_0$ such that $c\, e^{-R_0}\|Q\|^{3}_{H^1}<\frac{3}{5} \cdot \frac{1}{16}$, 
collecting \eqref{RHS1}-\eqref{RHS3} together, we have
$$
\frac{5}{3}\int u^{6}\psi^{\prime}\leq \frac{1}{8}\int (|\nabla u|^2+|u|^2)\psi^{\prime}+\frac{10}{3M\pi}\|Q\|_{\infty}\|Q\|^5_{H^1}e^{\frac{R_0}{M}}e^{-\frac{\left(\frac{1}{4}(t_0-t)+x_0\right)}{M}}.
$$
Inserting the previous estimate in \eqref{Ix0t0}, 
we get that there exists a universal constant $c>0$ such that
\begin{align*}
\nonumber
\frac{d}{dt} I_{x_0,t_0}(t) 
\leq &-\int \left(\frac{3}{2}u_{x_1}^2+\frac{1}{2}u_{x_2}^2 +\frac{1}{8}u^2\right)\psi^{\prime}+\frac{c}{M}\|Q\|_{\infty}\|Q\|^5_{H^1}e^{\frac{R_0}{M}} e^{-\frac{x_0}{M}}\cdot e^{-\frac{1}{4M}(t_0-t)}\\
\nonumber
\leq & \frac{c}{M}\|Q\|_{\infty}\|Q\|^5_{H^1} e^{\frac{R_0 - x_0}{M}}\cdot e^{-\frac{1}{4M}(t_0-t)}.
\end{align*}
Finally, integrating in time on $[t,t_0]$, we obtain the desired inequality for
$$
\theta=\theta(M) = 4 \, c \, \|Q\|_{\infty}\|Q\|^5_{H^1}e^{\frac{c}{M}}>0.
$$
\end{proof}
The next lemma will used to control several terms in the virial-type functional from Section \ref{S-6} (see also Combet \cite{Co10} for a similar result for the gKdV equation).
\begin{lemma}\label{AM2}
Let $x_1(t)$ satisfying the assumptions of Lemma \ref{AM}. Also assume that $x_1(t)\geq \frac{1}{2}t$ and $x_2(t)=0$ for all $t\geq 0$. Moreover, let $u\in C(\R, H^1(\R^2))$ be a solution of the gZK equation \eqref{gzk} satisfying \eqref{u-Q} with $\alpha<\alpha_0$ (where $\alpha_0$ is given in Lemma \ref{AM}) and with the initial data $u_0$ verifying 
$\int |u_0({x_1,x_2})|^2 \, dx_2 \leq c \, e^{-\delta|{x_1}|}$ for some $c>0$ and $\delta>0$. Fix $M\geq \max\{4, \frac{2}{\delta}\}$, then there exists $C=C(M,\delta)>0$ such that for all $t\geq 0$ and $x_0>0$
\begin{equation}\label{Eq:AM2}
\int_{\R}\int_{x_1>x_0}u^2(t, x_1+x_1(t), x_2)\, dx_1dx_2 \leq C \, e^{-\frac{x_0}{M}}.
\end{equation}
\end{lemma}
\begin{proof}
From Lemma \ref{AM} with $t=0$ and replacing $t_0$ by $t$, we deduce that for all $t\geq 0$
$$
I_{x_0,t}(t)-I_{x_0,t}(0)\leq \theta e^{-\frac{x_0}{M}}.
$$
This is equivalent to
$$
\int u^2(t, x_1, x_2)\psi(x_1-x_1(t)-x_0)\, dx_1dx_2
\leq \int u_0^2(x_1,x_2)\psi(x_1-x_1(t)+\frac{1}{2}t-x_0)\, dx_1dx_2 + \theta\,  e^{-\frac{x_0}{M}}.
$$
On the other hand,
\begin{align*}
\int u^2(t, x_1,x_2)\psi(x_1-x_1(t)-x_0)\, dx_1dx_2
= &\int u^2(t, x_1+x_1(t), x_2)\psi(x_1-x_0)\, dx_1dx_2\\
\geq & \frac{1}{2}\int_{\R}\int_{x_1>x_0}u^2(t, x_1+x_1(t), x_2) \, dx_1dx_2,
\end{align*}
where in the last inequality we used the fact that $\psi$ is increasing and $\psi(0)=1/2$.

Now, since $-x_1(t)+\frac{1}{2}t\leq 0$ and $\psi$ is increasing, we get
$$
\int u_0^2(x_1,x_2)\psi(x_1-x_1(t)+\frac{1}{2}t-x_0)\, dx_1dx_2
\leq \int u_0^2(x_1,x_2)\psi(x_1-x_0) \, dx_1dx_2.
$$
Moreover, the assumption $\int |u_0({x_1,x_2})|^2 \, dx_2 \leq c \, e^{-\delta|{x_1}|}$
and the fact that $\psi(x_1)\leq c \, e^{\frac{x_1}{M}}$ for all $x_1\in \R$, yield
\begin{align*}
\int u_0^2(x_1,x_2)\psi(x_1-x_0)\, dx_1dx_2
\leq & \, c\int e^{-\delta|x_1|} e^{\frac{x_1-x_0}{M}}\, dx_1 \\
\leq& c \, e^{-\frac{x_0}{M}}\int e^{-\left(\delta-\frac{1}{M}\right)|{x_1}|} \, dx_1\\
\leq& c \, e^{-\frac{x_0}{M}}\int e^{-\frac{\delta}2 |{x_1}|}dx_1,
\end{align*}
where in the last inequality we have used the fact that
$$
\delta -\frac{1}{M}\geq \frac{\delta}{2} \Longleftrightarrow M\geq \frac{2}{\delta}.
$$
Therefore, the desired inequality \eqref{Eq:AM2} holds by taking
$$
C = 4 c \, \delta^{-1} e^{\frac{\delta |x_0|}{2}}.
$$
\end{proof}

\section{Pointwise decay for $\ep$ and review of the $H^1$ well-posedness theory}\label{S-8}

We start this section with the main statement on the pointwise decay of $\ep(x,y)$ for $x>0$.

\begin{lemma}[Pointwise Decay]\label{Point-Decay}
There exists $\sigma_0>0$ (large), $\delta_0>0$ (small)  and $K>0$ (large)
such that the following holds for any $0< \delta \leq \delta_0$ and $\sigma\geq \sigma_0$.

Recall $\ep(s,x,y)$ solving the equation \eqref{ep1} (with $y = 0$), i.e.,
\begin{equation}\label{Eq:epsilon}
\partial_s \ep = (L \ep)_{x}+ \frac{\lambda_s}{\lambda} (\Lambda Q+\Lambda \ep) + ( \frac{x_s}{\lambda}-1) (Q_{x} + \ep_{x}) - 3(Q\ep^2)_{x} - (\ep^3)_{x},
\end{equation}
and suppose there exists $\delta>0$ such that
\begin{equation}\label{Bounds}
\| \ep(s) \|_{H_{x \, y}^1} + \left|\frac{\lambda_s(s)}{\lambda(s)}\right|+
\left|\lambda(s)-1 \right| +
\left|\frac{x_s(s)}{\lambda(s)} - 1 \right| \lesssim \delta
\end{equation}
for all times $s\geq 0$.

Moreover, assume that for $x> K$ and $y\in \mathbb{R}$,
\begin{equation}\label{Idecay}
|\ep(0, x, y)| \lesssim  \delta \la x \ra^{-\sigma}.
\end{equation}

Then for $x> K$ and $y\in \mathbb{R}$, we have
\begin{equation}
\label{E:ep-bounds}
| \ep(s, x, y)| \lesssim \delta
\begin{cases}
s^{-7/12} \la x \ra^{-\sigma+\frac74} & \text{if }0 \leq s \leq 1 \\
\la x \ra^{-\frac23\sigma+\frac34} & \text{if } s\geq 1.
\end{cases}
\end{equation}
\end{lemma}
\begin{remark}\label{R:K}
Note that we can fix $K$ in Lemma \ref{Point-Decay} so that
\begin{equation}\label{Kdelta0}
\la K \ra^{-1} \leq \delta_0,
\end{equation}
which also implies that $e^{-K/2} \leq \delta_0$.
\end{remark}

\begin{remark}
Rescale the time $s$ back to $t$ via $\ds \frac{ds}{dt} = \lambda^{-3}$, and define
$$
\eta(t, x,y) = \lambda^{-1} \epsilon(s(t), \lambda^{-1}(x+K) , \lambda^{-1}y),
$$
where $x_t = \frac{d}{dt} x(t)$ with $x(t)$ being the spatial shift.
Then $\eta$ solves
\begin{equation}
\label{E:eta10}
\partial_t \eta - \partial_{x} [(-\Delta  +  x_t)\eta] = F \,, \qquad F= f_1 + \partial_{x} f_2,
\end{equation}
where
\begin{equation}
\label{E:eta11}
\begin{aligned}
& f_1 = - (\lambda^{-1})_t \partial_{\lambda^{-1}} \tilde Q,  \\
& f_2 = + (x_t -1) \tilde Q -3\tilde Q^2 \eta - 3 \tilde Q \eta^2 - \eta^3.
\end{aligned}
\end{equation}
Here, $\tilde Q(x,y) = \lambda^{-1} Q(\lambda^{-1} (x + K), \lambda^{-1} y)$.
Note that since $|Q(y_1,y_2)| \lesssim \la \vec y \ra^{-1/2} e^{-|\vec y|}$, we have $|\tilde Q(x,y)| \leq \delta_0$ for $x>0$ by \eqref{Kdelta0}.  Also for $x>0$ and $y\in \mathbb{R}$, we have
$$
|\eta(0, x,y)| \lesssim \delta \la x \ra^{-\sigma}.
$$
To show \eqref{E:ep-bounds}, it suffices to prove, for $x>0$ and $y\in \mathbb{R}$,
\begin{equation}
\label{E:etabounds}
| \eta(t,x,y)| \lesssim \delta
\begin{cases} t^{-7/12} \la x \ra^{-\sigma+\frac74} & \text{if }0 \leq t \leq 1 \\
\la x \ra^{-\frac23\sigma+\frac34} & \text{if } t\geq 1 .
\end{cases}
\end{equation}
\end{remark}

Let $S(t,t_0)\phi$ be the solution $\rho(t,x,y)$ to the homogeneous problem
$$
\begin{cases} \partial_t \rho - \partial_x( -\Delta + x_t) \rho =0 \\
\rho(t_0,x,y) = \phi(x,y).
\end{cases}
$$
Then
$$
\eta(t, x,y) = [S(t,0) \phi](x,y,t)+ \int_0^t S(t, t') F(\bullet, \bullet, t')(x,y) \, dt'.
$$
Moreover,
\begin{equation}\label{E:S}
S(t,t_0) \phi(x,y) = \int A(x',y',t-t_0) \phi(x+x(t)-x',y-y') \, dx' \, dy',
\end{equation}
where
$$
A(x,y,t) = \iint_{\mathbb{R}^2} e^{i(t\xi^3+t\xi\eta^2+x\xi+y\eta)} \, d\xi d\eta.
$$
We use the notation $L_T^p$ as shorthand for $L_{[0,T]}^p$.

\begin{theorem}[following Faminskii \cite{Fa}, Linares-Pastor \cite{LP}]
\label{T:lwp-review}
For given functions $x(t)$, $\lambda(t)$, initial data $\eta_0(x,y)\in H^1_{xy}$ and $T>0$, there exists a \emph{unique} solution $\eta$ to \eqref{E:eta10}, \eqref{E:eta11} such that $\eta \in C([0,T];H_{xy}^1)$ and $\eta(x+x(t),y,t) \in L_x^4 L_{yT}^\infty$.
\end{theorem}

This type of uniqueness is called \emph{conditional}, since it is only known to hold with the auxiliary condition $\eta(x+x(t),y,t) \in L_x^4 L_{yT}^\infty$.

Let $u(t, x,y) = \eta(t, x+x(t),y)+\tilde Q(x,y)$.  Then $u$ solves
\begin{equation}
\label{E:eta12}
\partial_t u + \partial_x (\Delta u + u^3)=0.
\end{equation}
For existence of $\eta$, it suffices to prove the existence of $u$ solving \eqref{E:eta12} such that $u \in C([0,T];H_{xy}^1) \cap L_x^4 L_{yT}^\infty$.  Moreover, given two solutions $\eta_1$ and $\eta_2$, we can define corresponding $u_1$ and $u_2$ as above.  Provided we have proved the uniqueness of solutions $u$ to \eqref{E:eta12},  we have $u_1=u_2$, which implies $\eta_1=\eta_2$.  The existence and uniqueness of solutions $u$ to \eqref{E:eta12} in the function class $C([0,T];H_{xy}^1) \cap L_x^4 L_{yT}^\infty$ was established by Linares-Pastor \cite{LP}.  It is proved by a contraction argument using the following estimates.  Let $U(t)\phi$ denote the solution to the linear homogenous problem
$$
\begin{cases}
\partial_t \rho + \partial_x \Delta \rho =0 \\
\rho(t_0,x,y) = \phi(x,y).
\end{cases}
$$
Then
$$
u(t) = U(t) \phi + \int_0^t U(t-t') \partial_x [u(t')^3] \, dt'.
$$

\begin{lemma}[linear homogeneous estimates]
\label{L:linhom}
We have
\begin{enumerate}
  \item
$\ds \|U(t) \phi \|_{L_T^\infty H_{xy}^1} \lesssim \| \phi \|_{H_{xy}^1},$

  \item
$\ds \|\partial_x U(t) \phi \|_{L_x^\infty L_{yT}^2} \lesssim \| \phi \|_{L_{xy}^2}.$
\end{enumerate}
For $0< T \leq 1$,
\begin{enumerate}
  \item[(3)]
$\ds \| U(t) \phi \|_{L_x^4 L_{yT}^\infty} \lesssim \|\phi \|_{H_{xy}^1}.$
\end{enumerate}
\end{lemma}
\begin{proof}
The first estimate is a standard consequence of Plancherel and Fourier representation of the solution.
The second estimate (local smoothing) is Faminskii \cite{Fa}, Theorem 2.2 on p. 1004.    The third estimate (maximal function estimate) is a special case ($s=1$) of Faminskii \cite{Fa}, Theorem 2.4 on p. 1007.  All of these estimates are used by Linares-Pastor \cite{LP}, and quoted as Lemma 2.7 on p. 1326 of that paper.
\end{proof}

\begin{lemma}[linear inhomogeneous estimates]
\label{L:lininhom}
For $0<T\leq 1$,
\begin{enumerate}
  \item
$\ds \left\| \int_0^t U(t-t') \partial_x f(t') \, dt' \right\|_{L_T^\infty H_{xy}^1 \cap L_x^4 L_{yT}^\infty} \lesssim \| \partial_x f \|_{L_x^1 L_{yT}^2}+ \| \partial_y f \|_{L_x^1 L_{yT}^2},$

  \item
$\ds \left\| \int_0^t U(t-t') f(t') \, dt' \right\|_{L_T^\infty H_{xy}^1 \cap L_x^4 L_{yT}^\infty} \lesssim \| f \|_{L_T^1 H_{xy}^1}.$
\end{enumerate}
\end{lemma}
\begin{proof}
These follow from Lemma \ref{L:linhom} by duality, $T^*T$, and the Christ-Kiselev lemma.
\end{proof}

Let us now summarize the proof of $H_{xy}^1$ local well-posedness following from these estimates.  We note that Linares-Pastor \cite{LP}, in fact, achieved local well-posedness in $H_{xy}^s$ for $s>\frac34$, although we only need the $s=1$ case.
Let $X$ be the $R$-ball in the Banach space $C([0,T]; H_{xy}^1) \cap L_x^4 L_{yT}^\infty$, for $T$ and $R$ yet to be chosen.  Consider the mapping $\Lambda$ defined for $u\in X$ by
$$
\Lambda u = U(t) \phi + \int_0^t U(t-t') \partial_x [u(t')^3] \, dt'.
$$
Then we claim that for suitably chosen $R>0$ and $T>0$, we have $\Lambda:X\to X$ and $\Lambda$ is a contraction.  Indeed, by the estimates in Lemmas \ref{L:linhom} and \ref{L:lininhom}, we have
$$
\| \Lambda u \|_X \lesssim \| \phi\|_{H_{xy}^1} + \| \partial_x (u^3) \|_{L_x^1 L_{yT}^2} + \| \partial_y(u^3) \|_{L_x^1 L_{yT}^2}.
$$
We estimate
$$
\| u_x u^2 \|_{L_x^1 L_{yT}^2} \lesssim \| u_x \|_{L_x^2L_{yT}^2} \|u\|_{L_x^4 L_{yT}^\infty}^2 \leq  T^{1/2} \| u_x \|_{L_T^\infty L_{xy}^2} \|u\|_{L_x^4 L_{yT}^\infty}^2,
$$
and similarly, for the $x$ derivative replaced by the $y$-derivative.  Consequently,
$$
\| \Lambda u \|_X \leq C \| \phi\|_{H_{xy}^1} + C T^{1/2} \|u\|_X^3
$$
for some constant $C>0$.  By similar estimates,
$$
\| \Lambda u_2 - \Lambda u_1 \|_X \leq CT^{1/2} \| u_2-u_1 \|_X \max(\|u_1\|_X,\|u_2\|_X)^2.
$$
We can thus take $R = 2C\|\phi \|_{H_{xy}^1}$ and $T>0$ such that $CR^2T^{1/2}=\frac12$ to obtain that $\Lambda:X\to X$ and is a contraction.  The fixed point is the desired solution.

For the uniqueness statement, we can take $R \geq 2C\|\phi \|_{H_{xy}^1}$ large enough so that the two given solutions $u_1,u_2$ lie in $X$, and then take $T$ so that $CR^2T^{1/2}=\frac12$.  Then $u_1$ and $u_2$ are both fixed points of $\Lambda$ in $X$, and since fixed points of a contraction are unique, $u_1=u_2$.

This gives the local well-posedness in $H_{xy}^1$.  Global well-posedness follows from the energy conservation.

\section{Fundamental solution estimates}\label{S-9}
Recall the solution \eqref{E:S} and its kernel $A$. The first basic step is to get the estimates on this kernel, which is given in the following statement.
\begin{proposition}[fundamental solution estimate]
\label{P:FS-est}
Let $t>0$ and consider
$$
A(x,y,t) = \iint_{\mathbb{R}^2} e^{i(t\xi^3+t\xi\eta^2+x\xi+y\eta)} \, d\xi d\eta.
$$
Let $\lambda = |x|^{3/2}t^{-1/2}>0$ and $z=y|x|^{-1}$.   Then for $x>0$
$$
|A(x,y,t)| \lesssim_{\alpha,\beta}  t^{-2/3}
\begin{cases}
\la \lambda \ra^{-\alpha} & \text{if } |z| \leq 4, \text{ for any } \alpha \geq 0 \\
 \la \lambda |z|^{3/2} \ra^{-\beta} & \text{if } |z| \geq 4\,, \text{ for any }  \beta \geq 0.
\end{cases}
$$
If $x<0$, then
$$
|A(x,y,t)| \lesssim_{\beta}  t^{-2/3}
\begin{cases}
\la \lambda \ra^{-1/6} & \text{if } |z| \leq 4\\
 \la \lambda |z|^{3/2}\ra^{-\beta} & \text{if } |z| \geq 4\,, \text{ for any }  \beta \geq 0.
\end{cases}
$$
\end{proposition}
We give a proof of Proposition \ref{P:FS-est} based on factoring of $A$ into the product of two Airy functions, which is possible in two dimensions. It is a short proof and gives estimates actually sharper than those claimed in the proposition statement. We remark that it is possible to obtain Proposition \ref{P:FS-est} by direct oscillatory integral methods (non-stationary and stationary phase), although it involves tedious calculations. The advantage of the oscillatory integral approach would be that such a proof could be generalized to higher dimensions.

\begin{proof}[Proof of Prop. \ref{P:FS-est}]
Making the change of variable $\xi \mapsto \frac{|x|^{1/2}}{t^{1/2}} \xi$ and $\eta \mapsto \frac{|x|^{1/2}}{t^{1/2}}\eta$, we obtain
$$
A(x,y,t) = |x| t^{-1} \iint_{\xi,\eta} e^{i\lambda(\xi^3+\xi\eta^2+(\sgn x) \xi + z\eta)}\, d\xi \, d\eta,
$$
where $\lambda = |x|^{3/2}t^{-1/2}>0$ and $z=y|x|^{-1}\in \mathbb{R}$, as given in the proposition statement.
Rewriting $|x|t^{-1} = \lambda^{2/3}t^{-2/3}$, this becomes
$$
A(x,y,t) = t^{-2/3} \lambda^{2/3} B_{\sgn x}(\lambda, z),
$$
where
$$
B_\pm(\lambda,z) = \iint_{\xi,\eta} e^{i\lambda (\xi^3+\xi\eta^2\pm \xi + z\eta)}\, d\xi \, d\eta.
$$
To obtain symmetric conditions in the estimates, it is convenient to change variable $\xi\mapsto \frac{\xi}{\sqrt{3}}$, and make the inconsequential replacements $\frac{\lambda}{\sqrt{3}} \mapsto \lambda$ and $\sqrt{3} z \mapsto z$.  This gives
$$
B_\pm(\lambda,z) = \iint_{\xi,\eta} e^{i\lambda \phi(\xi,\eta;z)}\, d\xi \, d\eta,
$$
where
$$\phi(\xi,\eta;z) = \tfrac13\xi^3+\xi\eta^2 \pm \xi + z\eta.
$$
Now the goal is to prove the estimates
\begin{equation}
\label{E:B+est5}
|B_+(\lambda,z)| \lesssim_{\alpha,\beta}  \lambda^{-2/3}
\begin{cases}
\la \lambda \ra^{-\alpha} & \text{if } |z| \leq 4, \text{ for any } \alpha \geq 0 \\
\la \lambda |z|^{3/2}\ra ^{-\beta} & \text{if } |z| \geq 4\,, \text{ for any } \beta \geq 0
\end{cases}
\end{equation}
and
\begin{equation}
\label{E:B-est5}
|B_-(\lambda,z)| \lesssim_{\beta}  \lambda^{-2/3}
\begin{cases}
\la \lambda \ra^{-1/6} & \text{if } |z| \leq 4\\
 \la \lambda |z|^{3/2}\ra^{-\beta} & \text{if } |z| \geq 4\,, \text{ for any } \beta \geq 0.
\end{cases}
\end{equation}
Next, make the change of variable $\xi \mapsto \frac12(\xi+\eta)$, $\eta\mapsto \frac12(\xi-\eta)$, and replace $\frac{\lambda}{2} \mapsto \lambda$, which factors the exponential to obtain the splitting
$$
B_\pm(\lambda,z) = \int_\xi e^{i\lambda(\frac13\xi^3 + (z\pm 1)\xi)} \, d\xi \int_{\eta} e^{i\lambda(\frac13\eta^3 + (-z\pm 1) \eta)} \, d\eta.
$$
In terms of the Airy function $\mathcal{A}(x) = \int e^{i(\frac13\xi^3+x\xi)} \, d\xi$, this is
\begin{equation}
\label{E:Best10}
B_\pm(\lambda,z) = \lambda^{-2/3} \mathcal{A}( \lambda^{2/3}(z\pm 1)) \mathcal{A}( \lambda^{2/3}(-z\pm 1)).
\end{equation}

We note that in either $+$ or $-$ case, if $|z|>2$, then either $(z\pm 1) \geq \frac12|z|$ or $(-z\pm 1) \geq \frac12|z|$.  Hence, by the strong decay of the Airy function on the right, we obtain
$$
|B_\pm(\lambda, z)| \lesssim \lambda^{-2/3} \la \lambda^{2/3}|z| \ra^{-k}
$$
for any $k\geq 0$.  This gives the second half of the estimates \eqref{E:B+est5} and \eqref{E:B-est5}.

For $|z|<4$, first consider the $+$ case.  If $0\leq z < 4$, then $z+1>1$, so we can use $|\mathcal{A}(\lambda^{2/3}(z+1))| \lesssim \la \lambda^{2/3} \ra^{-k}$ for any $k\geq 0$, together with the simple estimate $|\mathcal{A}(\lambda^{2/3}(-z+1))| \lesssim 1$, to achieve the first part of \eqref{E:B+est5}.  If $-4< z\leq 0$, then $-z+1>1$, so we can use $|\mathcal{A}(\lambda^{2/3}(-z+1))| \lesssim \la \lambda^{2/3} \ra^{-k}$ for any $k\geq 0$, together with the simple estimate $|\mathcal{A}(\lambda^{2/3}(z+1))| \lesssim 1$, to achieve the first part of \eqref{E:B+est5}.

For $|z|<4$, now consider the $-$ case.  When $-1\leq z \leq 1$, both $z-1\leq 0$ and $-z-1\leq 0$, so the amount of decay we can obtain from the Airy functions is limited.  The worst situation is when $z=\pm 1$.  For example, when $z=1$, we have $\mathcal{A}(\lambda^{2/3}(z-1))=A(0)$ and $|\mathcal{A}(\lambda^{2/3}(-z-1))| \lesssim \la \lambda^{2/3} \ra^{-1/4}$.  When applied in \eqref{E:Best10}, this gives the bound $|B_-(\lambda,z)| \lesssim \lambda^{-5/6}$ for $\lambda>1$.
\end{proof}

Because of the form of the equation, we also need the $x$-derivative estimate.
\begin{proposition}[$x$-derivative fundamental solution estimate]
\label{P:dFS-est}
Let $t>0$ and consider
$$
A_x(x,y,t) = \iint_{\mathbb{R}^2}  i\xi e^{i(t\xi^3+t\xi\eta^2+x\xi+y\eta)} \, d\xi d\eta.
$$
Let $\lambda = |x|^{3/2}t^{-1/2}>0$ and $z=y|x|^{-1}$.   Then for $x>0$
$$
|A_x(x,y,t)| \lesssim_{\alpha,\beta}  t^{-1}
\begin{cases}
\la \lambda \ra^{-\alpha} & \text{if } |z| \leq 4, \text{ for any } \alpha \geq 0 \\
\la \lambda |z|^{3/2} \ra^{-\beta} & \text{if } |z| \geq 4\,, \text{ for any } \beta \geq 0.
\end{cases}
$$
If $x<0$, then
$$
|A_x(x,y,t)| \lesssim_{\beta}  t^{-1}
\begin{cases}
\la \lambda \ra^{1/6} & \text{if } |z| \leq 4\\
\la \lambda |z|^{3/2}\ra^{-\beta} & \text{if } |z| \geq 4\,, \text{ for any }  \beta \geq 0.
\end{cases}
$$
\end{proposition}
\begin{proof}
Rescale $\xi \mapsto \frac{|x|^{1/2}}{(3t)^{1/2}}\xi$, $\eta \mapsto \frac{|x|^{1/2}}{t^{1/2}} \eta$ to obtain
$$
A_x(x,y,t) = \lambda t^{-1} B_\pm(\lambda, z),
$$
where $\lambda  =  |x|^{3/2} (3t)^{-1/2}$ and $z= \sqrt 3y|x|^{-1}$, and
$$
B_\pm(\lambda,z) =  \iint i\xi e^{i\lambda \phi_\pm(\xi,\eta,z)} \, d\eta \, d\eta \,,
\qquad \phi_\pm(\xi,\eta;z) = \tfrac13 \xi^3 + \xi\eta^2 \pm \xi + z\eta.
$$
It suffices to prove
\begin{equation}
\label{E:Best6}
|B_+(\lambda,z)| \lesssim_{\alpha,\beta}  \lambda^{-1}
\begin{cases}
\la \lambda \ra^{-\alpha} & \text{if } |z| \leq 4, \text{ for any } \alpha \geq 0 \\
\la \lambda |z|^{3/2} \ra^{-\beta} & \text{if } |z| \geq 4\,, \text{ for any } \beta \geq 0
\end{cases}
\end{equation}
and
\begin{equation}
\label{E:Best7}
|B_-(\lambda,z)| \lesssim_{\beta}  \lambda^{-1}
\begin{cases}
\la \lambda \ra^{1/6} & \text{if } |z| \leq 4\\
\la \lambda |z|^{3/2}\ra^{-\beta} & \text{if } |z| \geq 4\,, \text{ for any }  \beta \geq 0.
\end{cases}
\end{equation}
Changing variables again as
$$
\xi \mapsto \tfrac12(\xi+\eta) \,, \qquad \eta \mapsto \tfrac12(\xi-\eta)
$$
and replacing $\frac12\lambda$ by $\lambda$ (thus, redefining $\lambda = \frac12 |x|^{3/2} (3t)^{-1/2}$),
we obtain the factorization
\begin{align*}
8B_\pm(\lambda,z) &= \iint i(\xi+\eta) e^{i\lambda(\frac13\xi^3+(\pm 1+z)\xi)} e^{i\lambda(\frac13\eta^3+(\pm 1-z)\eta)} d\xi \, d\eta\\
&=\begin{aligned}[t]
&\int i \xi e^{i\lambda(\frac13\xi^3+(\pm 1+z)\xi)} \, d\xi \int e^{i\lambda(\frac13\eta^3+(\pm 1-z)\eta)} d\eta \\
&+ \int  e^{i\lambda(\frac13\xi^3+(\pm 1+z)\xi)} \, d\xi \int i \eta e^{i\lambda(\frac13\eta^3+(\pm 1-z)\eta)} d\eta.
\end{aligned}
\end{align*}
Change variables $\xi \mapsto \lambda^{-1/3}\xi$, $\eta \mapsto \lambda^{-1/3}\eta$ to obtain
$$
8B_\pm(\lambda, z) =
\begin{aligned}[t]
&\lambda^{-1} \mathcal{A}'( \lambda^{2/3}(\pm 1 +z))\mathcal{A}(\lambda^{2/3}(\pm 1 -z)) \\
&+ \lambda^{-1}\mathcal{A}(\lambda^{2/3}(\pm 1+ z))\mathcal{A}'(\lambda^{2/3}(\pm 1-z)),
\end{aligned}
$$
where $\mathcal{A}(x) = \int e^{i(\frac13\xi^3+x\xi)} \, d\xi$ is the Airy function.

If $z>2$, then $(\pm 1 + z)> \frac12|z|$ and if $z<-2$, then $(\pm 1-z)>\frac12|z|$, so in either case, the strong rightward decay of $\mathcal{A}$ and $\mathcal{A}'$ gives $|B_\pm(\lambda, z)| \lesssim \lambda^{-1} \la \lambda^{2/3} |z| \ra^{-k}$ for all $k\geq 0$.    Thus, the second parts of the claimed estimates in \eqref{E:Best6} and \eqref{E:Best7} hold.

In the $+$ case, if $0\leq z <2$, then $(1+z) \geq 1$, so we use that $|\mathcal{A}'(\lambda^{2/3}(1+z))| \lesssim \la \lambda^{2/3} \ra^{-k}$, $|\mathcal{A}(\lambda^{2/3}(1-z))| \lesssim 1$, $|\mathcal{A}(\lambda^{2/3}(1+z))| \lesssim \la \lambda^{2/3} \ra^{-k}$, and $|\mathcal{A}'(\lambda^{2/3}(1-z))|\lesssim 1$ to achieve $|B_+(\lambda, z)| \lesssim \lambda^{-1} \la \lambda^{2/3} \ra^{-k}$ for any $k\geq 0$.   On the other hand, if $-2<z \leq 0$, then $1-z \geq 1$, so we use that $|\mathcal{A}'(\lambda^{2/3}(1+z))| \lesssim 1$, $|\mathcal{A}(\lambda^{2/3}(1-z))| \lesssim  \la \lambda^{2/3} \ra^{-k}$, $|\mathcal{A}(\lambda^{2/3}(1+z))| \lesssim 1$, and $|\mathcal{A}'(\lambda^{2/3}(1-z))|\lesssim \la \lambda^{2/3} \ra^{-k}$ to achieve $|B_+(\lambda, z)| \lesssim \lambda^{-1} \la \lambda^{2/3} \ra^{-k}$ for any $k\geq 0$.  Thus, the first part of the estimate \eqref{E:Best6} holds.

In the $-$ case,  if $-1\leq z \leq 1$, then both $-1+z\leq 0$ and $-1-z\leq 0$, and we only have access to the weaker leftward decay estimates for the Airy function.  The worst case arises when $z=\pm 1$.  For example, if $z=1$, then $-1+z=0$ and $-1-z=-2$, so $|\mathcal{A}'(\lambda^{2/3}(-1+z))| \lesssim 1$, $|\mathcal{A}(\lambda^{2/3}(-1-z))| \lesssim \la \lambda^{2/3} \ra^{-1/4}$, $|\mathcal{A}(\lambda^{2/3}(-1+z))| \lesssim 1$, and $|\mathcal{A}'(\lambda^{2/3}(-1-z))|\lesssim \la \lambda^{2/3} \ra^{1/4}$, which gives $|B_-(\lambda,z)| \lesssim \la \lambda^{2/3} \ra^{1/4} \sim \la \lambda \ra^{1/6}$.  The case of $z=-1$ is similar.
\end{proof}

\section{Linear solution decay estimates}\label{S-10}

For this section, we will need the following estimates.  Let
$$[\mu] =
\begin{cases} \mu & \text{if } \mu>0 \\
0+ & \text{if } \mu =0 \\
0 & \text{if }\mu<0.
\end{cases}
$$

We shall employ the two basic integral estimates \eqref{E:est1}, \eqref{E:est2} below.   Note that \eqref{E:est1} requires $x>0$ and restricts the integration to $x'<0$, but then yields a stronger bound than \eqref{E:est2} when $\sigma \gg \mu$.   In fact, \eqref{E:est1} even allows $\mu<0$.

For any $\sigma\in \mathbb{R}$, $\mu\in \mathbb{R}$, $1-\sigma<\mu$ and $x>0$
\begin{equation}
\label{E:est1}
\int_{x'=-\infty}^0 \la x - x' \ra^{-\sigma} \la x' \ra^{-\mu} \, dx' \lesssim \la x \ra^{-\sigma+[1-\mu]}.
\end{equation}

For any  $\sigma>1$, $\mu\geq 0$, and $x\in \mathbb{R}$,
\begin{equation}
\label{E:est2}
\int_{x'=-\infty}^{+\infty} \la x - x' \ra^{-\sigma} \la x' \ra^{-\mu} \, dx' \lesssim \la x \ra^{-\min(\sigma-[1-\mu],\mu)}.
\end{equation}

Let
\begin{equation}
\label{E:Phi-def}
\Phi(x,y,t) = \int A(x',y',t)\phi(x+t-x',y-y') \, dx' \, dy' = (S(t)\phi)(x,y),
\end{equation}
i.e., the unique solution to $\partial_t \Phi + \partial_x (1-\Delta_{xy})\Phi =0$ with initial condition $\Phi(x,y,0)=\phi(x,y)$.  For simplicity we have taken $x(t)=t$.

The proposition below gives rightward $x$-decay estimates for this linear solution.

\begin{proposition}[linear solution estimates]
\label{P:linear-decay}
Let   $\sigma > \frac94$, and suppose that
\begin{equation}
\label{E:decay-assumption}
\text{for }x>0, \qquad |\phi(x,y)| \leq C_1 \la x \ra^{-\sigma}
\end{equation}
and
\begin{equation}\label{E:weight-left}
\| \la x \ra^{-1} \phi(x,y) \|_{L_{y\in \mathbb{R}, x<0}^2} \leq C_1.
\end{equation}
Then for $t>0$,
$$
\text{for }x>0, \qquad |\Phi(x,y,t)| \lesssim C_1
\begin{cases}
t^{-7/12} \la x \ra^{-\sigma+\frac74} & \text{if }t<1 \\
t^{-13/12}\la x \ra^{-\tilde \sigma} & \text{if }t>1,
\end{cases}
$$
where $\tilde \sigma = \min(\frac23\sigma-\frac34, \sigma-\frac94)$.
\end{proposition}

Note that for $\sigma> \frac92$, we have $\sigma-\frac94>\frac23\sigma-\frac34$, and thus, for $\tilde \sigma = \frac23\sigma-\frac34$.  We also remark that the time decay factor $t^{-13/12}$ for $t>1$ can be replaced by any negative power of $t$, provided the definition of $\tilde \sigma$ is suitably altered.  We chose $t^{-13/12}$, since it is $<-1$, thus, integrable (over $t\geq 1$), and this integrability is needed in the Duhamel estimates.

\begin{proof} By linearity, it suffices to assume that $C_1\leq 1$. Recall that we are assuming $x>0$ and $t>0$.    From Prop. \ref{P:FS-est}, we have
\begin{equation}
\label{E:A-decay}
|A(x',y',t)| \lesssim t^{-2/3} \lambda^{-\alpha} (\lambda |z|^{3/2})^{-\beta} = t^{-\frac23+\frac12\alpha+\frac12\beta} |x'|^{-\frac32\alpha} |y'|^{-\frac32\beta}
\end{equation}
with different constraints on the allowed values (and optimal values) of $\alpha$ and $\beta$ depending upon whether
\begin{itemize}
\item $|z|<4$ or $|z|>4$,
\item $\lambda<1$ or $\lambda>1$,
\item $x'<0$ or $x'>0$
\end{itemize}
This is summarized in the following two tables:
\bigskip

\begin{center}
\begin{tabular}{c || c | c}
$x'>0$ & $\lambda<1$ &$ \lambda>1$ \\
 \hline \hline
$|z|<4$ &$\alpha=0$, $\beta=0$ & $\alpha\geq 0$, $\beta=0$ \\
\hline
$|z|>4$ & $\alpha=0$, $\beta \geq 0$ & $\alpha\geq 0$, $\beta\geq 0$
 \end{tabular}
\quad
\begin{tabular}{c || c | c}
$x'<0$ & $\lambda<1$ &$ \lambda>1$ \\
 \hline \hline
$|z|<4$ &$\alpha=0$, $\beta=0$ & $\alpha=\frac16$, $\beta=0$ \\
\hline
$|z|>4$ & $\alpha=0$, $\beta \geq 0$ & $\alpha\geq 0$, $\beta\geq 0$
 \end{tabular}
 \end{center}
\bigskip

From \eqref{E:Phi-def}, we see that we need to further subdivide according to $-\infty<x'<x+t$ and $x'>x+t$.
When $-\infty<x'<x+t$, we have $x+t-x'>0$, and we can use \eqref{E:decay-assumption},
and when $x'>x+t$, we have $x+t-x'<0$, and we can only use that $\phi \in L^2_{xy}$.

Below our decomposition of the $(x',y')$ integration space is given as 9 different regions.  Each of the regions can be further divided according to whether $|x'|$ and $|y'|$ are $<1$ or $>1$.  We label the corresponding subregions as $--$, $-+$, $+-$, and $++$, as follows:
\begin{itemize}
\item $--$ corresponds to $|x'|<1$ and $|y'|<1$
\item $-+$ corresponds to $|x'|<1$ and $|y'|>1$
\item $+-$ corresponds to $|x'|>1$ and $|y'|<1$
\item $++$ corresponds to $|x'|>1$ and $|y'|>1$.
\end{itemize}
Thus,
$$\Phi = \Phi_1 + \cdots + \Phi_{9},$$
where $\Phi_j$ denotes the convolution integral in \eqref{E:Phi-def} restricted to the region under consideration.
We further use the decompositions
$$\Phi_j = \Phi_{j--}+\Phi_{j-+} +\Phi_{j+-}+\Phi_{j++} $$
as needed.

In Regions 1--5, we begin as follows:  From \eqref{E:Phi-def}, \eqref{E:decay-assumption}, and \eqref{E:A-decay}
\begin{equation}
\label{E:first}
|\Phi_j(x,y,t)| \lesssim t^{-\frac23+\frac12\alpha+\frac12\beta} \iint_{(x',y')\in R}  |x'|^{-\frac32\alpha} \la x+t-x' \ra^{-\sigma} |y'|^{-\frac32\beta}  \, dx' \, dy',
\end{equation}
where $R$ denotes the subregion of $(x',y')$ space under consideration.

In Regions 6--9, the decay hypothesis \eqref{E:decay-assumption} is not available, so we start from \eqref{E:Phi-def} with Cauchy-Schwarz in $(x',y')$
$$
|\Phi_j(x,y,t)| \leq \left(\iint_{(x',y')\in R} |A(x',y',t)|^2 \la x +t -x ' \ra^2 \, dx'\, dy' \right)^{1/2} \|\la x \ra^{-1} \phi(x,y)\|_{L^2_{y\in \mathbb{R}, x<0}}.
$$
Since $\| \la x \ra^{-1}\phi \|_{L^2_{y\in \mathbb{R}, x<0}} \leq 1$, this term can be dropped above and \eqref{E:A-decay} yields
\begin{equation}
\label{E:second}
|\Phi_j(x,y,t)| \leq t^{-\frac23+\frac12\alpha+\frac12\beta}  \left(\iint_{(x',y')\in R} |x'|^{-3\alpha} |y'|^{-3\beta} \la x +t -x'\ra^2 \, dx'\, dy' \right)^{1/2}.
\end{equation}
An argument used repeatedly below is
\begin{equation}
\label{E:lambda-below-1}
|x'|<t^{1/3} \implies \la x+t - x' \ra \sim \la x +t \ra.
\end{equation}
Indeed, if $t<8$, then $|x'|<t^{1/3}<2$, so $\la x+t-x' \ra \sim \la x+t\ra$.  On the other hand, if $t>8$, then $|x'|<t^{1/3} \ll t \leq x+t$, so again  $\la x+t-x' \ra \sim \la x+t\ra$.

Finally, we remark that it is Region 2 below that seems to limit $t^{-7/12}$ as the least singular power of $t$ for $0<t<1$.
\bigskip

\noindent \textbf{1. Region $x' < x+t$, $|x'|<t^{1/3}$, $|y'|<4|x'|$}.

Here, $\lambda<1$, $|z|<4$.  We take $\alpha=0$, $\beta=0$.   From \eqref{E:lambda-below-1}, we have $\la x+t -x'\ra \sim \la x+t\ra$.  Starting with \eqref{E:first}, we get
$$
|\Phi_1(x,y,t)| \lesssim  t^{-2/3} \la x+t \ra^{-\sigma} \int_{x' \,, |x'|<t^{1/3}}   \int_{y' \,, |y'|<4|x'|} \, dy' \, dx' \lesssim \la x+t \ra^{-\sigma}.
$$
\bigskip

\noindent \textbf{2. Region $x' < 0$, $|x'|>t^{1/3}$, $|y'|<4|x'|$}.

Here, $x'<0$, $\lambda>1$, $|z|<4$.  We take $\beta=0$, and are limited to $\alpha=\frac16$.  Starting with \eqref{E:first},
$$
|\Phi_2(x,y,t)| \lesssim  t^{-7/12} \int_{x' \,, -\infty<x'<0} \la x+t-x'\ra^{-\sigma}|x'|^{-1/4}  \int_{y' \,, |y'|<4|x'|} \, dy' \, dx'
$$
$$
\lesssim  t^{-7/12} \int_{x' \,, -\infty<x'<0} \la x+t-x'\ra^{-\sigma}|x'|^{3/4} \, dx'.
$$
By \eqref{E:est1} with $\mu=-\frac34$, provided $\sigma>\frac74$, we have
$$
\lesssim t^{-7/12} \la x+t \ra^{-\sigma+\frac74}.
$$
For $t\geq 1$, we note that $t^{-7/12}\la x +t \ra^{-\sigma+\frac74} \leq t^{-13/12} \la x \ra^{-\sigma+\frac94}$.

\bigskip

\noindent\textbf{3.  Region $0<x'<x+t$, $|x'|> t^{1/3}$, $|y'| < 4|x'|$}.

Here, $x'>0$, $\lambda>1$, $|z|<4$.   We take $\beta=0$.  For $|x'|<1$ (the $-*$ subregion), we take $\alpha=0$, but for $|x'|>1$ (the $+*$ subregion), we take $\alpha \gg 1$.

For $|x'|<1$, we take $\alpha=\frac16$ and use that $\la x+t-x'\ra \sim \la x+t\ra$ to obtain, starting with \eqref{E:first},
$$
|\Phi_{3-*}(x,y,t)| \lesssim t^{-\frac7{12}} \la x+t\ra^{-\sigma} \int_{x', |x'|<1} |x'|^{-1/4}  \, \int_{y', \, |y'|<4|x'|} dy' \, dx' \sim t^{-7/12} \la x+t \ra^{-\sigma}.
$$
In the case $t\geq 1$, we note that $t^{-7/12} \la x + t\ra^{-\sigma} \leq t^{-13/12} \la x \ra^{-\sigma+\frac12}$.

For $|x'|>1$, we take $\alpha \gg 1$, starting with \eqref{E:first},
$$
|\Phi_{3+*}(x,y,t)| \lesssim t^{-\frac23+\frac12\alpha} \int_{x', \,|x'|>1} |x'|^{-\frac32\alpha} \la x+t -x'\ra^{-\sigma} \,\int_{y', \, |y'|<4|x'|} dy' \, dx'
$$
$$
\lesssim t^{-\frac23+\frac12\alpha} \int_{x', \,|x'|>1} |x'|^{1-\frac32\alpha} \la x+t -x'\ra^{-\sigma} \, dx'.
$$
For $\sigma>1$, taking $\alpha = \frac23(\sigma+1)$ gives by \eqref{E:est2}
$$
\lesssim t^{-\frac13+\frac13\sigma} \la x+ t\ra^{-\sigma}.
$$
For $t\geq 1$, we decompose the exponent as $\la x+ t\ra^{-\sigma} = \la x+ t\ra^{\frac13-\frac13\sigma-\frac{13}{12}} \la x+t \ra^{\frac34-\frac23\sigma}$, and use $t\geq 1$ and $x\geq 0$ to obtain the bound by $t^{-13/12}\la x\ra^{\frac34-\frac23\sigma}$.

\bigskip

\noindent\textbf{4.  Region $-\infty<x'<x+t$, $|x'|<t^{1/3}$, and $|y'|> 4|x'|$}.

Here, $|z|>4$ and $\lambda<1$.   By \eqref{E:lambda-below-1}, we have $\la x+t-x' \ra \sim \la x+t \ra$.   We take $\alpha=0$ and any $\frac23<\beta<\frac43$ (so that $-2<-\frac32\beta<-1$) and starting from \eqref{E:first}, we have
$$
|\Phi_4(x,y,t)| \lesssim t^{-\frac23+\frac12\beta}  \la x+t \ra^{-\sigma} \int_{x' , \, |x'|<t^{1/3}}  \int_{y' ,\, |y'|>|x'|} |y'|^{-\frac32\beta}  \, dy' \, dx'
$$
$$
\lesssim t^{-\frac23+\frac12\beta}\la x+t \ra^{-\sigma} \int_{x' , \, |x'|<t^{1/3}}  |x'|^{-\frac32\beta+1}  \, dx' \lesssim  \la x+t \ra^{-\sigma},
$$
where, in the last step, we used that $-\frac32\beta+1>-1$.
\bigskip

\noindent\textbf{5.  Region $-\infty<x'<x+t$, $|x'|>t^{1/3}$, and $|y'|> 4|x'|$}.

Here, $|z|>4$ and $\lambda>1$.   Any choice of $\alpha, \beta \geq 0$ is permitted in \eqref{E:A-decay}.

For $|x'|<1$ and $|y'|<1$ (the $--$ subregion), we take $\alpha =\frac16$ and $\beta =0$.  Since $|x'|<1$, we have $\la x+t-x' \ra \sim \la x+t \ra$.  Starting from \eqref{E:first}, we get
$$
|\Phi_{5--}(x,y,t)| \lesssim t^{-\frac7{12}} \la x+t \ra^{-\sigma} \int_{x', \, |x'|<1}   |x'|^{-1/4}  \int_{y', \, |y'|<1} \, dy' \, dx' \lesssim t^{-7/12} \la x+t \ra^{-\sigma}.
$$

For $|x'|<1$ and $|y'|>1$ (the $-+$ subregion), we take $\alpha=0$ and $\beta=\frac23+$.  Since $|x'|<1$, we have $\la x+t-x' \ra \sim \la x+t\ra$.  Starting from \eqref{E:first},
$$
|\Phi_{5-+}(x,y,t)| \lesssim t^{-\frac23+\frac12\beta} \la x+t\ra^{-\sigma} \int_{x'\,, |x'|<1}  \, dx' \int_{y'\,, |y'|>1} |y'|^{-\frac32\beta}  \, dy'   \lesssim t^{-\frac13+} \la x+t\ra^{-\sigma}.
$$

For $|x'|>1$ and $|y'|<1$ (the $+-$ subregion), we take $\alpha\gg 1$ and $\beta =0$.   Starting from \eqref{E:first},
$$
|\Phi_{5+-}(x,y,t)| \lesssim t^{-\frac23+\frac12\alpha} \int_{x'\,, |x'|>1}  |x'|^{-\frac32\alpha} \la x+t-x' \ra^{-\sigma} \, dx' \int_{y', \, |y'|<1} \, dy'.
$$
For $\sigma>1$, take  $\alpha = \frac23\sigma$ and apply \eqref{E:est2} to obtain
$$
\lesssim t^{-\frac23+\frac13\sigma} \la x+t \ra^{-\sigma}.
$$
For $t\geq 1$, we  decompose as $\la x+t\ra^{-\sigma} = \la x+t \ra^{\frac23-\frac13\sigma-\frac{13}{12}} \la x+t\ra^{\frac{5}{12}-\frac23\sigma}$, and using $x\geq 0$, we obtain the bound $t^{-13/12}\la x \ra^{\frac{5}{12}-\frac23\sigma}$.

For $|x'|>1$ and $|y'|>1$ (the $++$ subregion), we take $\alpha \gg 1$ and any $\frac23<\beta<\frac43$ (so that $-2<-\frac32\beta<-1$ and $-1<-\frac32\beta+1<0$).  Starting from \eqref{E:first}, we obtain
$$
|\Phi_{5++}(x,y,t)| \lesssim t^{-\frac23+\frac12\alpha+\frac12\beta} \int_{x',\, |x'|>1} |x'|^{-\frac32\alpha} \la x+t-x' \ra^{-\sigma} \int_{y', \, |y'|>\max(1,|x'|)} |y'|^{-\frac32\beta}  \, dy' \, dx'
$$
$$
\lesssim t^{-\frac23+\frac12\alpha+\frac12\beta} \int_{x',\, |x'|>1} |x'|^{1-\frac32\alpha-\frac32\beta} \la x+t-x' \ra^{-\sigma}  \, dx'.
$$
Take $\alpha$ so that $\alpha+\beta = \frac23(\sigma+1)$ (and hence $1-\frac32\alpha-\frac32\beta = -\sigma$) and apply \eqref{E:est2} to obtain
$$
\lesssim t^{-\frac13+\frac13\sigma} \la x+t\ra^{-\sigma}.
$$
For $t\geq 1$, decompose $\la x+ t\ra^{-\sigma} = \la x+t \ra^{\frac13-\frac13\sigma-\frac{13}{12}} \la x + t\ra^{\frac34-\frac23\sigma}$ to obtain the bound $t^{-13/12}\la x \ra^{\frac34-\frac23\sigma}$.
\bigskip

{Recall that for Regions 6--9 below, we must use \eqref{E:second}.}  Thus, we need to recover the $\la x + t \ra^{-\sigma}$ decay factor from the constraint $x'>x+t$ and the decay on $A(x',y',t)$.
\bigskip

\noindent\textbf{6.  Region $x'>x+t$, $|x'|<t^{1/3}$, and $|y'|< 4|x'|$}.

Here, $|z|<4$ and $\lambda<1$, so $\alpha=0$ and $\beta=0$.    Note that the constraints imply that $x+t<t^{1/3}$.  Since $x+t<t^{1/3}$, it follows that $t<t^{1/3}$, from which we conclude that $t<1$.  Also from $x+t<t^{1/3}$, we conclude that $x<t^{1/3}$, and since $t<1$, this implies $x<1$.  Consequently, $\la x +t -x' \ra \sim 1$.

Starting from \eqref{E:second}, we obtain
$$
|\Phi_6(x,y,t)| \lesssim t^{-2/3}  \left( \int_{x' , \, |x'|<t^{1/3}} \int_{y' , \, |y'|<4|x'|} \, dy' \, dx' \right)^{1/2} \lesssim t^{-1/3}.
$$
\bigskip

\noindent\textbf{7.  Region $x'>x+t$, $|x'|>t^{1/3}$, and $|y'|< 4|x'|$}.

Here, $|z|<4$ and $\lambda>1$, so we take $\beta=0$, and we are allowed any $\alpha \geq 0$.

If $|x'|<1$, we take $\alpha=\frac16$.  Since $|x'|<1$, we have $x+t<1$, so $\la x+t-x'\ra \sim 1$, $x\leq 1$, and $t\leq 1$.   By \eqref{E:second},
$$
|\Phi_{7-*}(x,y,t)| \lesssim t^{-\frac7{12}} \left( \int_{x',\, |x'|<1} |x'|^{-1/2} \int_{y', \, |y'|<4|x'|} \, dy'  \, dx' \right)^{1/2} \lesssim t^{-7/12}
$$

If $|x'|>1$, we take $\alpha \gg 1$.  Starting from \eqref{E:second}, we get
$$
|\Phi_{7+*}(x,y,t)| \lesssim t^{-\frac23+\frac12\alpha} \left( \int_{x',\, x'>\max(1,x+t)} |x'|^{-3\alpha} \la x + t-x' \ra^2  \int_{y',\, |y'|<4|x'|} \, dy' \, dx' \right)^{1/2}
$$
By changing variable $\tilde x = x+t -x'$,
$$
 \lesssim t^{-\frac23+\frac12\alpha} \left( \int_{\tilde x<0} \la x+t-\tilde x\ra^{-3\alpha+1} \la \tilde x \ra^2 \, d\tilde x \right)^{1/2}
$$
By \eqref{E:est1},
$$
\lesssim t^{-\frac23+\frac12\alpha} \la x + t \ra^{-\frac32\alpha+2}.
$$
With $\alpha = \frac23\sigma+\frac43$, this becomes
$$
\lesssim t^{\frac13\sigma} \la x +t \ra^{-\sigma},
$$
which is suitable for $t<1$.  For $t>1$, first decompose $\la x+ t\ra^{-\frac32\alpha+2} = \la x + t \ra^{\frac23-\frac12\alpha-\frac{13}{12}} \la x + t\ra^{2+\frac5{12}-\alpha}$, which gives a bound by $t^{-13/12}\la x \ra^{2+\frac5{12}-\alpha}$.  We can then take $\alpha = \sigma+2+\frac{5}{12}$.
\bigskip

\noindent\textbf{8.  Region $x'>x+t$, $|x'|<t^{1/3}$, and $|y'|> 4|x'|$}.

Here, $|z|>4$ and $\lambda<1$, so we take $\alpha=0$.
Note that the constraints imply that $x+t<t^{1/3}$.  Since $x+t<t^{1/3}$, it follows that $t<t^{1/3}$, from which we conclude that $t<1$.  Also from $x+t<t^{1/3}$, we conclude that $x<t^{1/3}$ and since $t<1$, this implies $x<1$.  Consequently, $\la x+t -x'\ra \sim 1$.

For $|y'|<1$ we take $\beta =0$.  From \eqref{E:second}, we have
$$
|\Phi_{8*-}(x,y,t)| \lesssim t^{-\frac23} \left( \int_{x' ,\, |x'|<t^{1/3}} \int_{y',\, |y'|<1}\, dy' \, dx' \right)^{1/2} \lesssim t^{-1/2}.
$$
For $|y'|>1$, we take  any $\frac13<\beta<\frac23$ (so that $-3\beta<-1$ and $-1<-3\beta+1$).   From \eqref{E:second}, we have
$$
|\Phi_{8*+}(x,y,t)| \lesssim t^{-\frac23+\frac12\beta} \left( \int_{x' ,\, |x'|<t^{1/3}} \int_{y',\, |y'|>\max(1,|x'|)} |y'|^{-3\beta} \, dy' \, dx' \right)^{1/2}
$$
$$
\lesssim t^{-\frac23+\frac12\beta} \left( \int_{x' ,\, |x'|<t^{1/3}}  |x'|^{-3\beta+1}  \, dx' \right)^{1/2} \lesssim t^{-\frac13}.
$$
\bigskip

\noindent\textbf{9.  Region $x'>x+t$, $|x'|>t^{1/3}$, and $|y'|> 4|x'|$}.

Here, $|z|>4$ and $\lambda>1$.

For $|x'|<1$ and $|y'|<1$ (the $--$ subregion), we take $\alpha=\frac16$ and $\beta =0$.  Since $|x'|<1$, it follows that $x+t\leq 1$, and thus, $\la x+t -x'\ra \sim 1$.  From \eqref{E:second}, we have
$$
|\Phi_{9++}(x,y,t)| \lesssim t^{-7/12} \left( \int_{x',\, |x'|<1} |x'|^{-1/2}  \, dx' \int_{y',\, |y'|<1} \, dy' \right)^{1/2}\lesssim t^{-7/12}.
$$

For $|x'|<1$ and $|y'|>1$ (the $-+$ subregion), we take $\alpha=0$ and $\beta = \frac13+$.    Since $|x'|<1$, it follows that $x+t\leq 1$ and hence $\la x+t -x'\ra \sim 1$.  From \eqref{E:second}, we have
$$
|\Phi_{9-+}(x,y,t)| \lesssim t^{-\frac23+\frac12\beta} \left( \int_{x',\, |x'|<1} \, dx' \int_{y',\, |y'|>1} |y'|^{-3\beta} \, dy' \right)^{1/2}\lesssim t^{-\frac23+\frac12\beta} =t^{-\frac12+}.
$$

For $|x'|>1$ and $|y'|<1$ (the $+-$ subregion) , we take $\alpha \gg 1$ and $\beta=0$.  From \eqref{E:second}, we have
$$
|\Phi_{9+-}(x,y,t)| \lesssim t^{-\frac23+\frac12\alpha} \left( \int_{x',\, x'>\max(x+t,1)} |x'|^{-3\alpha}\la x +t -x' \ra^2 \, dx' \int_{y',\, |y'|<1} \, dy' \right)^{1/2}
$$
By the change of variable $\tilde x = x+t -x'$,
$$
 \lesssim t^{-\frac23+\frac12\alpha} \left( \int_{\tilde x<0} \la x+t-\tilde x\ra^{-3\alpha} \la \tilde x \ra^2 \, d\tilde x  \right)^{1/2}
$$
By \eqref{E:est1},
$$
\lesssim t^{-\frac23+\frac12\alpha} \la x+t \ra^{-\frac32\alpha+\frac32}.
$$
For $t\leq 1$, we take $\alpha = \frac23\sigma+1$ to obtain
$$
\lesssim t^{-\frac16+\frac13\sigma} \la x+t \ra^{-\sigma}.
$$
For $t>1$, we decompose the exponent $-\frac32\alpha+\frac32= (-\frac12\alpha+\frac23 -\frac{13}{12}) + (- \alpha+\frac{23}{12})$ and use $x\geq 0$ to obtain the bound $t^{-13/12} \la x\ra^{-\alpha + \frac{23}{12}}$.  Then set $\alpha = \sigma +\frac{23}{12}$ to obtain the bound $t^{-13/12}\la x\ra^{-\sigma}$.

For $|x'|>1$ and $|y'|>1$ (the $++$ subregion) , we take $\alpha \gg 1$ and $\beta = \frac13+$.  From \eqref{E:second}, we have
$$
|\Phi_{9++}(x,y,t)| \lesssim t^{-\frac23+\frac12\alpha+\frac12\beta} \left( \int_{x',\, x'>\max(x+t,1)} |x'|^{-3\alpha}\la x +t-x'\ra^2 \int_{y',\, |y'|>\max(1,|x'|)} |y'|^{-3\beta} \, dy' \, dx' \right)^{1/2}
$$
$$
\lesssim t^{-\frac23+\frac12\alpha+\frac12\beta} \left( \int_{x',\, x'>\max(x+t,1)} |x'|^{-3\alpha-3\beta+1} \la x+t-x'\ra^2 \, dx' \right)^{1/2}
$$
By changing variable $\tilde x = x+t-x'$,
$$
\lesssim t^{-\frac23+\frac12\alpha+\frac12\beta} \left( \int_{\tilde x<0} \la x+t-\tilde x \ra^{-3\alpha-3\beta+1} \la \tilde x\ra^2 \, d\tilde x \right)^{1/2}
$$
By \eqref{E:est1},
$$
\lesssim t^{-\frac23+\frac12\alpha+\frac12\beta} \la x+t \ra^{-\frac32\alpha-\frac32\beta+2}.
$$
For $t<1$, take  $\alpha$ such that $\frac32\alpha+\frac32\beta-2=\sigma$, which gives
$$
\lesssim t^{\frac13\sigma} \la x+t \ra^{-\sigma}.
$$
For $t>1$, we first decompose the exponent as $-\frac32\alpha-\frac32\beta+2=(\frac23-\frac12\alpha-\frac12\beta-\frac{13}{12}) + (-\alpha-\beta+\frac{29}{12})$, and use $x\geq 0$ to obtain the bound $t^{-13/12}\la x\ra^{-\alpha-\beta+\frac{29}{12}}$.  Then select $\alpha$ so that $\alpha+\beta-\frac{29}{12} = \sigma$ to obtain the bound by $t^{-13/12} \la x \ra^{-\sigma}$.
\end{proof}

Now we consider
\begin{equation}
\label{E:Phi-def-p}
\Phi_x(x,y,t) = \int A_x(x',y',t)\phi(x+t-x',y-y') \, dx' \, dy' = (\partial_x S(t)\phi)(x,y),
\end{equation}
i.e., the $x$-derivative of the unique solution to $\partial_t \Phi + \partial_x (1-\Delta_{xy})\Phi =0$ with initial condition $\Phi(x,y,0)=\phi(x,y)$.

\begin{proposition}[derivative linear solution estimates]
\label{P:derivative-linear-decay}
Let  $\sigma > \frac94$, and suppose that
\begin{equation}
\label{E:decay-assumption-p}
\text{for }x>0, \qquad |\phi(x,y)| \leq C_1 \la x \ra^{-\sigma}
\end{equation}
and
\begin{equation*}
\| \la x \ra^{-1} \phi(x,y) \|_{L^2_{y\in \mathbb{R}, x<0}} \leq C_1.
\end{equation*}
Then for $t>0$,
$$
\text{for }x>0, \qquad |\Phi_x(x,y,t)| \lesssim C_1 t^{-13/12}
\begin{cases}
\la x \ra^{-\sigma+\frac94} & \text{if } t<1 \\
\la x \ra^{-\tilde \sigma} & \text{if }t>1
\end{cases}
$$
Here, $\tilde \sigma=\min(\sigma-\frac94, \frac23\sigma-\frac5{12})$.  Note that if $\sigma>\frac{11}{2}$, then $\sigma-\frac94> \tilde \sigma$.
\end{proposition}

\begin{proof} By linearity, it suffices to assume that $C_1\leq 1$ and $\|\phi\|_{L^2_{xy}}\leq 1$.
 Recall that we are assuming $x>0$ and $t>0$.    From Prop. \ref{P:dFS-est}, we have
\begin{equation}
\label{E:A-decay-p}
|A_x(x',y',t)| \lesssim t^{-1} \lambda^{-\alpha} (\lambda |z|^{3/2})^{-\beta} = t^{-1+\frac12\alpha+\frac12\beta} |x'|^{-\frac32\alpha} |y'|^{-\frac32\beta}
\end{equation}
with different constraints on the allowed values (and optimal values) of $\alpha$ and $\beta$ depending upon whether
\begin{itemize}
\item $|z|<4$ or $|z|>4$,
\item $\lambda<1$ or $\lambda>1$,
\item $x'<0$ or $x'>0$.
\end{itemize}
This is summarized in the following two tables:
\bigskip

\begin{center}
\begin{tabular}{c || c | c}
$x'>0$ & $\lambda<1$ &$ \lambda>1$ \\
 \hline \hline
$|z|<4$ &$\alpha=0$, $\beta=0$ & $\alpha\geq 0$, $\beta=0$ \\
\hline
$|z|>4$ & $\alpha=0$, $\beta \geq 0$ & $\alpha\geq 0$, $\beta\geq 0$
 \end{tabular}
\quad
\begin{tabular}{c || c | c}
$x'<0$ & $\lambda<1$ &$ \lambda>1$ \\
 \hline \hline
$|z|<4$ &$\alpha=0$, $\beta=0$ & $\alpha=-\frac16$, $\beta=0$ \\
\hline
$|z|>4$ & $\alpha=0$, $\beta \geq 0$ & $\alpha\geq 0$, $\beta\geq 0$
 \end{tabular}
 \end{center}
\bigskip

From \eqref{E:Phi-def}, we see that we need to further subdivide according to $-\infty<x'<x+t$ and $x'>x+t$.  When $-\infty<x'<x+t$, we have $x+t-x'>0$ and we can use \eqref{E:decay-assumption}, and when $x'>x+t$, we have $x+t-x'<0$ and we can only use that $\phi \in L^2_{xy}$.

Below our decomposition of the $(x',y')$ integration space is given as 9 different regions.  Each of the regions can be further divided according to whether $|x'|$ and $|y'|$ are $<1$ or $>1$.  We label the corresponding subregions as $--$, $-+$, $+-$, and $++$, as follows:
\begin{itemize}
\item $--$ corresponds to $|x'|<1$ and $|y'|<1$
\item $-+$ corresponds to $|x'|<1$ and $|y'|>1$
\item $+-$ corresponds to $|x'|>1$ and $|y'|<1$
\item $++$ corresponds to $|x'|>1$ and $|y'|>1$.
\end{itemize}
Thus,
$$
\Phi = \Phi_1 + \cdots + \Phi_{9},
$$
where $\Phi_j$ denotes the convolution integral in \eqref{E:Phi-def} restricted to the region under consideration.
We further use the decompositions
$$
\Phi_j = \Phi_{j--}+\Phi_{j-+} +\Phi_{j+-}+\Phi_{j++}
$$
as needed.

In Regions 1--5, we begin as follows:  From \eqref{E:Phi-def}, \eqref{E:decay-assumption}, and \eqref{E:A-decay}
\begin{equation}
\label{E:first-p}
|\Phi_j(x,y,t)| \lesssim t^{-1+\frac12\alpha+\frac12\beta} \iint_{(x',y')\in R}  |x'|^{-\frac32\alpha} \la x+t-x' \ra^{-\sigma} |y'|^{-\frac32\beta}  \, dx' \, dy',
\end{equation}
where $R$ denotes the subregion of $(x',y')$ space under consideration.

In Regions 6--9, the decay hypothesis \eqref{E:decay-assumption} is not available, so we start from \eqref{E:Phi-def} with Cauchy-Schwarz in $(x',y')$
$$
|\Phi_j(x,y,t)| \leq \left(\iint_{(x',y')\in R} |A(x',y',t)|^2 \la x +t-x' \ra^2 \, dx'\, dy' \right)^{1/2} \|\la x \ra^{-1} \phi\|_{L^2_{xy}}.
$$
Since $\| \la x \ra^{-1} \phi \|_{L^2_{xy}} \leq 1$, this term can be dropped above and \eqref{E:A-decay} yields
\begin{equation}
\label{E:second-p}
|\Phi_j(x,y,t)| \leq t^{-1+\frac12\alpha+\frac12\beta}  \left(\iint_{(x',y')\in R} |x'|^{-3\alpha} |y'|^{-3\beta} \la x+t-x'\ra^2 \, dx'\, dy' \right)^{1/2}.
\end{equation}
An argument used repeatedly below is
\begin{equation}
\label{E:lambda-below-1-p}
|x'|<t^{1/3} \implies \la x+t - x' \ra \sim \la x +t \ra.
\end{equation}
Indeed, if $t<8$, then $|x'|<t^{1/3}<2$, so $\la x+t-x' \ra \sim \la x+t\ra$.  On the other hand, if $t>8$, then $|x'|<t^{1/3} \ll t \leq x+t$, so again  $\la x+t-x' \ra \sim \la x+t\ra$.

Another frequently employed inequality is for $\mu, \mu_1, \mu_2 \geq 0$ with $\mu_1+\mu_2=\mu$,
\begin{equation}
\label{E:t-above-1-p}
\la x + t\ra^{-\mu} \lesssim \la x \ra^{-\mu_1}\la t \ra^{-\mu_2},
\end{equation}
which is straightforward, since we are assuming $x\geq 0$ and $t\geq 0$.
\bigskip

\noindent \textbf{1. Region $x' < x+t$, $|x'|<t^{1/3}$, $|y'|<4|x'|$}.

Here, $\lambda<1$, $|z|<4$.  We take $\alpha=0$, $\beta=0$.   From \eqref{E:lambda-below-1-p}, we have $\la x+t -x'\ra \sim \la x+t\ra$.  Starting with \eqref{E:first-p}, we get
$$
|\Phi_1(x,y,t)| \lesssim  t^{-1} \la x+t \ra^{-\sigma} \int_{x' \,, |x'|<t^{1/3}}   \int_{y' \,, |y'|<4|x'|} \, dy' \, dx' \lesssim t^{-1/3} \la x+t \ra^{-\sigma} \lesssim  t^{-13/12} \la x\ra^{-\sigma+\frac34}.
$$
\bigskip

\noindent\textbf{2. Region $x' < 0$, $|x'|>t^{1/3}$, $|y'|<4|x'|$}.

Here, $x'<0$, $\lambda>1$, $|z|<4$.  We take $\beta=0$, and we are limited to $\alpha=-\frac16$.  Starting with \eqref{E:first-p},
$$
|\Phi_2(x,y,t)| \lesssim  t^{-13/12} \int_{x' \,, -\infty<x'<0} \la x+t-x'\ra^{-\sigma}|x'|^{1/4}  \int_{y' \,, |y'|<4|x'|} \, dy' \, dx'
$$
$$
\lesssim  t^{-13/12} \int_{x' \,, -\infty<x'<0} \la x+t-x'\ra^{-\sigma}|x'|^{5/4} \, dx'.
$$
By \eqref{E:est1} with $\mu=-\frac54$, provided $\sigma>\frac94$, we have
$$
\lesssim t^{-13/12} \la x+t \ra^{-\sigma+\frac94}.
$$
\bigskip

\noindent\textbf{3.  Region $0<x'<x+t$, $|x'|> t^{1/3}$, $|y'| < 4|x'|$}.

Here, $x'>0$, $\lambda>1$, $|z|<4$.  We take $\beta=0$.  For $|x'|<1$ (the $-*$ subregion), we take $\alpha=0$, but for $|x'|>1$ (the $+*$ subregion), we take $\alpha \gg 1$.

For $|x'|<1$, we take $\alpha=0$ and use that $\la x+t-x'\ra \sim \la x+t\ra$ to obtain, starting with \eqref{E:first-p},
$$
|\Phi_{3-*}(x,y,t)| \lesssim t^{-1} \la x+t\ra^{-\sigma} \int_{x', |x'|<1}   \, \int_{y', \, |y'|<4|x'|} dy' \, dx' \sim t^{-1} \la x+t \ra^{-\sigma} \lesssim t^{-\frac{13}{12}} \la x \ra^{-\sigma+\frac1{12}}.
$$

For $|x'|>1$, we take $\alpha \gg 1$, and starting with \eqref{E:first-p}, we get
$$
|\Phi_{3+*}(x,y,t)| \lesssim t^{-1+\frac12\alpha} \int_{x', \,|x'|>1} |x'|^{-\frac32\alpha} \la x+t -x'\ra^{-\sigma} \,\int_{y', \, |y'|<4|x'|} dy' \, dx'
$$
$$
\lesssim t^{-1+\frac12\alpha} \int_{x', \,|x'|>1} |x'|^{1-\frac32\alpha} \la x+t -x'\ra^{-\sigma} \, dx'.
$$
For $\sigma>1$, taking $\alpha = \frac23(\sigma+1)$ gives by \eqref{E:est2}
$$
\lesssim t^{-\frac23+\frac13\sigma} \la x+ t\ra^{-\sigma}.
$$
For $t\geq 1$, we apply \eqref{E:t-above-1-p} with $\mu=\sigma$, $\mu_1= \frac23\sigma-\frac5{12}$, and $\mu_2 = \frac13\sigma+\frac{5}{12}$, we obtain
$$
\lesssim  t^{-13/12} \la x\ra^{-\frac23\sigma+\frac{5}{12}}.
$$
\bigskip

\noindent\textbf{4.  Region $-\infty<x'<x+t$, $|x'|<t^{1/3}$, and $|y'|> 4|x'|$}.

Here, $|z|>4$ and $\lambda<1$.   By \eqref{E:lambda-below-1-p}, we have $\la x+t-x' \ra \sim \la x+t \ra$.   We take $\alpha=0$ and any $\frac23<\beta<\frac43$ (so that $-2<-\frac32\beta<-1$) and starting from \eqref{E:first-p}, we have
$$
|\Phi_4(x,y,t)| \lesssim t^{-1+\frac12\beta}  \la x+t \ra^{-\sigma} \int_{x' , \, |x'|<t^{1/3}}  \int_{y' ,\, |y'|>|x'|} |y'|^{-\frac32\beta}  \, dy' \, dx'
$$
$$
\lesssim t^{-1+\frac12\beta}\la x+t \ra^{-\sigma} \int_{x' , \, |x'|<t^{1/3}}  |x'|^{-\frac32\beta+1}  \, dx' \lesssim  t^{-1/3}\la x+t \ra^{-\sigma},
$$
where, in the last step, we used that $-\frac32\beta+1>-1$.  By \eqref{E:t-above-1-p} with $\mu=\sigma$, $\mu_1=\sigma-\frac34$, $\mu_2=\frac34$, we obtain
$$
\lesssim t^{-13/12} \la x \ra^{-\sigma+\frac34}.
$$
\bigskip

\noindent\textbf{5.  Region $-\infty<x'<x+t$, $|x'|>t^{1/3}$, and $|y'|> 4|x'|$}.

Here, $|z|>4$ and $\lambda>1$.   Any choice of $\alpha, \beta \geq 0$ is permitted in \eqref{E:A-decay-p}.

For $|x'|<1$ and $|y'|<1$ (the $--$ subregion), we take $\alpha = 0$ and $\beta =0$.  Since $|x'|<1$, we have $\la x+t-x' \ra \sim \la x+t \ra$.  Starting from \eqref{E:first-p}, we obtain
$$
|\Phi_{5--}(x,y,t)| \lesssim t^{-1} \la x+t \ra^{-\sigma} \int_{x', \, |x'|<1}    \int_{y', \, |y'|<1} \, dy' \, dx' \lesssim t^{-1} \la x+t \ra^{-\sigma}\lesssim t^{-13/12}\la x \ra^{-\sigma+\frac1{12}}.
$$

For $|x'|<1$ and $|y'|>1$ (the $-+$ subregion), we take $\alpha=0$ and $\beta=\frac23+$.  Since $|x'|<1$, we have $\la x+t-x' \ra \sim \la x+t\ra$.  Starting from \eqref{E:first-p}, we have
$$
|\Phi_{5-+}(x,y,t)| \lesssim t^{-1+\frac12\beta} \la x+t\ra^{-\sigma} \int_{x'\,, |x'|<1}  \, dx' \int_{y'\,, |y'|>1} |y'|^{-\frac32\beta}  \, dy'   \lesssim t^{-\frac23+} \la x+t\ra^{-\sigma}.
$$
By \eqref{E:t-above-1-p} with $\mu = \sigma$, $\mu_1 = \sigma-\frac{5}{12}-$, $\mu_2=\frac{5}{12}+$, we obtain
$$
\lesssim t^{-13/12} \la x \ra^{-\sigma-\frac{5}{12}-}.
$$

For $|x'|>1$ and $|y'|<1$ (the $+-$ subregion), we take $\alpha\gg 1$ and $\beta =0$.   Starting from \eqref{E:first-p}, we get
$$
|\Phi_{5+-}(x,y,t)| \lesssim t^{-1+\frac12\alpha} \int_{x'\,, |x'|>1}  |x'|^{-\frac32\alpha} \la x+t-x' \ra^{-\sigma} \, dx' \int_{y', \, |y'|<1} \, dy'.
$$
For $\sigma>1$, take  $\alpha = \frac23\sigma$ and apply \eqref{E:est2} to obtain
$$
\lesssim t^{-1+\frac13\sigma} \la x+t \ra^{-\sigma}.
$$
For $t\geq 1$, we apply \eqref{E:t-above-1-p} with $\mu=\sigma$, $\mu_1=\frac23\sigma-\frac{1}{12}$, $\mu_2=\frac13\sigma+\frac1{12}$ to obtain
$$
\lesssim t^{-13/12} \la x \ra^{-\frac23\sigma+\frac1{12}}.
$$

For $|x'|>1$ and $|y'|>1$ (the $++$ subregion), we take $\alpha \gg 1$ and any $\frac23<\beta<\frac43$ (so that $-2<-\frac32\beta<-1$ and $-1<-\frac32\beta+1<0$).  Starting from \eqref{E:first-p}, we get
$$
|\Phi_{5++}(x,y,t)| \lesssim t^{-1+\frac12\alpha+\frac12\beta} \int_{x',\, |x'|>1} |x'|^{-\frac32\alpha} \la x+t-x' \ra^{-\sigma} \int_{y', \, |y'|>\max(1,|x'|)} |y'|^{-\frac32\beta}  \, dy' \, dx'
$$
$$
\lesssim t^{-1+\frac12\alpha+\frac12\beta} \int_{x',\, |x'|>1} |x'|^{1-\frac32\alpha-\frac32\beta} \la x+t-x' \ra^{-\sigma}  \, dx'.
$$
Take $\alpha$ so that $\alpha+\beta = \frac23(\sigma+1)$ (and hence $1-\frac32\alpha-\frac32\beta = -\sigma$) and apply \eqref{E:est2} to obtain
$$
\lesssim t^{-\frac23+\frac13\sigma} \la x+t\ra^{-\sigma}.
$$
For $t\geq 1$, apply \eqref{E:t-above-1-p} with $\mu= \sigma$, $\mu_1 = \frac23\sigma-\frac5{12}$, $\mu_2=\frac13\sigma+\frac5{12}$ to obtain the bound of
$$
\lesssim t^{-13/12} \la x \ra^{-\frac23\sigma+\frac{5}{12}}.
$$
\bigskip

{Recall that for Regions 6--9 below, we must use \eqref{E:second}.}  Thus, we need to recover the $\la x + t \ra^{-\sigma}$ decay factor from the constraint $x'>x+t$ and the decay on $A(x',y',t)$.

\bigskip

\noindent\textbf{6.  Region $x'>x+t$, $|x'|<t^{1/3}$, and $|y'|< 4|x'|$}.

Here, $|z|<4$ and $\lambda<1$, so $\alpha=0$ and $\beta=0$.    Note that the constraints imply that $x+t<t^{1/3}$.  Since $x+t<t^{1/3}$, it follows that $t<t^{1/3}$, from which we conclude that $t<1$.  Also from $x+t<t^{1/3}$, we conclude that $x<t^{1/3}$, and since $t<1$, this implies $x<1$.  Consequently, $\la x+t -x' \ra \sim 1$.

Starting from \eqref{E:second-p}, we obtain
$$
|\Phi_6(x,y,t)| \lesssim t^{-1}  \left( \int_{x' , \, |x'|<t^{1/3}} \int_{y' , \, |y'|<4|x'|} \, dy' \, dx' \right)^{1/2} \lesssim t^{-2/3}.
$$
\bigskip

\noindent\textbf{7.  Region $x'>x+t$, $|x'|>t^{1/3}$, and $|y'|< 4|x'|$}.

Here, $|z|<4$ and $\lambda>1$, so we take $\beta=0$, and we are allowed any $\alpha \geq 0$.

If $|x'|<1$, we take $\alpha=0$.  Since $|x'|<1$, we have $x+t<1$, so $\la x+t-x'\ra \sim 1$.    By \eqref{E:second-p},
$$
|\Phi_{7-*}(x,y,t)| \lesssim t^{-1} \left( \int_{x',\, |x'|<1} \int_{y', \, |y'|<4|x'|} \, dy'  \, dx' \right)^{1/2} \lesssim t^{-1}.
$$

If $|x'|>1$, we take $\alpha \gg 1$.  Starting from \eqref{E:second-p}, we get
$$
|\Phi_{7+*}(x,y,t)| \lesssim t^{-1+\frac12\alpha} \left( \int_{x',\, x'>\max(1,x+t)} |x'|^{-3\alpha}\la x+t -x'\ra^2  \int_{y',\, |y'|<4|x'|} \, dy' \, dx' \right)^{1/2}
$$
By the change of variable $\tilde x = x+t-x'$,
$$ \lesssim t^{-1+\frac12\alpha} \left( \int_{\tilde x<0} \la x-t-\tilde x\ra^{-3\alpha+1}\la \tilde x \ra^2  \, dx' \right)^{1/2}$$
By \eqref{E:est1},
$$
\lesssim t^{-1+\frac12\alpha} \la x + t \ra^{-\frac32\alpha+2}.
$$
By \eqref{E:t-above-1-p} with $\mu=\frac32\alpha-2$, $\mu_1=\sigma$ and $\mu_2=\frac32\alpha-2-\sigma$, we obtain
$$
\lesssim t^{\sigma-\alpha+1} \la x \ra^{-\sigma}.
$$
We thus take $\alpha = \sigma + \frac{25}{12}$.
\bigskip

\noindent\textbf{8.  Region $x'>x+t$, $|x'|<t^{1/3}$, and $|y'|> 4|x'|$}.

Here, $|z|>4$ and $\lambda<1$, so we take $\alpha=0$.    Note that the constraints imply that $x+t<t^{1/3}$.  Since $x+t<t^{1/3}$, it follows that $t<t^{1/3}$, from which we conclude that $t<1$.  Also from $x+t<t^{1/3}$, we conclude that $x<t^{1/3}$, and since $t<1$, this implies $x<1$.  Consequently, $\la x+t-x' \ra \sim 1$.

For $|y'|<1$ we take $\beta =0$.  From \eqref{E:second-p}, we have
$$
|\Phi_{8*-}(x,y,t)| \lesssim t^{-1} \left( \int_{x' ,\, |x'|<t^{1/3}} \int_{y',\, |y'|<1}\, dy' \, dx' \right)^{1/2} \lesssim t^{-5/6}.
$$
For $|y'|>1$, we take  any $\frac13<\beta<\frac23$ (so that $-3\beta<-1$ and $-1<-3\beta+1$).   From \eqref{E:second-p}, we have
$$
|\Phi_{8*+}(x,y,t)| \lesssim t^{-1+\frac12\beta} \left( \int_{x' ,\, |x'|<t^{1/3}} \int_{y',\, |y'|>\max(1,|x'|)} |y'|^{-3\beta} \, dy' \, dx' \right)^{1/2}
$$
$$
\lesssim t^{-1+\frac12\beta} \left( \int_{x' ,\, |x'|<t^{1/3}}  |x'|^{-3\beta+1}  \, dx' \right)^{1/2} \lesssim t^{-\frac13-\beta}.
$$
\bigskip

\noindent\textbf{9.  Region $x'>x+t$, $|x'|>t^{1/3}$, and $|y'|> 4|x'|$}.

Here, $|z|>4$ and $\lambda>1$.

For $|x'|<1$ and $|y'|<1$ (the $--$ subregion), we take $\alpha=0$ and $\beta =0$.  Since $|x'|<1$, it follows that $x+t\leq 1$, and thus, $\la x+t -x'\ra \sim 1$.  From \eqref{E:second-p}, we have
$$
|\Phi_{9++}(x,y,t)| \lesssim t^{-1} \left( \int_{x',\, |x'|<1} \, dx' \int_{y',\, |y'|<1} \, dy' \right)^{1/2}\lesssim t^{-1}.
$$

For $|x'|<1$ and $|y'|>1$ (the $-+$ subregion) , we take $\alpha=0$ and $\beta = \frac13+$.    Since $|x'|<1$, it follows that $x+t\leq 1$, and thus, $\la x+t -x'\ra \sim 1$. From \eqref{E:second-p}, we have
$$
|\Phi_{9-+}(x,y,t)| \lesssim t^{-1+\frac12\beta} \left( \int_{x',\, |x'|<1} \, dx' \int_{y',\, |y'|>1} |y'|^{-3\beta} \, dy' \right)^{1/2}\lesssim t^{-1+\frac12\beta} =t^{-\frac56+}.
$$

For $|x'|>1$ and $|y'|<1$ (the $+-$ subregion) , we take $\alpha \gg 1$ and $\beta=0$.  From \eqref{E:second-p}, we have
$$
|\Phi_{9+-}(x,y,t)| \lesssim t^{-1+\frac12\alpha} \left( \int_{x',\, x'>\max(x+t,1)} |x'|^{-3\alpha}\la x +t-x'\ra^2 \, dx' \int_{y',\, |y'|<1} \, dy' \right)^{1/2}
$$
By the change of variable $\tilde x=x+t-x'$,
$$
|\Phi_{9+-}(x,y,t)| \lesssim t^{-1+\frac12\alpha} \left( \int_{\tilde x<0} \la x+t-\tilde x\ra^{-3\alpha}\la \tilde x \ra^2  \right)^{1/2}
$$
$$
\lesssim t^{-1+\frac12\alpha} \la x+t \ra^{-\frac32\alpha+\frac32}.
$$
Apply \eqref{E:t-above-1-p} with $\mu=\frac32\alpha-\frac12$, $\mu_1=\sigma$, and $\mu_2=\frac32\alpha-\frac32-\sigma$ to obtain
$$
|\Phi_{9+-}(x,y,t)| \lesssim t^{\frac12-\alpha+\sigma} \la x \ra^{-\sigma}.
$$
Taking $\alpha = \sigma+\frac{19}{12}$, we obtain the desired bound.

For $|x'|>1$ and $|y'|>1$ (the $++$ subregion), we take $\alpha \gg 1$ and $\beta = \frac13+$.  From \eqref{E:second-p}, we have
$$
|\Phi_{9++}(x,y,t)| \lesssim t^{-1+\frac12\alpha+\frac12\beta} \left( \int_{x',\, x'>\max(x+t,1)} |x'|^{-3\alpha} \la x+t-x'\ra^2 \int_{y',\, |y'|>\max(1,|x'|)} |y'|^{-3\beta} \, dy' \, dx' \right)^{1/2}
$$
$$
\lesssim t^{-1+\frac12\alpha+\frac12\beta} \left( \int_{x',\, x'>\max(x+t,1)} |x'|^{-3\alpha-3\beta+1}\la x +t-x'\ra^2 \, dx' \right)^{1/2}
$$
By the change of variable $\tilde x = x+t-x'$, we obtain
$$
\lesssim t^{-1+\frac12\alpha+\frac12\beta} \left( \int_{\tilde x<0} \la x+t-\tilde x\ra^{-3\alpha-3\beta+1}\la \tilde x\ra^2 \, dx' \right)^{1/2}
$$
By \eqref{E:est1},
$$
\lesssim t^{-1+\frac12\alpha+\frac12\beta} \la x+t \ra^{-\frac32\alpha-\frac32\beta+2}.
$$
By \eqref{E:t-above-1-p} with $\mu=\frac32\alpha+\frac23\beta-2$, $\mu_1 = \sigma$, and $\mu_2 = \frac32\alpha+\frac32\beta-2-\sigma$, we obtain
$$
\lesssim t^{1-\alpha-\beta+\sigma} \la x \ra^{-\sigma}.
$$
Taking $\alpha = -\beta + \sigma + \frac{25}{12}$ gives $t^{-13/12} \la x\ra^{-\sigma}$.
\end{proof}

\section{Duhamel estimate}\label{S-11}

We will now use Prop. \ref{P:derivative-linear-decay} to prove a Duhamel estimate in Prop. \ref{P:Duhamel} below.  First, we need
\begin{lemma}
\label{L:time-integral}
Suppose $\mu>1$ and $\nu>1$, $0<\sigma_2\leq \sigma_1$, and thus, $\tilde \sigma_2\leq \tilde \sigma_1$ (where $\tilde \sigma_j$ is given in terms of $\sigma_j$ as in Prop. \ref{P:derivative-linear-decay}) , $t>0$ and that $f(t,t')\geq 0$ satisfies for $0<t'<t$,
$$f(t,t') \lesssim (t-t')^{-\mu}
\begin{cases}
(t')^{-\nu} \la x \ra^{-\sigma_1} & \text{if } t'<1\,,\; t-t'<1 \\
\la x \ra^{-\sigma_2} & \text{if } t'>\frac12 \,,\; t-t'<1 \\
(t')^{-\nu} \la x \ra^{-\tilde \sigma_1} & \text{if } t'<1 \,,\; t-t'>\frac12  \\
\la x \ra^{-\tilde \sigma_2} & \text{if } t'>\frac12 \,, \; t-t'>\frac12.
\end{cases}
$$
(Note that the regions are overlapping for convenience in application.  The meaning is that $f(t,t')$ is bounded by the minimum of the bounds across all applicable regions.)  Then for any $0\leq a,b \leq \frac12$ with $0\leq a+b\leq t$, we have
\begin{equation}
\label{E:est3}
\int_b^{t-a} f(t,t') \, dt'\lesssim
\begin{cases}
(t^{-\nu}a^{-\mu+1}+t^{-\mu} b^{-\nu+1}) \la x \ra^{-\sigma_1} & \text{if } t<1 \\
b^{-\nu+1} \la x \ra^{-\tilde \sigma_1} + \la x \ra^{-\tilde \sigma_2} + a^{-\mu+1} \la x \ra^{-\sigma_2} & \text{if }t>1.
\end{cases}
\end{equation}
If, instead, $0\leq \nu<1$, then
\begin{equation}
\label{E:est4}
\int_0^{t-a} f(t,t') \, dt'\lesssim
\begin{cases}
t^{-\nu}a^{-\mu+1} \la x \ra^{-\sigma_1} & \text{if } t<1 \\
 \la x \ra^{-\tilde \sigma_2} + a^{-\mu+1} \la x \ra^{-\sigma_2} & \text{if }t>1.
\end{cases}
\end{equation}
\end{lemma}
\begin{proof}
First consider the case $\nu>1$.
If $t<1$,
$$
\int_b^{t-a} f(t,t') \, dt' \lesssim \la x \ra^{-\sigma_1} \int_b^{t-a} (t-t')^{-\mu} (t')^{-\nu} \, dt',
$$
and changing variable to $s= t'/t$, we continue
$$
= t^{-\mu-\nu+1}  \la x \ra^{-\sigma_1} \int_{b/t}^{1-a/t} (1-s)^{-\mu} s^{-\nu} \,ds \lesssim (t^{-\nu}a^{-\mu+1}+ t^{-\mu}b^{-\nu+1}) \la x \ra^{-\sigma_1}.
$$
If $t>1$, then we split into three pieces
$$
\int_b^{1/2} f(t,t') \, dt' \lesssim \la x \ra^{-\tilde \sigma_1} \int_b^{1/2} (t-t')^{-\mu} (t')^{-\nu} \,dt'  \lesssim b^{-\nu+1} \la x \ra^{-\tilde \sigma_1},
$$
$$
\int_{1/2}^{t-\frac12} f(t,t') \, dt' \lesssim \la x \ra^{-\tilde \sigma_2} \int_{1/2}^{t-\frac12} (t-t')^{-\mu} \,d t' \lesssim  \la x\ra^{-\tilde \sigma_2},
$$
$$
\int_{t-\frac12}^{t-a} f(t,t') \, dt' \lesssim \la x \ra^{-\sigma_2} \int_{t-\frac12}^{t-a} (t-t')^{-\mu} \,d t'  \lesssim a^{-\mu+1} \la x \ra^{-\sigma_2}.
$$
Now we turn to the case $0\leq \nu<1$. If $t<1$,
$$
\int_a^{t-a} f(t,t') \, dt' \lesssim \la x \ra^{-\sigma_1} \int_0^{t-a} (t-t')^{-\mu} (t')^{-\nu} \, dt',
$$
and changing variable to $s= t'/t$, we obtain
$$
= t^{-\mu-\nu+1}  \la x \ra^{-\sigma_1} \int_0^{1-a/t} (1-s)^{-\mu} s^{-\nu} \,ds \lesssim t^{-\nu} a^{-\mu+1}  \la x \ra^{-\sigma_1}.
$$
If $t>1$, then we split into three pieces
$$
\int_0^{1/2} f(t,t') \, dt' \lesssim \la x \ra^{-\tilde \sigma_1} \int_0^{1/2} (t-t')^{-\mu} (t')^{-\nu} \,dt'  \lesssim  \la x \ra^{-\tilde \sigma_1},
$$
$$
\int_{1/2}^{t-\frac12} f(t,t') \, dt' \lesssim \la x \ra^{-\tilde \sigma_2} \int_{1/2}^{t-\frac12} (t-t')^{-\mu} \,dt' \lesssim  \la x\ra^{-\tilde \sigma_2},
$$
$$
\int_{t-\frac12}^{t-a} f(t,t') \, dt' \lesssim \la x \ra^{-\sigma_2} \int_{t-\frac12}^{t-a} (t-t')^{-\mu} \,d t'  \lesssim a^{-\mu+1} \la x \ra^{-\sigma_2}.
$$
\end{proof}

\begin{proposition}[Duhamel estimate]
\label{P:Duhamel}
Suppose that $F(x,y,t)$ satisfies, for some $\sigma \gg 1$ and $\nu \geq 0$,
\begin{itemize}
\item
$\text{for }x>0 \,, \; | F(x,y,t)| \lesssim C_2
\begin{cases} t^{-\nu} \la x \ra^{-\sigma_1} & \text{if } t<1 \\
\la x \ra^{-\sigma_2} &\text{if }t>\frac12
\end{cases}$,
where $1<\sigma_2\leq \sigma_1$,

\item
$\| \la x \ra^{-1} F(x, y, t) \|_{L_t^\infty L_{y\in \mathbb{R}, x<0}^2} \lesssim C_2$,

\item
$\| \la x \ra^{-1} \partial_x F(x, y, t) \|_{L_t^\infty L_{y\in \mathbb{R}, x<0}^1} \lesssim C_2$,

\item
$\| \partial_x F(x, y, t) \|_{L_t^\infty L_{y\in \mathbb{R},x>0}^1} \lesssim C_2$.
\end{itemize}
Then, if $\nu>1$, we have for $x>0$,
$$
\begin{aligned}
\indentalign \left| \int_0^t [\partial_x S(t-t') F(\bullet, \bullet, t')](x,y) \, dt' \right| \\
&\lesssim C_2
\begin{cases}
t^{1/3} &  \text{if } 0< t^{\nu+\frac5{12}} \lesssim  \la x \ra^{-(\sigma_1-\frac94)}   \\
t^{-\frac45\nu} \la x \ra^{-\frac45(\sigma_1-\frac94)}&  \text{if } \la x \ra^{-(\sigma_1-\frac94)} \ll  t^{\nu+\frac5{12}} < 1  \text{ and } \nu <\frac54\\
t^{-\frac23-\frac{5}{12\nu}} \la x \ra^{-\frac1\nu(\sigma_1-\frac94)} &  \text{if } \la x \ra^{-(\sigma_1-\frac94)} \ll  t^{\nu+\frac5{12}} < 1  \text{ and } \nu >\frac54\\
\la x \ra^{-\min( \frac{\tilde \sigma_1}{\nu}, \tilde \sigma_2, \frac45(\sigma_2-\frac94))} & \text{if }t>1,
\end{cases}
\end{aligned}
$$
where
$$
\tilde \sigma_1 = \min( \sigma_1 - \tfrac94, \tfrac23\sigma_1-\tfrac5{12}),
$$
$$
\tilde \sigma_2 = \min( \sigma_2 - \tfrac94, \tfrac23\sigma_2-\tfrac5{12}).
$$
\end{proposition}

\begin{proof}  By linearity, we can take $C_2=1$.
Let
$$
f(x,y,t,t') = [\partial_x S(t-t')F(\cdot,\cdot,t')](x,y).
$$
By the assumed pointwise decay on $F(x,y,t)$ for $x>0$, and the assumption $\|F\|_{L_t^\infty L_{xy}^2} \lesssim 1$, Prop \ref{P:derivative-linear-decay} gives, for $x>0$, the pointwise estimate
$$
|f(x,y,t,t')| \lesssim
(t-t')^{-13/12}  \left\{ \begin{aligned}
&(t')^{-\nu} \la x \ra^{-\sigma_1+\frac94} & \text{if } t-t'<1, \; t'<1\\
&(t')^{-\nu} \la x \ra^{-\tilde \sigma_1} & \text{if }t-t'>\frac12, \; t'<1 \\
& \la x \ra^{-\sigma_2+\frac94} & \text{if }t-t'<1, \; t'>\frac12 \\
&\la x \ra^{-\tilde \sigma_2} & \text{if }t-t'>\frac12, \; t'>\frac12.
\end{aligned} \right.
$$
By Lemma \ref{L:time-integral}, with $\sigma_1$ replaced by $\sigma_1-\frac94$, $\sigma_2$ replaced by $\sigma_2-\frac94$, and $\tilde \sigma_1$, $\tilde \sigma_2$ as given above, we obtain, for $x>0$,
$$
\left| \int_b^{t-a} f(x,y,t,t') \,dt' \right| \lesssim
\begin{cases}
(t^{-\nu} a^{-1/12} + t^{-13/12}b^{-\nu+1}) \la x \ra^{-\sigma_1+\frac94} & \text{if } t<1 \\
b^{-\nu+1} \la x \ra^{-\tilde \sigma_1} + \la x \ra^{-\tilde \sigma_2} + a^{-1/12} \la x \ra^{-\sigma_2+\frac94} & \text{if } t>1.
\end{cases}
$$
We have
$$[S(t)\phi](x,y) = \iint A(x+t-x',y-y',t) \phi(x',y') \, dx' \, dy'$$
By \eqref{E:A-decay} with $\beta=0$
$$|[S(t)\phi](x,y)| \lesssim t^{-\frac23+\frac12\alpha} \iint |x+t-x'|^{-\frac32\alpha} |\phi(x',y')| \, dx' \, dy'$$
with the corresponding restrictions on $\alpha$.  Splitting the integration into $x'>-1$, where we use $\alpha=0$, and $x'<-1$, where we use $\alpha=1$, we obtain
$$|[S(t)\phi](x,y)| \lesssim \begin{aligned}[t]
&t^{-\frac23} \iint_{x'>-1} |\phi(x',y')| \, dx' \, dy' \\
&\qquad + t^{-\frac16} \iint_{x'<-1} \la x+t-x'\ra^{-3/2} |\phi(x',y')| \, dx' \, dy'
\end{aligned}
$$
In the second integral, since $x'<-1$, we have that $\la x+t-x' \ra^{-3/2} \leq \la t \ra^{-\frac12} \la x' \ra^{-1}$ and thus
\begin{align*}
\|S(t)\phi\|_{L_{xy}^\infty} &\lesssim t^{-2/3}( \|\phi\|_{L^1_{y\in \mathbb{R}, x>-1}} + \| \la x \ra^{-1} \phi\|_{L^1_{y\in \mathbb{R}, x<-1}}) \\
&\approx t^{-2/3}( \|\phi\|_{L^1_{y\in \mathbb{R}, x>0}} + \| \la x \ra^{-1} \phi\|_{L^1_{y\in \mathbb{R}, x<0}})
\end{align*}
Hence,
\begin{align*}
\left| \int_0^b f(x,y,t,t') \,dt' \right|
& \lesssim \int_0^b \| S(t-t') \partial_x F(\cdot,\cdot,t')(x,y) \|_{L_{xy}^\infty} \, dt' \\
&\lesssim (\| \partial_xF(x,y,t) \|_{L_t^\infty L_{y\in \mathbb{R},x>0}^1} + \| \la x \ra^{-1} \partial_xF(x,y,t) \|_{L_t^\infty L_{y\in \mathbb{R},x<0}^1}) \int_0^b (t-t')^{-2/3} \, dt' \\
&\lesssim t^{-2/3} b,
\end{align*}
where, in the last step, we assumed that $b\leq \frac12 t$.  Similarly,
$$
\left| \int_{t-a}^t f(x,y,t,t') \,dt' \right| \lesssim \int_{t-a}^t (t-t')^{-2/3} \,dt' \lesssim a^{1/3}.
$$

Now, if we take
$$
G(x,y,t) = \int_0^t f(x,y,t,t') \, dt',
$$
then the above estimates give, for $x>0$,
\begin{equation}
\label{E:G-bound1}
|G(x,y,t)| \lesssim t^{-2/3} b + a^{1/3} + \begin{cases}
(t^{-\nu} a^{-1/12} + t^{-13/12}b^{-\nu+1}) \la x \ra^{-\sigma_1+\frac94} & \text{if } t<1 \\
b^{-\nu+1} \la x \ra^{-\tilde \sigma_1} + \la x \ra^{-\tilde \sigma_2} + a^{-1/12} \la x \ra^{-\sigma_2+\frac94} & \text{if } t>1,
\end{cases}
\end{equation}
provided $a+b\leq t$ and $b\leq \frac12 t$.
\bigskip

\noindent\textbf{Case 1.} $t<1$ and $\la x \ra^{-\sigma_1+\frac94} \ll t^{\nu+\frac5{12}}$ (corresponding to $a\ll t$ and $b\ll t$, where $a$ and $b$ are defined below).   In this case, we obtain from the first component of \eqref{E:G-bound1}, for $x>0$,
$$
|G(x,y,t)| \lesssim t^{-\frac23}b + t^{-\frac{13}{12}} b^{-\nu+1} \la x \ra^{-\sigma_1+\frac94} + t^{-\nu} a^{-\frac{1}{12}} \la x \ra^{-\sigma_1+\frac94} + a^{\frac13}.
$$
The optimal values of $a$ and $b$ are $a = [t^{-\nu} \la x \ra^{-\sigma_1+\frac94}]^{12/5}$ and $b = [t^{-\frac{5}{12}} \la x \ra^{-\sigma_1+\frac94}]^{1/\nu}$.  This furnishes the bound
$$
|G(x,y,t)| \lesssim t^{- (\frac23 + \frac{5}{12 \nu})} \la x \ra^{-\frac{1}{\nu}(\sigma_1-\frac94)} + t^{-\frac{4\nu}{5}} \la x \ra^{-\frac45 (\sigma_1-\frac94)}.
$$
We consider two further subcases
\bigskip

\noindent\textbf{Case 1A.}  $\nu < \frac54$.  Then $\frac{1}{\nu}-\frac45>0$, so by raising $\la x \ra^{-(\sigma_1-\frac94)} < t^{\nu+\frac5{12}}$ to the positive power  $\frac{1}{\nu}-\frac45$, we obtain
$$
\la x \ra^{-(\sigma_1-\frac94)( \frac{1}{\nu} - \frac45)} < t^{(\nu+\frac5{12})(\frac{1}{\nu}-\frac45)} = t^{\frac23+\frac{5}{12\nu}-\frac45\nu}.
$$
Hence,
$$
t^{- (\frac23 + \frac{5}{12 \nu})} \la x \ra^{-\frac{1}{\nu}(\sigma_1-\frac94)} \leq t^{-\frac45\nu} \la x \ra^{-\frac45(\sigma_1-\frac94)}.
$$
Consequently, in this case, we have the bound
$$
|G(x,y,t)| \lesssim  t^{-\frac45\nu} \la x \ra^{-\frac45(\sigma_1-\frac94)}.
$$
\bigskip

\noindent\textbf{Case 1B.} $\nu>\frac54$.  Then $\frac45-\frac{1}{\nu}>0$, so by raising $\la x \ra^{-(\sigma_1-\frac94)} < t^{\nu+\frac5{12}}$ to the positive power  $\frac45-\frac{1}{\nu}$, we obtain
$$
\la x \ra^{-(\sigma_1-\frac94)( \frac45 - \frac{1}{\nu})} < t^{(\nu+\frac5{12})(\frac45-\frac{1}{\nu})} = t^{-\frac23-\frac{5}{12\nu}+\frac45\nu}.
$$
Hence,
$$
t^{-\frac45\nu} \la x \ra^{-\frac45(\sigma_1-\frac94)} < t^{-\frac23-\frac{5}{12\nu}} \la x \ra^{-\frac{1}{\nu}(\sigma_1-\frac94)}.
$$
Consequently, in this case, we have the bound
$$
|G(x,y,t)| \lesssim t^{-\frac23-\frac{5}{12\nu}} \la x \ra^{-\frac{1}{\nu}(\sigma_1-\frac94)}.
$$
We also remark that in this case ($\nu>\frac54$ case), it follows that $\frac23+ \frac{5}{12\nu}<1$, so
$$
t^{-\frac23-\frac{5}{12\nu}}\la x \ra^{-\frac{1}{\nu}(\sigma_1-\frac94)} < t^{-1} \la x\ra^{-\frac{1}{\nu}(\sigma_1-\frac94)}.
$$
\bigskip

\noindent\textbf{Case 2.} $t<1$ and $ t^{\nu+\frac5{12}}\lesssim  \la x \ra^{-\sigma_1+\frac94}$.  In this case, we just use
$$
|G(x,y,t)|\lesssim t^{1/3}.
$$
\bigskip

\noindent\textbf{Case 3.}  $t>1$.  In this case, we apply the second component of \eqref{E:G-bound1} to obtain
$$
|G(x,y,t)| \lesssim b+ b^{-\nu+1}\la x \ra^{-\tilde \sigma_1} + \la x \ra^{-\tilde \sigma_2} + a^{-\frac{1}{12}} \la x \ra^{-\sigma_2+\frac94}+a^{\frac13}.
$$
In this case, the optimal choices of $a$ and $b$ are $b = \la x \ra^{-\frac{\tilde \sigma_1}{\nu}}$ and $a = \la x \ra^{-\frac{12}{5} (\sigma_2-\frac94)}$.
\end{proof}

In the nonlinear argument in the next section, we will need the following consequence of Prop. \ref{P:Duhamel}, which we state as a corollary.

\begin{corollary}[Duhamel estimate]
\label{C:Duhamel}
Suppose that $F(x,y,t)$ satisfies, for some $\sigma \gg 1$ and $\nu \geq 0$,
\begin{itemize}
\item
$\text{for }x>0 \,, \; | F(x,y,t)| \lesssim C_2
\begin{cases} t^{-\nu} \la x \ra^{-\sigma_1} & \text{if } t<1 \\
\la x \ra^{-\sigma_2} &\text{if }t>1
\end{cases}$,
where $1<\sigma_2\leq \sigma_1$,
\item $\| \la x \ra^{-1} F(x, y, t) \|_{L_t^\infty L_{y\in \mathbb{R}, x<0}^2} \lesssim C_2$,
\item $\| \la x \ra^{-1} \partial_x F(x, y, t) \|_{L_t^\infty L_{y\in \mathbb{R}, x<0}^1} \lesssim C_2$,
\item $\| \partial_x F(x, y, t) \|_{L_t^\infty L_{y\in \mathbb{R},x>0}^1} \lesssim C_2$.
\end{itemize}
Then, if $\nu>\frac{5}{4}$ and $r\geq 0$, we have for $x>0$
\begin{equation}
\label{E:Duhamel1}
\left| \int_0^t [\partial_x S(t-t') F(\bullet, \bullet, t')](x,y) \, dt' \right| \lesssim C_2
\begin{cases}
t^{-\nu/3} \la x \ra^{-\sigma_1/3}  & \text{if }0 \leq t \leq 1 \\
\la x \ra^{-\frac13\sigma_2 - r}  & \text{if }t>\frac12,
\end{cases}
 \end{equation}
provided $\sigma_j\geq \frac{11}{2}$ for $j=1,2$ and
\begin{equation}
\label{E:imp4mod}
\sigma_2 \geq \max(\frac{27}{7}+\frac{15}{7}r, \frac{5}{4}+3r)
\end{equation}
and
\begin{equation}
\label{E:imp5mod}
\sigma_1 \geq \frac{27}{7}(\nu+1) \text{ and } \sigma_1 \geq \frac{\nu}{2}\sigma_2+\frac{5}{8}+ \frac32\nu r.
\end{equation}
\end{corollary}

\begin{proof}
By Prop. \ref{P:Duhamel}, the first line of \eqref{E:Duhamel1} will hold provided
\begin{equation}
\label{E:imp1}
t^{\nu+\frac5{12}} \lesssim \la x \ra^{-(\sigma_1-\frac94)} \implies t^{1/3} \lesssim t^{-\nu/3}\la x \ra^{-\sigma_1/3}
\end{equation}
\begin{equation}
\label{E:imp2}
\la x\ra^{-(\sigma_1-\frac94)} \ll t^{\nu+\frac5{12}}<1 \implies t^{-\frac23-\frac{5}{12\nu}} \la x \ra^{-\frac1{\nu}(\sigma_1-\frac94)} \lesssim t^{-\nu/3} \la x\ra^{-\sigma_1/3}.
\end{equation}
For $0\leq t \leq 1$, we will show that \eqref{E:imp1} and \eqref{E:imp2} hold provided the left inequality of\eqref{E:imp5mod} holds.

Regarding \eqref{E:imp1}, the right side of the implication can be reexpressed as follows:
$$
\text{RHS} \iff t^{\nu+1} \lesssim \la x \ra^{-\sigma_1} \iff (t^{\nu+1})^\frac{\sigma_1-\frac94}{\sigma_1} \lesssim \la x \ra^{-(\sigma_1-\frac94)}.
$$
Thus, the implication in \eqref{E:imp1} will be true if
$$
(t^{\nu+1})^{\frac{\sigma_1-\frac94}{\sigma_1}} \lesssim t^{\nu+\frac5{12}}.
$$
We can reexpress this as $t^{(\nu+1)(\sigma_1-\frac{9}{4})} \lesssim t^{(\nu+\frac{5}{12}) \sigma_1}$.  Since $0\leq t\leq 1$, this is equivalent to
$$
(\nu+1)(\sigma_1-\tfrac94) \geq (\nu+\tfrac5{12})\sigma_1.
$$
With some algebra, this reduces to the left inequality of \eqref{E:imp5mod}.

Regarding \eqref{E:imp2}, the right side of the implication can be reexpressed as follows:
$$
\la x \ra^{ - [ (1-\frac{\nu}{3}) \sigma_1 - \frac94]} \lesssim t^{\frac23\nu - \frac13 \nu^2 + \frac{5}{12}},
$$
and also equivalently, by exponentiating
$$
\la  x \ra^{-(\sigma_1 - \frac{9}{4})}  \lesssim t^\mu \,, \quad \mu = \frac{ (-\frac{\nu^2}{3}+\frac{2\nu}{3}+\frac{5}{12}) (\sigma_1- \frac{9}{4})}{(1- \frac{\nu}{3})\sigma_1 - \frac{9}{4}},
$$
where we have assumed that $(1- \frac{\nu}{3}) \sigma_1 - \frac{9}{4} >0$.  Thus, the implication in \eqref{E:imp2} is true provided that
$$
t^{\nu+\frac5{12}} \lesssim t^\mu.
$$
Since $0\leq t\leq 1$, this is equivalent to
$$
\nu + \frac{5}{12} \geq \mu,
$$
which we reexpress as
$$
(\nu + \frac{5}{12})( (1-\frac{\nu}{3}) \sigma_1 - \frac{9}{4}) \geq ( - \frac{\nu^2}{3}+\frac{2\nu}{3}+\frac{5}{12})(\sigma_1-\frac{9}{4}).
$$
Some algebra reduces this to the condition on the left in \eqref{E:imp5mod}.  Thus, we have established that the left inequality in \eqref{E:imp5mod} suffices to imply \eqref{E:imp1}, \eqref{E:imp2}, from which it follows from Prop. \ref{P:Duhamel} that the first line of \eqref{E:Duhamel1} holds.

By Prop \ref{P:Duhamel},  we have the second line of \eqref{E:Duhamel1} holds provided
\begin{equation}
\label{E:imp3}
\min(\frac{\tilde \sigma_1}{\nu}, \tilde \sigma_2, \frac45(\sigma_2-\frac{9}{4})) \geq \frac{\sigma_2}{3}+r
\end{equation}
holds, where
$$
\tilde \sigma_j = \min( \sigma_j - \tfrac94, \tfrac23\sigma_j-\tfrac5{12}).
$$
Since we assume that $\sigma_j\geq \frac{11}{2}$ for $j=1,2$ we have $\tilde \sigma_j = \tfrac23\sigma_j-\frac{5}{12}$.
We observe that \eqref{E:imp3} holds provided \eqref{E:imp4mod} and the second inequality of \eqref{E:imp5mod} holds.
\end{proof}

\section{Nonlinear estimate}\label{S-12}

Now we return to the problem of estimating $\eta$.  Recall that $\eta$ solves
\begin{equation}
\label{E:eta10b}
\partial_t \eta - \partial_{x} [(-\Delta  +  x_t)\eta] = F \,, \qquad F= f_1 + \partial_{x} f_2,
\end{equation}
where
\begin{equation}
\label{E:eta11b}
\begin{aligned}
& f_1 = - (\lambda^{-1})_t \partial_{\lambda^{-1}} \tilde Q  \\
& f_2 = + (x_t -1) \tilde Q -3\tilde Q^2 \eta - 3 \tilde Q \eta^2 - \eta^3.
\end{aligned}
\end{equation}
Here, $\tilde Q(x,y) = \lambda^{-1} Q(\lambda^{-1} (x + K), \lambda^{-1} y)$.  Note that since $|Q(y_1,y_2)| \lesssim \la \vec y \ra^{-1/2} e^{-|\vec y|}$, we have $|\tilde Q(x,y)| \leq e^{-K/2}$ for $x>0$ (see Remark \ref{R:K}).   We know that for all $t>0$,
$$
\| \eta(t) \|_{H_{xy}^1} \lesssim \delta
$$
and
$$
|\lambda(t) -1 |\lesssim \delta, \qquad |\lambda_t| \lesssim \delta, \qquad |x_t-1| \lesssim \delta \quad \mbox{or} \quad (1-\delta) t \lesssim x(t) \lesssim (1+\delta)t.
$$
Furthermore, we know that $\phi(x,y) = \eta(0,x,y)$ satisfies, for $x>0$, $y\in \mathbb{R}$,
$$
|\phi(x,y)| \leq \delta \la x \ra^{-\sigma}.
$$
Finally, we know that for any $T>0$, $\eta$ is the \emph{unique} solution of \eqref{E:eta10b} in $C([0,T];H_{xy}^1)$ such that $\eta(t, x+x(t),y) \in L_x^4L_{yT}^\infty.$
Our goal is to show that for $x>0$ and $y\in \mathbb{R}$,
\begin{equation}
\label{E:etaboundsb}
| \eta(t,x,y)| \lesssim \delta
\begin{cases} t^{-7/12} \la x \ra^{-\sigma+\frac74} & \text{if }0 \leq t \leq 1 \\
\la x \ra^{-\frac23\sigma+\frac34} & \text{if } t\geq 1.
\end{cases}
\end{equation}

\begin{proposition}
\label{P:decay}
There exists $\delta_0>0$ (small), $K>0$ (large), and $\sigma_0 >0$  (large) such that the following holds true.   Suppose that $\sigma\geq \sigma_0$, $0<\delta \leq \delta_0$, $\phi \in H^1$ with $\|\phi\|_{H^1} \leq \delta$, and
$$
\text{for }x>0 \,, \qquad |\phi(x,y)| \leq \delta \la x \ra^{-\sigma}.
$$
Then the unique solution $\eta(t,x,y)$ solving \eqref{E:eta10b} for all $t$ satisfies
\begin{equation}
\label{E:targetdecay}
\text{for }x>0 \,, \qquad |\eta(t,x,y) | \lesssim \delta
\begin{cases}
t^{-7/12} \la x \ra^{-\sigma+\frac74} & \text{if }0<t\leq 1\\
\la x \ra^{-\frac23\sigma+\frac34} & \text{if }t\geq 1.
\end{cases}
\end{equation}
\end{proposition}

The proof consists of the following steps.  The following lemma provides a key short-time step result.

\begin{lemma}
\label{L:decay1}
There exists $\sigma_0>0$ (large), $K>0$ (large), and $\delta_0>0$ (small) such that if $\sigma\geq \sigma_0$, $0< \delta\leq \delta_0$, then the unique solution $\eta(t,x,y)$ solving \eqref{E:eta10b} satisfies, for all $0\leq t\leq 1$,
\begin{equation}
\label{E:dec1}
\text{for }x>0 \,, \qquad |\eta(t,x,y) | \lesssim \delta \,  t^{-7/12} \la x \ra^{-\sigma+\frac74}.
\end{equation}
\end{lemma}

\begin{proof}
This is done using a contraction argument and the available decay estimates, and the Duhamel estimate Corollary \ref{C:Duhamel}, as follows.  Take $T=1$ and define the $Y$ norm as follows
$$
\|\eta\|_Y = \|\eta(t,x,y)\|_{L_T^\infty H_{xy}^1} + \|\eta(t, x+x(t),y)\|_{L_x^4 L_{yT}^\infty} + \|\eta(t,x,y) t^{7/12} \la x \ra^{\sigma-\frac74} \|_{L_{Ty, x>0}^\infty}.
$$
Let $\Lambda$ be defined on $Y$ by
$$
\Lambda \eta = S(t,0) \phi + \int_0^t  S(t,t') f_1(\bullet,\bullet,t') \, dt'+ \int_0^t \partial_x S(t,t') f_2(\bullet,\bullet,t') \, dt'.
$$
By Proposition \ref{P:linear-decay}  with $C_1=\delta$, we obtain that for $x>0$ and $y\in \mathbb{R}$, $0\leq t \leq 1$,
\begin{equation}
\label{E:eta20}
|[S(t,0)\phi](x,y)| \leq \tfrac14 C_3 \delta t^{-7/12} \la x \ra^{-\sigma+\frac74}
\end{equation}
for some absolute constant $C_3$ (which for convenience in writing below, we will take $\geq 1$).

Since $f_1 = \lambda^{-3}\lambda_t (Q + (x+K)Q_x + yQ_y)$, where $Q$, $Q_x$, and $Q_y$ are evaluated at $(\lambda^{-1}(x+K),\lambda^{-1}y)$, we have that  $|f_1(x,y,t)|  \lesssim \delta \delta_0 \la x \ra^{-\sigma}$ for all $t$ and $x>0$.  Thus by Prop. \ref{P:linear-decay}, for $x>0$,
$$\left| \int_0^t S(t,t') f_1(\bullet,\bullet,t') \, dt' \right| \lesssim \delta \delta_0 \la x \ra^{-\sigma+\frac74}\leq \tfrac14 C_3 \delta t^{-7/12} \la x \ra^{-\sigma+\frac74}$$

Suppose that $\| \eta\|_Y \leq C_3\delta$.  Then, requiring $K$ large enough so that $\la K \ra^{-1} \leq \delta_0$ (which also implies that $e^{-K/2} \leq \delta_0$, see our Remark \ref{R:K}), we have for $x>0$, $y\in \mathbb{R}$,
$$
|f_2(t,x,y)| \leq |x_t-1| \tilde Q + 3|\eta| \tilde Q^2+ 3|\eta|^2\tilde Q + |\eta|^3 \lesssim  C_3^3\delta_0 \delta \, t^{-7/4} \la x \ra^{-3(\sigma-\frac74)}.
$$
Moreover, by Sobolev
$$
\| \la x\ra^{-1} \partial_x f_2 \|_{L_T^\infty L_{y\in \mathbb{R}, x<0}^1} +\| \partial_x f_2 \|_{L_T^\infty L_{y\in \mathbb{R}, x>0}^1} + \| \la x \ra^{-1} f_2 \|_{L_T^\infty L_{y\in \mathbb{R}, x<0}^2} \lesssim C_3^3 \delta_0 \delta .
$$
Thus, in the hypothesis of Corollary \ref{C:Duhamel}, we can take $C_2=C_3^3\delta_0\delta$ and $\sigma_1=3(\sigma-\frac74)$, and conclude that for $x>0$ and $y \in \mathbb{R}$,
\begin{equation}
\label{E:eta21}
\left| \int_0^t [\partial_x S(t,t') f_2(\bullet, \bullet, t')](x,y) \, dt' \right| \leq C_4 C_3^3\delta_0 \delta \, t^{-7/12} \la x \ra^{-\sigma+\frac74}
\end{equation}
for some absolute constant $C_4>0$.  Taking $\delta_0>0$ sufficiently small so that $C_4C_3^2 \delta_0 \leq \frac12$, we obtain from \eqref{E:eta20} and \eqref{E:eta21} that for $x>0$ and $y\in \mathbb{R}$, $0\leq t\leq 1$,
\begin{equation}
\label{E:eta22}
|(\Lambda\eta)(t,x,y)| \leq \frac12 \,C_3 \delta t^{-7/12} \la x \ra^{-\sigma+\frac74}.
\end{equation}
Moreover, by the estimates in Lemmas \ref{L:linhom} and \ref{L:lininhom}, if $\|\eta\|_{L_T^\infty H_{xy}^1} \leq \frac14 \,C_3 \delta$ and $\|\eta(t,x+x(t),y)\|_{L_x^4L_{yT}^\infty} \leq \frac14 \, C_3 \delta$, then
\begin{equation}
\label{E:eta23}
\|\Lambda \eta \|_{L_T^\infty H_{xy}^1} + \| \Lambda \eta(t,x+x(t),y) \|_{L_x^4 L_{yT}^\infty} \leq \frac12 \, C_3 \delta
\end{equation}
as in the discussion following Theorem \ref{T:lwp-review} reviewing the local well-posedness (with a possible adjust to $C_3$ and $\delta_0$, as required by the absolute constants in those estimates).  Combining \eqref{E:eta22} and \eqref{E:eta23}, we obtain
\begin{equation}
\label{E:eta24}
\| \Lambda \eta \|_Y \leq C_3 \delta.
\end{equation}
Moreover, it also follows similarly from these estimates that
\begin{equation}
\label{E:eta25}
\| \Lambda \eta_2 - \Lambda \eta_1 \|_Y \leq \frac12 \|\eta_2-\eta_1\|_Y
\end{equation}
for two $\eta_1, \eta_2 \in Y$ such that $\eta_1(0,x,y)=\eta_2(0,x,y)=\phi$. Hence, $\Lambda$ is a contraction and the fixed point solves \eqref{E:eta10b}.  By the uniqueness in Theorem \ref{T:lwp-review}, this fixed point is the unique solution in the function class stated in that Theorem.
\end{proof}

Now the proof proceeds as follows:
\begin{itemize}
\item
Let $T_*\geq 0$ be the sup of all times for which \eqref{E:targetdecay} holds.
\item
By Lemma \ref{L:decay1}, $T_*\geq 1$.
\item
If $T_*<\infty$, then we will obtain a contradiction in the following series of steps. First, we know that at $T_1\defeq T_*-\frac12$,
$$
\text{for }x>0 \,, \qquad |\eta(T_1,x,y) | \lesssim \delta  \la x \ra^{-\frac23\sigma+\frac34}.
$$
\item
Apply Lemma \ref{L:decay1} with $t=0$ replaced by $t=T_1$ to obtain that $\eta$ satisfies, for all $T_*-\frac12\leq t \leq T_*+\frac12$, the estimate
$$
\text{for }x>0 \,, \qquad |\eta(t,x,y) | \lesssim \delta (t-T_1)^{-7/12} \la x \ra^{-\frac23\sigma+\frac52}.
$$
Restricting to $T_*\leq t \leq T_*+\frac12$, this is simplifies to
$$
\text{for }x>0 \,, \qquad |\eta(t,x,y) | \lesssim \delta  \la x \ra^{-\frac23\sigma+\frac52}.
$$
\item
Now we know that for $0\leq t\leq T_*+\frac12$,
\begin{equation}
\label{E:targetdecayinput}
\text{for }x>0 \,, \qquad |\eta(t,x,y) | \lesssim \delta
\begin{cases} t^{-7/12} \la x \ra^{-\sigma+\frac74} \\
\la x \ra^{-\frac23\sigma+\frac52}
\end{cases}
\end{equation}
holds, which is slightly weaker than \eqref{E:targetdecay}.
\item
We know that, on $0 \leq t \leq T_*+\frac12$, $\eta$ satisfies
$$
\eta = S(t,0) \phi + \int_0^t  S(t,t') f_1(\bullet,\bullet,t') \, dt'+ \int_0^t \partial_x S(t,t') f_2(\bullet,\bullet,t') \, dt'.
$$
Apply the estimates in Proposition \ref{P:linear-decay} and Corollary \ref{C:Duhamel} to show that \eqref{E:targetdecayinput} suffices to conclude \eqref{E:targetdecay} holds on $0\leq t\leq T_*+\frac12$, which is a contradiction to the definition of $T_*$.  Indeed, we apply Corollary \ref{C:Duhamel} with $\nu=\frac74$, $\sigma_1= 3(\sigma-\frac74)$, $\sigma_2 = 3(\frac23\sigma-\frac52)$, and $r=\frac74$.  Then $\frac13\sigma_2+r=\frac23\sigma-\frac34$, so that \eqref{E:targetdecay} is obtained.
\end{itemize}

\section{$H^1$-instability of $Q$ for the critical gZK}\label{S-13}

We are now ready to prove our main result, Theorem \ref{Theo-Inst}.

\begin{proof}[Proof of Theorem \ref{Theo-Inst}]
For $n\in \mathbb{N}$ to be chosen later, let
$$
u^n_0=Q+\ep^n_0,
$$
where
\begin{equation}\label{ep_0}
\ep^n_0=\frac{1}{n}\left(Q+a\chi_0\right),
\end{equation}
and $a\in \R$ is such that $\ep^n_0 \perp \chi_0$, that is,
$$a=\displaystyle -\frac{\int \chi_0 Q}{\|\chi_0\|_2^2}.$$
From Theorem \ref{L-prop}, we have that for every $n\in \mathbb{N}$
$$
\ep^n_0 \perp \{Q_{y_1}, Q_{y_2}, \chi_0\}.
$$
Denote by $u^n(t)$ the solution of \eqref{gZK} associated to $u^n_0$.

Assume by contradiction that $Q$ is stable. Then, for  $\alpha_0<\overline{\alpha}$, where $\overline{\alpha}>0$ is given by Proposition \ref{ModThI}, if $n$ is sufficiently large, we have $u^n(t)\in U_{\alpha_0}$ (recall \eqref{tube}). Thus, from Definition \ref{eps}, there exist functions $\lambda^n(t)$ and $x^n(t)$ such that $\ep^n(t)$, defined in \eqref{eq-ep2}, satisfies
\begin{equation*}
\ep^n(t) \perp \{Q_{y_1}, Q_{y_j}, \chi_0\},
\end{equation*}
and also $\lambda^n(0)=1$ and $x^n(0)=0$.

To simplify the notation we drop the index $n$ in what follows. Rescaling the time $t \mapsto s$ by  $\frac{ds}{dt} = \frac1{\lam^3}$ and taking $\alpha_0<\alpha_1$, where $\alpha_1>0$ is given by Lemma \ref{Lemma-param}, we have that $\lambda(s)$ and $x(s)$ are $C^1$ functions, and that $\ep(s)$ satisfies equation \eqref{Eq-ep}. Moreover, from Proposition \ref{ModThI}, since $u(t)\in U_{\alpha_0}$,  we have
\begin{equation}\label{Bound_ep}
\|\ep(s)\|_{H^1}\leq C_1 \alpha_0 \quad \textrm{and} \quad |\lambda(s)-1|\leq C_1\alpha_0,
\end{equation}
thus, taking $\alpha_0<(2C_1)^{-1}$, we obtain
\begin{equation}\label{Bound}
\|\ep(s)\|_{H^1}\leq 1 \quad \textrm{and} \quad \frac12\leq \lambda(s)\leq \frac32, \quad \textrm{for all} \quad s\geq 0.
\end{equation}

Furthermore, in view of \eqref{ControlParam}, if $\alpha_0>0$ is small enough, we deduce
$$
\left|\frac{\lambda_s}{\lambda}\right|+\left|\frac{x_s}{\lambda}-1\right|\leq C_2\|\ep(s)\|_2\leq C_1C_2\alpha_0.
$$
Since $ x_t=x_s/\lambda^3$, we conclude that
$$
\frac{1-C_1C_2\alpha_0}{(1+C_1\alpha_0)^2}\leq \frac{1-C_1C_2\alpha_0}{\lambda^2}\leq x_t\leq \frac{1+C_1C_2\alpha_0}{\lambda^2}\leq \frac{1+C_1C_2\alpha_0}{(1-C_1\alpha_0)^2}
$$
Hence, we can choose $\alpha_0>0$, small enough, such that
$$
\frac{3}{4}\leq x_t \leq \frac{5}{4}.
$$
The last inequality implies that $x(t)$ is increasing and by the Mean Value Theorem
$$
x(t_0)-x(t)\geq \frac{3}{4}(t_0-t)
$$
for every $t_0, t\geq 0$ with $t\in [0,t_0]$.
Also, recalling $x(0)=0$, another application of the Mean Value Theorem yields
$$
x(t)\geq \frac{1}{2}t
$$
for all $t\geq 0$. Finally, by assumption \eqref{ep_0} and properties of $Q$, we have
$$
|u_0(\vec{x})|\leq ce^{-{\delta}|\vec{x}|},
$$
for some $c>0$ and ${\delta}>0$.

From the monotonicity properties in Section \ref{S-7}, we obtain the $L^2$ exponential decay on the right for $\ep(s)$.
\begin{corollary}\label{AM3}
Let $M\geq 4$. If $\alpha_0>0$ is sufficiently small, then there exists $C=C(M,\delta)>0$ such that for every $s\geq 0$ and $y_0>0$
$$
\int_{\R}\int_{y_1>y_0}\ep^2(s, y_1, y_2)dy_1dy_2\leq Ce^{-\frac{y_0}{2M}}.
$$
\end{corollary}
\begin{proof}
Applying Lemma \ref{AM2}, for a fixed $M\geq 4$, there exists $C=C(M)>0$ such that for all $t\geq 0$ and $x_0>0$ we have
$$
\int_{\R}\int_{x_1>x_0}u^2(t, x_1+x(t), x_2)dx_1dx_2\leq Ce^{-\frac{x_0}{M}}.
$$
From the definition of $\ep(s)$, we have that
$$
\frac{1}{\lambda(s)}\ep\left(s,\frac{y_1}{\lambda(s)}, \frac{y_2}{\lambda(s)}\right)=u(s,y_1+x(s), y_2)-\frac{1}{\lambda(s)}Q\left(\frac{y_1}{\lambda(s)}, \frac{y_2}{\lambda(s)}\right).
$$
Moreover, if $\alpha_0<(2C_1)^{-1}$, we have $1/2\leq \lambda(s)\leq 3/2$, and using \eqref{Q-decay}, we get
\begin{equation}\label{Q-decay}
\frac{1}{\lambda(s)}Q\left(\frac{y_1}{\lambda(s)}, \frac{y_2}{\lambda(s)}\right)\leq \frac{c}{\lambda(s)}e^{-\frac{\left| \vec{y}\right|}{\lambda(s)}}\leq 2c \, e^{-\frac{2}{3} |\vec{y}|}\leq c \, e^{-\frac{|\vec{y}|}{M}},
\end{equation}
since $M\geq 3/2$.

Therefore, we deduce that
\begin{align*}
\int_{\R}\int_{y_1>y_0}\frac{1}{\lambda^2(s)}\ep^2\left(s,\frac{y_1}{\lambda(s)}, \frac{y_2}{\lambda(s)}\right)dy_1dy_2\leq &2\int_{\R}\int_{y_1>y_0}u^2(s,y_1+x(s),y_2)dy_1dy_2\\
&+2\int_{\R}\int_{y_1>y_0}\frac{1}{\lambda^2(s)}Q^2\left(\frac{y_1}{\lambda(s)}, \frac{y_2}{\lambda(s)}\right)dy_1dy_2\\
\leq & 2ce^{-\frac{y_0}{M}}+2c\int_{\R}\int_{y_1>y_0}e^{-\frac{|\vec{y}|}{M}}dy\\
\leq & {C}e^{-\frac{y_0}{M}}
\end{align*}
for some ${C}={C}(M)>0$.

Finally, by the scaling invariance of the $L^2$-norm, we get
$$
\int_{\R}\int_{y>y_0}\ep^2(s, y_1,y_2)dy_1dy_2 =\int_{\R} \int_{y>\lambda(s)y_0}\frac{1}{\lambda^2(s)}\ep^2\left(s,\frac{y_1}{\lambda(s)},\frac{y_2}{\lambda(s)}\right)\, dy_1dy_2
\leq C\, e^{-\frac{\lambda(s)y_0}{M}}
\leq C\, e^{-\frac{y_0}{2M}},
$$
since $\lambda(s)\geq 1/2$.
\end{proof}

Next, we define a rescaled and shifted quantity of the virial-type.
Recall the definition of $J_A$ in \eqref{def-JA} and let
$$
K_A(s)=\lambda(s)(J_A(s)-\kappa).
$$
(We remark that this quantity is similar to the corresponding one in Martel-Merle \cite{MM-KdV-instability}.)

From \eqref{Bound-JA} and \eqref{Bound}, it is clear that
\begin{equation}\label{Bound-KA}
|K_A(s)|\leq c\left((1+A^{1/2})\|\ep(s)\|_{2}+\kappa\right)<+\infty,
\end{equation}
for all $s\geq 0$.

Moreover, using Lemma \ref{J'_A}, we also have
\begin{align}\label{K'_A}
\frac{d}{ds}K_A =& \lambda_s\left(J_A-\kappa\right)+ \lambda \frac{d}{ds}J_A \nonumber\\
=& \lambda \left(\frac{d}{ds}J_A+\frac{\lambda_s}{\lambda}\left(J_A-\kappa\right)\right) \nonumber \\
=&\lambda \left(2\left(1-\frac12\left(\frac{x_s}{\lambda}-1\right)\right)\int\ep Q +R(\ep,A)\right).
\end{align}

In the next result we obtain a strictly positive lower bound for $\frac{d}{ds}K_A(s)$ for a certain choice of $\alpha_0>0$, $n\in \mathbb{N}$ and $A\geq 1$.

\begin{theorem}\label{Lemma-K'_A}
There exist $\alpha_0>0$ sufficiently small, $n_0\in \mathbb{N}$ and $A\geq 1$ sufficiently large such that
\begin{equation}\label{Bound-K'_A}
\frac{d}{ds}K_A(s)\geq \frac{b}{2n_0} >0, \quad \textrm{for all} \quad s\geq 1,
\end{equation}
where
$$
b=\int (Q+a\chi_0)Q=\|Q\|_2^2-\frac{\left(\int Q\chi_0\right)^2}{\|\chi_0\|_2^2}.
$$
\end{theorem}

\begin{remark}
Note that $b>0$, since $Q\notin \textrm{span}\left\{\chi_0\right\}$.
\end{remark}

\begin{proof}
In view of \eqref{Bound_ep}, let $\alpha_0<\min\{\alpha_1(C_1)^{-1}, \alpha_2(C_1)^{-1}, (2C_1)^{-1}, 1/2\}$ so that we can apply Lemmas \ref{Lemma-param} and \ref{H^1-control}. From \eqref{K'_A} and the definition of $M_0$ (see \eqref{M_0}), we have
\begin{equation}\label{K'AM_0}
\frac{d}{ds}K_A(s) = \lambda \left(2\left(1-\frac12\left(\frac{x_s}{\lambda}-1\right)\right)M_0 +\widetilde{R}(\ep,A)\right),
\end{equation}
where $\widetilde{R}(\ep,A)={R}(\ep,A)- \left(1-\frac12\left(\frac{x_s}{\lambda}-1\right)\right) \int \ep^2$.

Since $\alpha_0<(2C_1)^{-1}$, we have $1/2\leq \lambda(s)\leq 3/2$, and using \eqref{ControlParam2}, we obtain
$$
\lambda \left(1-\frac12\left(\frac{x_s}{\lambda}-1\right)\right)\geq \frac12 \cdot  \frac12=\frac14.
$$
Moreover, from the definition of $M_0$, we also get
$$
M_0=2\int \ep_0 Q +\int \ep_0^2\geq 2\int \ep_0Q=\frac{2b}{n}.
$$
Therefore,
\begin{equation}\label{M_0term}
2\lambda \left(1-\frac12\left(\frac{x_s}{\lambda}-1\right)\right)M_0\geq \frac{b}{n}.
\end{equation}
On the other hand, by Lemma \ref{Lemma-param}, we have
$$
\left|\frac{\lambda_s}{\lambda}\right|+\left|\frac{x_s}{\lambda}-1\right|\leq C_2\|\ep(s)\|_2.
$$
Therefore, using the inequalities \eqref{R} and \eqref{Bound},
there exists a universal constant $C_6>0$, such that for $A\geq 1$ we have
\begin{equation}\label{Rtilda}
\lambda \widetilde{R}(\ep, A)\leq C_6\|\ep(s)\|_2\left(\|\ep(s)\|_2+A^{-1/2}+A^{1/2}\|\ep(s)\|_{L^2(y_1\geq A)}+\left|\int_{\R^2}y_2F_{y_2}\ep\varphi_A\right|\right).
\end{equation}
Moreover, by Lemma \ref{H^1-control}, we deduce
$$
\|\ep(s)\|^2_{H^1}\leq C_5\left(C_1\alpha_0 \left|\int \ep_0 Q\right|+\|\ep_0\|^2_{H^1}\right),
$$
and thus, the assumption \eqref{ep_0} yields
\begin{align}\label{ep(s)2}
\|\ep(s)\|^2_{H^1}\leq& C_5\left(C_1\alpha_0 \left(\frac{b}{n}\right)+\frac{d}{n^2} \right)\nonumber \\
\leq & C_5\left(C_1+\frac{d}{b}\right) \left(\alpha_0+\frac{1}{n}\right)\left(\frac{b}{n}\right),
\end{align}
where $d=\|Q+a\chi_0\|_{H^1}$.

Set $C_7= C_5\left(C_1+\frac{d}{b}\right) $. Collecting \eqref{Rtilda}-\eqref{ep(s)2}, we obtain
\begin{align*}
\lambda \widetilde{R}(\ep, A)\leq & C_7C_6\left(\alpha_0+\frac{1}{n}\right)\left(\frac{b}{n}\right)+\\
&+ \sqrt{C_7}C_6\left(A^{-1/2}+A^{1/2}\|\ep(s)\|_{L^2(y_1\geq A)}\right)\left(\alpha_0+\frac{1}{n}\right)^{1/2}\left(\frac{b}{n}\right)^{1/2}\\
&+\sqrt{C_7}C_6\left|\int_{\R^2}y_2F_{y_2}\ep\varphi_A\right| \left(\alpha_0+\frac{1}{n}\right)^{1/2}\left(\frac{b}{n}\right)^{1/2}.
\end{align*}
Let $K\geq 1$ satisfy \eqref{Kdelta0} and split the integral on the right hand side of the last inequality into two parts
$$
\int_{\R^2}y_2F_{y_2}\ep\varphi_A=\int_{\R}\int_{y_1<K}y_2F_{y_2}\ep\varphi_A \, dy_1dy_2
+\int_{\R}\int_{y_1>K}y_2F_{y_2}\ep\varphi_A \, dy_1dy_2.
$$
From \eqref{F1}-\eqref{F2} we have, for every $A> K\geq 1$, that
\begin{align}\label{y_2F_y_21}
\nonumber
\int_{\R}\int_{y_1<K}y_2F_{y_2}\ep\varphi_Ady_1dy_2\leq &c\left(\int_{\R}\int_{y_1<K} |y_2F_{y_2}|^2dy_1dy_2\right)^{1/2}\|\ep(s)\|_2\\
\nonumber
\leq & c K^{1/2}\|\ep(s)\|_2\\
\leq & c\sqrt{C_7}\left(\alpha_0+\frac{1}{n}\right)^{1/2}\left(\frac{b}{n}\right)^{1/2},
\end{align}
where in the last line we used \eqref{ep(s)2}.\\
To bound the second part we use Lemma \ref{Point-Decay} with $n$ large such that
$$
\delta=\sqrt{C_7}\left(\alpha_0+\frac{1}{n}\right)^{1/2}\left(\frac{b}{n}\right)^{1/2}<\delta_0.
$$
Indeed, since $\ep_0$ given by \eqref{ep_0} has an exponential decay, the relation \eqref{Idecay} is satisfied for any $\sigma>\frac{21}8$. Therefore, $\sigma^{\ast}=-\frac23\sigma+\frac34<-1$ and for every $s\geq 1$ and $A> K\geq 1$, that
\begin{align*}
\int_{\R}\int_{y_1>K}y_2F_{y_2}\ep\varphi_Ady_1dy_2\leq & \int_{\R}\sup_{y_1}|y_2F_{y_2}|\left(\int_{y_1>K}|\ep|dy_1\right)dy_2\\
\leq &c(\sqrt{C_7}+1)\left(\frac{1}{n}+\left(\alpha_0+\frac{1}{n}\right)^{1/2}\left(\frac{b}{n}\right)^{1/2}\right), \end{align*}
where we also used \eqref{F1} in the last line.
Now, there exists a constant $C_8>0$ such that
\begin{equation}\label{y_2F_y_22}
\int_{\R}\int_{y_1>K}y_2F_{y_2}\ep\varphi_Ady_1dy_2\leq C_8\left(\frac{1}{n}+\left(\alpha_0+\frac{1}{n}\right)^{1/2}\left(\frac{b}{n}\right)^{1/2}\right).
\end{equation}
Collecting \eqref{y_2F_y_21} and \eqref{y_2F_y_22}, for every $s\geq 1$ and $A\geq 1$, we have
$$
\left|\int_{\R^2}y_2F_{y_2}\ep\varphi_A\right|\leq C_9\left(\frac{1}{b^{1/2}n^{1/2}}+\left(\alpha_0+\frac{1}{n}\right)^{1/2}\right)\left(\frac{b}{n}\right)^{1/2}
$$
for some constant $C_9>0$.

Now, we choose $\alpha_0>0$ sufficiently small and $n_0\in \mathbb{N}$ sufficiently large such that
$$
\sqrt{C_7}C_6\left(\alpha_0+\frac{1}{n_0}\right)^{1/2}\max\left\{\sqrt{C_7}\left(\alpha_0+\frac{1}{n_0}\right)^{1/2}, 1, C_9\left(\frac{1}{b^{1/2}n_0^{1/2}}+\left(\alpha_0+\frac{1}{n_0}\right)^{1/2}\right)\right\}<1/6.
$$
For fixed $\alpha_0$ and $n_0$, satisfying the previous inequality, we choose $A\geq 1$ such that
$$
A^{-1/2}+A^{1/2}\|\ep(s)\|_{L^2(y_1\geq A)}\leq \left(\frac{b}{n_0}\right)^{1/2},
$$
which is possible due to Corollary \ref{AM3}.

Therefore, we finally deduce that
$$
\lambda \widetilde{R}(\ep, A)\leq \frac{b}{2n_0},
$$
which implies from \eqref{K'AM_0} and \eqref{M_0term} that
$$
\frac{d}{ds}K_A(s)\geq \frac{b}{2n_0}>0, \quad \textrm{for all} \quad s\geq 1.
$$
\end{proof}

Now we have all the ingredients to finish the proof of our main result.

{\bf Last Step in the proof of Theorem \ref{Theo-Inst}.}

Integrating in $s$ variable both sides of inequality \eqref{Bound-K'_A}, we get
$$
K_A(s)\geq s\left(\frac{b}{2n_0}\right) + K_A(0), \quad \textrm{for all} \quad s\geq 1.
$$
Therefore,
$$
\lim_{s\rightarrow \infty}K_A(s)=\infty,
$$
which is a contradiction to \eqref{Bound-KA}. Hence, our original assumption that $Q$ is stable is not valid and we conclude the proof of the theorem.
\end{proof}

\bibliographystyle{amsplain}

\end{document}